\documentclass[10pt]{article}
\usepackage{amsmath}
\usepackage{amsfonts}
\usepackage{mathrsfs}
\usepackage{amssymb}
\usepackage{amsthm}
\usepackage{palatino}
\usepackage[margin=2.5cm, vmargin={1.5cm},includefoot]{geometry}

\usepackage{pst-plot}
\usepackage{pst-grad}

\newrgbcolor{LemonChiffon}{1.0 0.98 0.8}
\newrgbcolor{magenta}{1.0 0.3 0.4}
\newrgbcolor{mistyrose}{1.0 0.0 0.6}
\newrgbcolor{deeppink}{1.0 0.08 0.58}
\newrgbcolor{lightsalmon}{1.0 0.63 0.48}
\newrgbcolor{lightred}{0.7 0.0 0.0}
\newrgbcolor{lightblue}{0.68 0.85 0.90}
\newrgbcolor{indianred}{0.80  0.36  0.36}
\newrgbcolor{lightgreen}{0.56 0.93 0.56}
\newrgbcolor{byellow}{0.93 0.76 0.3}

\DeclareMathOperator{\cl}{Cl_2}

\begin{document}

\title{Asymptotics for the partial fractions of the restricted partition generating function I}

\author{Cormac O'Sullivan\footnote{
\newline
Support for this project was provided by a PSC-CUNY Award, jointly funded by The Professional Staff Congress and The City University of New York.}}

\maketitle

\begin{abstract}
The generating function for $p_N(n)$, the number of  partitions of $n$ into at most $N$ parts, may be written as a product of $N$ factors. We find the behavior of coefficients in the partial fraction decomposition of this product as $N \to \infty$ by applying the saddle-point method, where the  saddle-point we need is associated to a zero of the analytically continued dilogarithm. Our main result disproves a conjecture of Rademacher.
\end{abstract}

\def\s#1#2{\langle \,#1 , #2 \,\rangle}

\def\H{{\mathbf{H}}}
\def\F{{\frak F}}
\def\C{{\mathbb C}}
\def\R{{\mathbb R}}
\def\Z{{\mathbb Z}}
\def\Q{{\mathbb Q}}
\def\N{{\mathbb N}}
\def\G{{\Gamma}}
\def\GH{{\G \backslash \H}}
\def\g{{\gamma}}
\def\L{{\Lambda}}
\def\ee{{\varepsilon}}
\def\K{{\mathcal K}}
\def\Re{\mathrm{Re}}
\def\Im{\mathrm{Im}}
\def\PSL{\mathrm{PSL}}
\def\SL{\mathrm{SL}}
\def\Vol{\operatorname{Vol}}
\def\lqs{\leqslant}
\def\gqs{\geqslant}
\def\sgn{\operatorname{sgn}}
\def\res{\operatornamewithlimits{Res}}
\def\li{\operatorname{Li_2}}
\def\lis{\operatorname{Li_2^*}}
\def\clp{\operatorname{Cl}'_2}
\def\clpp{\operatorname{Cl}''_2}
\def\farey{\mathscr F}

\newcommand{\stira}[2]{{\left[ #1 \atop #2  \right]}}
\newcommand{\stirb}[2]{{\left\{ #1 \atop #2  \right\}}}
\newcommand{\norm}[1]{\left\lVert #1 \right\rVert}

\newcommand{\spr}[2]{\sideset{}{_{#2}^{-1}}{\textstyle \prod}({#1})}
\newcommand{\spn}[2]{\sideset{}{_{#2}}{\textstyle \prod}({#1})}

\newtheorem{theorem}{Theorem}[section]
\newtheorem{lemma}[theorem]{Lemma}
\newtheorem{prop}[theorem]{Proposition}
\newtheorem{conj}[theorem]{Conjecture}
\newtheorem{cor}[theorem]{Corollary}
\newtheorem{remark}[theorem]{Remark}
\newtheorem*{maina}{Theorem 1.6}

\newcounter{coundef}
\newtheorem{adef}[coundef]{Definition}

\renewcommand{\labelenumi}{(\roman{enumi})}

\numberwithin{equation}{section}

\bibliographystyle{alpha}

\tableofcontents


\section{Introduction} \label{sec-int}
\subsection{Rademacher's coefficients}
Let $p(n)$ denote the number of partitions of $n$. The generating function for  $p(n)$ is an infinite product and Rademacher,
in \cite[pp. 292 - 302]{Ra},  obtained a  decomposition for it
\begin{equation}\label{37}
\sum_{n=0}^\infty p(n) q^n =
\prod_{j=1}^\infty \frac{1}{1-q^j}=\sum_{\substack{0\leqslant h<k  \\ (h,k)=1}}
\sum_{\ell=1}^{\infty} \frac{C_{hk\ell}(\infty)}{(q-e^{2\pi ih/k})^\ell} \qquad(|q|<1)
\end{equation}
by using his famous exact formula for $p(n)$. The coefficients $C_{hk\ell}(\infty)$ are given explicitly in \cite[Eq. (130.6)]{Ra}
with, for example,
$$
C_{011}(\infty) = -\frac 6{25} -\frac{12 \sqrt{3}}{125 \pi}, \qquad C_{012}(\infty) = \frac{144}{1225}+ \frac{5616 \sqrt{3}}{42875 \pi}.
$$
In this notation,
$C_{011}$ is the coefficient of $(q-1)^{-1}$ and $C_{012}$  the coefficient of $(q-1)^{-2}$.
Truncating the infinite product in \eqref{37} gives the generating function for  $p_N(n)$, the number of  partitions of $n$ into at most $N$ parts, and its partial fraction decomposition may be written as
\begin{equation}\label{tp}
\sum_{n=0}^\infty p_N(n) q^n =
\prod_{j=1}^N \frac{1}{1-q^j}=\sum_{\substack{0\leqslant h<k \leqslant N \\ (h,k)=1}}
\sum_{\ell=1}^{\lfloor N/k \rfloor} \frac{C_{hk\ell}(N)}{(q-e^{2\pi ih/k})^\ell}.
\end{equation}
Comparing \eqref{37} and \eqref{tp}, Rademacher conjectured in \cite[p. 302]{Ra} that
\begin{equation}\label{rademach}
    C_{hk\ell}(N) \to C_{hk\ell}(\infty) \qquad \text{as} \qquad N \to \infty.
\end{equation}
Investigations in \cite{An}, \cite{DG}, \cite{Mu} were inconclusive, but Sills and Zeilberger in \cite{SZ} developed recursive formulas for $C_{hk\ell}(N)$ and gave convincing numerical evidence that
$C_{hk\ell}(N) \not\to C_{hk\ell}(\infty)$. They  saw  that  the points $(N,  C_{01\ell}(N))$  start to trace  curves oscillating with periods approaching  32 and with amplitude growing exponentially. Their conjecture  \cite[Conj. 2.1]{SZ}  is that this is the true description.

In \cite{OS} we found relatively simple, explicit formulas for Rademacher's coefficients $C_{hk\ell}(N)$, linking them to formulas of Sylvester \cite{Sy3} and Glaisher \cite{Gl}. For example \cite[Eq. (2.12)]{OS} is
\begin{equation*}
    C_{01\ell}(N)  = \frac{(-1)^N (\ell-1)!}{ N!} \sum_{j_0+j_1+j_2+ \cdots + j_N = N-\ell}
     \stirb{\ell+j_0}{\ell}\frac{B_{j_1}B_{j_2}  \cdots  B_{j_N}}{(\ell-1+j_0)!}\frac{ 1^{j_1} 2^{j_2} \cdots
N^{j_N}}{j_1 ! j_2 ! \cdots j_N!}
\end{equation*}
where $B_j$ is the $j$th Bernoulli number  and $\stirb{n}{m}$ the Stirling number  denoting the number of ways to partition a set of size $n$ into $m$  non-empty subsets. Also in \cite{OS}, based on an earlier stage of the work in this paper, the exact asymptotic behavior of $C_{011}(N)$ was conjectured. This requires the solution $w_0 \approx 0.916 - 0.182 i$ to
\begin{equation} \label{dilogw0}
\li(w)-2\pi i \log(w) = 0
\end{equation}
where $\li$ denotes the dilogarithm. (It may be seen that $w_0$ is a zero of the dilogarithm on a non-principal branch, see Section \ref{dilogg}.) Set $z_0:= 1+\log(1-w_0)/(2\pi i)$ so that
\begin{equation} \label{w0x}
    w_0=1-e^{2\pi i z_0}, \quad 1/2 < \Re(z_0) < 3/2.
\end{equation}

\begin{conj}\cite[Sect.  6]{OS} \label{coonj1}
We have\footnote{This statement is equivalent to Conjecture 6.2 in \cite{OS} where $z_0$ and $w_0$ are replaced by their conjugates.}
\begin{equation}\label{conjj1}
   C_{011}(N)=\Re\left[(-2  z_0 e^{-\pi i z_0})\frac{w_0^{-N}}{N^2}\right] +O\left( \frac{|w_0|^{-N}}{N^3}\right).
\end{equation}
\end{conj}
Equivalently, we may present \eqref{conjj1} more explicitly as
\begin{equation} \label{cj1real}
    C_{011}(N)= \frac{ e^{U N}}{N^2}\left(\alpha \sin(\beta +V N) +O\left( \frac 1{N}\right)\right)
\end{equation}
for
\begin{equation} \label{abuv}
       U :=-\log |w_0| \approx 0.0680762, \quad V :=\arg(1/w_0) \approx 0.196576
\end{equation}
and  $\alpha :=|-2 i  z_0 e^{-\pi i z_0}| \approx 5.39532$, $\beta := \arg(-2 i  z_0 e^{-\pi i z_0}) \approx 1.21367$. This implies the period of $C_{011}(N)$ is $2\pi/V \approx 31.9631$.
As we will see, the numbers $w_0$ and $z_0$  control the asymptotics for all of the Rademacher coefficients  that we examine.

\subsection{Main results}
Write the Farey fractions of order $N$ in $[0,1)$ as
\begin{equation}\label{farey}
    \farey_N:=\Bigl\{ h/k \ : \ 1 \lqs k \lqs N, \ 0\lqs h < k,  \ (h,k)=1\Bigr\}.
\end{equation}
Our first result is a kind of averaged version of Conjecture \ref{coonj1}, with $C_{011}(N)$ replaced by
\begin{equation*}
    C_{011}(N)+C_{121}(N)+ \cdots + C_{(99)(100)1}(N) = \sum_{h/k \in \farey_{100}} C_{hk1}(N).
\end{equation*}
\begin{theorem}\label{mainthma} For an absolute implied constant
\begin{equation} \label{maineq1}
    \sum_{h/k \in \farey_{100}} C_{hk1}(N) = \Re\left[(-2  z_0 e^{-\pi i z_0})\frac{w_0^{-N}}{N^2}\right] +O\left( \frac{|w_0|^{-N}}{N^3}\right).
\end{equation}
\end{theorem}

This has the following consequence.

\begin{cor} \label{resu}
 There exists a pair $(h,k)$ with $h < k \lqs 100$ such that  $\lim_{N\to \infty} C_{hk1}(N)$ does not exist. Hence Rademacher's conjecture that $C_{hk\ell}(N) \to C_{hk\ell}(\infty)$ as $N \to \infty$ is false.
\end{cor}
\begin{proof}
Expressing the right side of \eqref{maineq1} as in \eqref{cj1real}, we see that this side cannot have a limit as $N \to \infty$ since $\alpha \neq 0$, $U>0$ and $\beta +V N$ comes within $1/10$ of $\pi/2$, say, infinitely often since $V<1/5$. But the left side of \eqref{maineq1} is a finite sum, so Rademacher's conjecture implies that its limit as $N \to \infty$ exists. The corollary follows.
\end{proof}

Theorem \ref{mainthma}  is the $\ell=m=1$ case of the next  result where we extend the right side of \eqref{maineq1} to include the first $m$ terms of the asymptotic expansion and generalize $C_{011}(N)$ to $C_{01\ell}(N)$.

\begin{theorem}\label{mainthmb} There are explicit coefficients $c_{\ell,0},$ $c_{\ell,1}, \dots $ so that
\begin{multline} \label{maineq2}
   C_{01\ell}(N)+ \sum_{0<h/k \in \farey_{100}} \sum_{j=1}^\ell (e^{2\pi i h/k} -1)^{\ell-j} C_{hkj}(N) \\
   = \Re\left[\frac{w_0^{-N}}{N^{\ell+1}} \left( c_{\ell,0}+\frac{c_{\ell,1}}{N}+ \dots +\frac{c_{\ell,m-1}}{N^{m-1}}\right)\right] + O\left(\frac{|w_0|^{-N}}{N^{\ell+m+1}}\right)
\end{multline}
where $c_{\ell,0}=-2  z_0 e^{-\pi i z_0} (2\pi i z_0)^{\ell -1}$ and the implied constant depends only on  $\ell$ and $m$.
\end{theorem}

See \eqref{clff} for the next coefficient $c_{\ell,1}$. The reason we need to include the sum over $0<h/k \in \farey_{100}$ on the left of \eqref{maineq2} is given in Remark \ref{wha}. Numerically, this sum looks to be  much smaller than $C_{01\ell}(N)$, so it is natural to generalize Conjecture \ref{coonj1} to:
\begin{conj} \label{coj}
For the coefficients $c_{\ell,0},$ $c_{\ell,1}, \dots $ of Theorem \ref{mainthmb},
\begin{equation} \label{cj5}
    C_{01\ell}(N)
   = \Re\left[\frac{w_0^{-N}}{N^{\ell+1}} \left( c_{\ell,0}+\frac{c_{\ell,1}}{N}+ \dots +\frac{c_{\ell,m-1}}{N^{m-1}}\right)\right] + O\left(\frac{|w_0|^{-N}}{N^{\ell+m+1}}\right).
\end{equation}
\end{conj}
Numerical evidence for Conjecture \ref{coj} is given in Section \ref{sec-fur} and the asymptotics of the next cases, $C_{121}(N)$ and $C_{131}(N)$, are also discussed there. The following subsection  outlines how the proofs of Theorems \ref{mainthma} and \ref{mainthmb} are constructed.

As this work was being completed, I was contacted by Drmota and Gerhold who provided me with their paper \cite{DrGe}. They have given an independent disproof of Rademacher's conjecture by combining a Mellin transform approach with the saddle-point method  to obtain the asymptotics of  $C_{01\ell}(N)$. Their main result is equivalent to Conjecture \ref{coj} in the case $m=1$ though with a weaker error term. Combining their techniques with ours should lead to improved asymptotics and a better understanding of all the Rademacher coefficients.


\subsection{Method of proof}
We have from \cite[Eq. (2.1)]{OS} that
\begin{equation} \label{cuvrn1}
  C_{hk\ell }(N) =  2\pi i \res_{z=h/k} \frac{e^{2\pi i z}(e^{2\pi i z}-e^{2\pi i
h/k})^{\ell -1}}{(1-e^{2\pi i 1 z})(1-e^{2\pi i2 z}) \cdots (1-e^{2\pi i N z})}.
\end{equation}
The  right of \eqref{cuvrn1} may be expressed in terms of the simpler function
\begin{equation}\label{qzns}
Q(z;N,\sigma):=\frac{  e^{2\pi i \sigma z}}{(1-e^{2\pi i 1 z})(1-e^{2\pi i2 z}) \cdots (1-e^{2\pi i N z})}
\end{equation}
and we write
\begin{equation} \label{defqn}
    Q_{hk\sigma}(N):=2\pi i \res_{z=h/k} Q(z;N,\sigma).
\end{equation}
Expanding the numerator on the right of \eqref{cuvrn1} then produces
\begin{equation} \label{exprc}
    C_{hk\ell }(N) =   \sum_{\sigma=1}^\ell \binom{\ell -1}{\sigma-1} (-e^{2\pi i h/k})^{\ell-\sigma} Q_{hk\sigma}(N).
\end{equation}
The numbers $Q_{hk\sigma }(N)$ are slightly easier to work with than $C_{hk\ell }(N)$, though of course  for $\ell=1$ we have $C_{hk1}(N) = Q_{hk1}(N)$. Each $Q_{hk\sigma }(N)$ is a component of the Sylvester wave $W_k$, as described in Section \ref{sec-gen}, and expressions such as $Q(z;N,\sigma)$ and $Q_{hk\sigma}(N)$ appear when counting lattice points in a polytope dilated by a factor $-\sigma>0$, see \cite[Thm. 1]{Be2} and \cite{Be1}.
We may also invert \eqref{exprc} to get
\begin{equation} \label{exprcinv}
    Q_{hk\sigma }(N) =   \sum_{\ell=1}^\sigma \binom{\sigma-1}{\ell -1} (e^{2\pi i h/k})^{\sigma-\ell} C_{hk\ell}(N).
\end{equation}

As a function of $z$, $Q(z;N,\sigma)$ is  meromorphic  of period $1$  when $\sigma \in \Z$. Fixing a positive integer $\sigma$ and summing all the residues then leads to the
key identity on which Theorem \ref{mainthmb} is based:
\begin{equation}
\sum_{h/k \in \farey_N} Q_{hk\sigma}(N) =  0  \qquad \text{for} \qquad N(N+1)/2 > \sigma. \label{ch}
\end{equation}
There is a large contribution to the left of \eqref{ch} from $Q_{01\sigma}(N)$ as well as  other $Q_{hk\sigma}(N)$ with $k$ small, corresponding to high-order poles of $Q(z;N,\sigma)$. Balancing that are contributions from coefficients $Q_{hk\sigma}(N)$ with $k$ large, corresponding to simple poles.
Put
\begin{equation} \label{a(n)}
    \mathcal A(N)  := \Bigl\{ h/k \ : \ \frac{N}{2}  <k \lqs N,  \ h=1 \text{ \ or \ } h=k-1 \Bigr\} \subseteq \farey_N
\end{equation}
and
decompose \eqref{ch} into
\begin{equation*}
    \sum_{h/k \in \farey_{100}} Q_{hk\sigma}(N) + \sum_{h/k \in \farey_N- (\farey_{100} \cup \mathcal A(N))} Q_{hk\sigma}(N) + \sum_{h/k \in \mathcal A(N)} Q_{hk\sigma}(N)=  0.
\end{equation*}
The reason we focus on the subset $\mathcal A(N)$ is given in the next section, but it may already be noticed  that, numerically,
\begin{equation} \label{already}
    C_{011}(N) \approx -\mathcal A_1(N,1) \quad \text{ as } \quad N\to \infty
\end{equation}
for
\begin{equation} \label{sum}
    \mathcal A_1(N,\sigma):=\sum_{h/k \in \mathcal A(N)} Q_{hk\sigma}(N).
\end{equation}

 Computing the residues of the simple poles  lets us describe \eqref{sum} more explicitly as
\begin{equation}\label{cj1}
\mathcal A_1(N,\sigma) = \Im \sum_{ \frac{N}{2}  <k \lqs N} \frac{2(-1)^{k}}{k^2}\exp\left( \frac{i\pi}{2}\left[ \frac{-N^2-N+4\sigma}{k}+3N \right]\right) \spr{1/k}{N-k}
\end{equation}
for $\sigma \in \Z$, where we write
\begin{equation} \label{sidef}
    \spn{\theta}{m} :=\prod_{j=1}^m 2\sin (\pi j \theta)
\end{equation}
with $\spn{\theta}{0}:=1$, following  Sudler's notation in \cite{Su} except that we don't take the  absolute value.
The main part of the proof of  Theorem \ref{mainthmb} then consists of establishing the following two results. Recall $w_0$ and $z_0$ from \eqref{w0x}.

\begin{theorem}\label{maina} With $b_{0}=2  z_0 e^{-\pi i z_0}$ and  explicit  $b_{1}(\sigma),$ $b_{2}(\sigma), \dots $ depending on $\sigma \in \Z$ we have
\begin{equation*}
   \mathcal A_1(N,\sigma) = \Re\left[\frac{w_0^{-N}}{N^{2}} \left( b_{0}+\frac{b_{1}(\sigma)}{N}+ \dots +\frac{b_{m-1}(\sigma)}{N^{m-1}}\right)\right] + O\left(\frac{|w_0|^{-N}}{N^{m+2}}\right)
\end{equation*}
for an implied constant depending only on  $\sigma$ and $m$.
\end{theorem}
\begin{theorem}\label{mainb} There exists  $W<U:=-\log |w_0| \approx 0.068076$ so that
\begin{equation*}
   \sum_{h/k \in \farey_N-(\farey_{100} \cup \mathcal A(N))} Q_{hk\sigma}(N) = O\left(e^{WN}\right)
\end{equation*}
for an implied constant depending only on  $\sigma$. We may take $W=0.055$.
\end{theorem}

The proof of Theorem \ref{maina} is carried out as follows. In Section \ref{sec-pre} we derive the  sum for $\mathcal A_1(N,\sigma)$ in \eqref{cj1} and also give some results on the dilogarithm we will need. Section \ref{sec-est} is quite technical and includes  estimates of the sine product $\spn{h/k}{N-k}$ using Euler-Maclaurin summation, where the number of terms required is proportional to $k$ and $N$. The sum $\mathcal A_1(N,\sigma)$ is replaced by an integral in Section \ref{sec-a1n}, and in Section \ref{sec-sad} the saddle-point method is introduced and applied. The required  saddle-point is $z_0$ and with work of Wojdylo \cite{Woj} we explicitly get the full asymptotic expansion of $\mathcal A_1(N,\sigma)$. This proves  Theorem \ref{maina}.

See Section \ref{sec-fur} for a summary of the proof of Theorem \ref{mainb}.
The bounds required for $Q_{hk\sigma}(N)$ in this proof  can be obtained directly in most cases, but three families also require  saddle-point arguments, with these saddle-points corresponding to further zeros of the dilogarithm on other branches.
 The details  are carried out in the companion paper \cite{OS2}.

Linear combinations of Theorems \ref{maina} and  \ref{mainb} then give Theorem \ref{mainthmb} in Section \ref{prfs}.
In Section \ref{sec-fur} we also discuss extensions and generalizations of our results and applications to the restricted partition function and Sylvester waves.

\section{Preliminary  material} \label{sec-pre}
\subsection{The residues of $Q(z;N,\sigma)$}
For $Q(z;N,\sigma)$ defined in \eqref{qzns} with $\sigma \in \C$, we clearly have
\begin{align}
\overline{Q(\overline{z};N,\sigma)} & = Q(-z;N,\overline{\sigma}), \label{q1b}\\
Q(-z;N,\sigma) & = (-1)^N Q(z;N,N(N+1)/2-\sigma) \label{q2}
\end{align}
and, if $\sigma \in \Z$,
\begin{equation}\label{q1}
    Q(z+1;N,\sigma)  = Q(z;N,\sigma).
\end{equation}

As a function of $z$, $Q(z;N,\sigma)$ is meromorphic with all poles contained in $\Q$. More precisely, the set of poles of $Q(z;N,\sigma)$ in $[0,1)$ equals $\farey_N$, the Farey fractions of order $N$ in $[0,1)$.

\begin{theorem} \label{qn} For $N \in \Z_{\gqs 1}$ and $\sigma \in \Z$
\begin{equation} \label{th1}
2\pi i \sum_{h/k \in \farey_N} \res_{z=h/k} Q(z;N,\sigma) =  \begin{cases}
-p_N(-\sigma) & \text{ \ if \ } \quad \sigma \lqs 0 \\
0 & \text{ \ if \ } \quad 0<\sigma<N(N+1)/2 \\
(-1)^N p_N \bigl(\sigma-N(N+1)/2 \bigr)  & \text{ \ if \ } \quad N(N+1)/2 \lqs \sigma.
\end{cases}
\end{equation}
\end{theorem}
\begin{proof}
We have
\begin{equation}\label{pn}
    \sum_{n=0}^\infty p_N(n) e^{2\pi i n z} = \prod_{j=1}^N \frac{1}{1-e^{2\pi i j z}}
\end{equation}
and since $p_N(n) \lqs p(n) \lqs 2^{n-1}$, the number of ordered partitions of $n$, we see the left side of \eqref{pn} is absolutely convergent for $\Im (z)$ large enough. (Better bounds for $p_N(n)$, $p(n)$ imply absolute convergence for $\Im (z) >0$. See for example \cite{Pr}, employing the dilogarithm.) Hence, for $\Im (w)$ large enough,
\begin{equation}\label{wpn}
    \int_w^{w+1} Q(z;N,\sigma) \, dz = \begin{cases}
0 & \text{ \ if \ } \quad \sigma>0 \\
p_N(-\sigma) & \text{ \ if \ } \quad \sigma \lqs 0.
\end{cases}
\end{equation}
Let $\mathcal R$ be the rectangle in $\C$ with upper corners $w$, $w+1$ and lower corners $v$, $v+1$ where $\Im (v) <0$. Integrating around $\mathcal R$ in a positive direction and choosing $\Re(w)=\Re(v)$ between $0$ and the next pole to the left,
\begin{equation}\label{rq}
    \int_\mathcal R Q(z;N,\sigma) \, dz = 2\pi i \sum_{h/k \in \farey_N} \res_{z=h/k} Q(z;N,\sigma).
\end{equation}
The integral along the top of $\mathcal R$ is $-1$ times \eqref{wpn}. The integral along the bottom can be made arbitrarily small by letting $\Im(v) \to -\infty$ provided $\sigma < N(N+1)/2$ and the integrals along the vertical sides cancel with \eqref{q1}. If $\sigma \gqs N(N+1)/2$ then use \eqref{q2}. This completes the proof.
\end{proof}

Theorem \ref{qn} for negative integer $\sigma$ is a restatement of a special case of Sylvester's Theorem. See for example \cite[Sect.  4]{OS}.

Each $h/k \in \farey_N$ is a pole of $Q(z;N,\sigma)$ of order $s=\lfloor N/k \rfloor$. Equivalently, $h/k$ is a pole of order $s$ exactly when
\begin{equation}\label{ord}
\frac{N}{s+1} <k \lqs  \frac{N}{s}.
\end{equation}
Thus $2\pi i Q(z;N,\sigma)$ has one pole of order $N$ in $[0,1)$ at $h/k=0/1$ with residue $Q_{01\sigma}(N)$. The next highest order pole has order $\lfloor N/2 \rfloor$ at $h/k=1/2$ with residue $Q_{12\sigma}(N)$.
By \eqref{ord}, $h/k$ is a simple pole when $N/2<k \lqs N$ and
the residues of the simple poles of $Q(z;N,\sigma)$ may be computed quite easily.
\begin{prop} \label{simple}
For  $N/2<k \lqs N$
\begin{multline*}
    Q_{hk\sigma}(N)  = \frac{(-1)^{k+1}}{k^2}  \exp\left(\frac{-\pi i h(N^2+N-4\sigma)}{2k}\right) \\
    \times \exp\left(\frac{\pi i}{2}(2Nh+N+h+k-hk)\right) \prod_{j=1}^{N-k} \frac{1}{2 \sin(\pi j h/k)}.
\end{multline*}
\end{prop}
\begin{proof}
With \eqref{cuvrn1}, write
$$
Q_{hk\sigma}(N)  = \res_{z=h/k} \frac{2\pi i  e^{2\pi i \sigma z}}{\left[(1-e^{2\pi i z}) \cdots (1-e^{2\pi i (k-1) z}) \right](1-e^{2\pi i k z})\left[(1-e^{2\pi i (k+1)z}) \cdots (1-e^{2\pi i N z}) \right]}.
$$
Then
$$
\res_{z=h/k} \frac{1}{1-e^{2\pi i k z}} = \frac{-1}{2\pi i k}.
$$
Also
\begin{equation}\label{1-zk}
(1-\zeta)(1-\zeta^2) \cdots (1-\zeta^{k-1})= k
\end{equation}
for $\zeta=e^{2\pi i h/k}$  a primitive $k$th root of unity, by \cite[Lemma 4.4]{OS} for example.
Hence
\begin{equation}\label{cuv1}
Q_{hk\sigma}(N) =\frac{-e^{2\pi i \sigma h/k}}{k^2}\prod_{j=k+1}^N \frac{1}{1-e^{2\pi i j h/k}} =\frac{-e^{2\pi i \sigma h/k}}{k^2}\prod_{j=1}^{N-k} \frac{1}{1-e^{2\pi i j h/k}}.
\end{equation}
A straightforward calculation, with $0\lqs m <k$, shows
\begin{equation}\label{wla}
\prod_{j=1}^{m} \frac{1}{1-e^{2\pi i j h/k}} = \exp\left(\frac{\pi i m}2\left(1-\frac hk(m+1)\right)\right)\prod_{j=1}^{m} \frac{1}{2\sin(\pi  j h/k)}
\end{equation}
and combining this with \eqref{cuv1} and simplifying completes the proof.
\end{proof}

\subsection{Products of sines}

Recall our notation \eqref{sidef}.
Then for integers $k>h \gqs 1$ with $(h,k)=1$
\begin{equation}\label{sipr}
\spr{h/k}{m} =\prod_{j=1}^m \frac{1}{2\sin (\pi j h/k)} \qquad (0 \lqs m <k)
\end{equation}
is a real number. For example,
\begin{equation}\label{simp1}
    \spr{h/k}{k-1} = (-1)^{(h-1)(k-1)/2}\frac 1k
\end{equation}
 follows from setting $m=k-1$ in \eqref{wla} and using \eqref{1-zk}.

With Proposition \ref{simple} we see that
\begin{equation*}
    \left| Q_{hk\sigma}(N) \right| = \left| \spr{h/k}{N-k} \right|/k^2 \qquad (N/2<k \lqs N, \ \sigma \in \R)
\end{equation*}
and so the size of $Q_{hk\sigma}(N)$ is controlled by the sine product $\spr{h/k}{m}$ for $m=N-k$. As $m$ varies we need to know how big $\spr{h/k}{m}$ can be. For example, if $h=1$ and $k$ is large then the first terms
\begin{equation*}
    \frac{1}{2\sin (\pi 1/k)} \approx \frac k{2\pi} , \quad \frac{1}{2\sin (\pi 2/k)}\approx \frac k{4\pi}, \quad \dots
\end{equation*}
are all greater than $1$. The maximum is reached with
\begin{equation} \label{mx}
    \frac{1}{2\sin (\pi 1/k)} \times \frac{1}{2\sin (\pi 2/k)} \times \dots \times \frac{1}{2\sin (\pi (k/6)/k)} = \spr{1/k}{k/6}
\end{equation}
since after that the factors become less than $1$. If $h=2$, the maximum value of $\spr{2/k}{m}$ is reached for $m=k/12$ and this value is approximately the square root of \eqref{mx}. For other values of $h$ the maximum of the product does not become as large because values greater than $1$ are multiplied by more values less than $1$. An exception is when $h=(k-1)/2$ since here again large products can build up. We illustrate this with Figure \ref{afig} which graphs
\begin{equation}\label{phihk}
\Psi(h/k):=\max_{0\lqs m <k} \left\{ \frac{1}{k} \Bigl| \log \bigl| \spn{h/k}{m} \bigr| \Bigr|  \right\}.
\end{equation}
for $1\lqs h < k$ with $k$ prime and equalling $211$. We see that the  largest values of $\Psi(h/k)$ are for  $h \in \{1,2,(k-1)/2\}$ (and symmetrically $h \in \{k-1,k-2,(k+1)/2\}$) with exactly these values greater than $U\approx 0.068076$. These observations are made precise  in Section \ref{sec-fur1}.



\SpecialCoor
\psset{griddots=5,subgriddiv=0,gridlabels=0pt}
\psset{xunit=0.04cm, yunit=24cm}
\psset{linewidth=1pt}
\psset{dotsize=2pt 0,dotstyle=*}

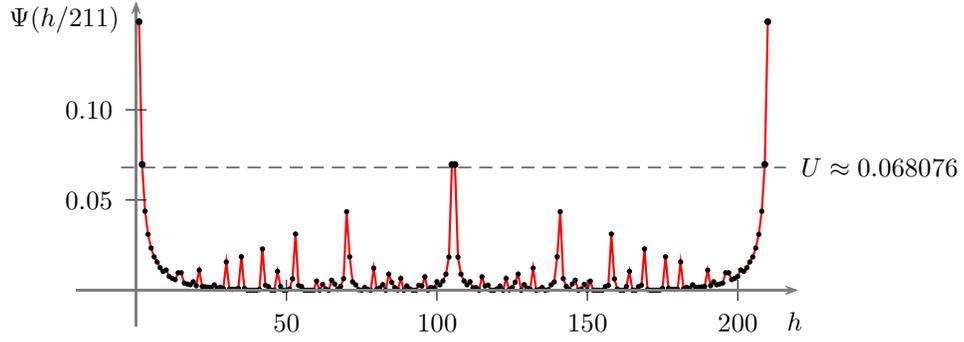
\begin{figure}[h]
\begin{center}
\begin{pspicture}(-30,-0.02)(230,0.16) 

\savedata{\mydata}[
{{1, 0.148849}, {2, 0.0697363}, {3, 0.043772}, {4, 0.0309863}, {5,
  0.023397}, {6, 0.0184698}, {7, 0.0156639}, {8, 0.0124713}, {9,
  0.0104293}, {10, 0.0112525}, {11, 0.00767528}, {12,
  0.00659574}, {13, 0.0058591}, {14, 0.0096738}, {15,
  0.00971531}, {16, 0.00396796}, {17, 0.00342261}, {18,
  0.00301237}, {19, 0.00470466}, {20, 0.00252641}, {21, 0.01113}, {22,
   0.0020895}, {23, 0.00188677}, {24, 0.00169337}, {25,
  0.00150855}, {26, 0.00300304}, {27, 0.00116221}, {28,
  0.00163458}, {29, 0.000843381}, {30, 0.0156784}, {31,
  0.000548606}, {32, 0.000473552}, {33, 0.000554153}, {34,
  0.00101775}, {35, 0.0186021}, {36, 0.00106814}, {37, 0.}, {38,
  0.}, {39, 0.}, {40, 0.}, {41, 0.000381978}, {42, 0.0229158}, {43,
  0.0027325}, {44, 0.00196647}, {45, 0.}, {46, 0.}, {47,
  0.0104684}, {48, 0.00208189}, {49, 0.}, {50, 0.}, {51,
  0.00066038}, {52, 0.00632593}, {53, 0.0311822}, {54,
  0.00265721}, {55, 0.00204909}, {56, 0.}, {57, 0.}, {58, 0.}, {59,
  0.}, {60, 0.00504905}, {61, 0.}, {62, 0.0032529}, {63,
  0.00070455}, {64, 0.}, {65, 0.00556107}, {66, 0.00345539}, {67,
  0.000454217}, {68, 0.00215718}, {69, 0.00643135}, {70,
  0.0435319}, {71, 0.01853}, {72, 0.00460294}, {73, 0.0030058}, {74,
  0.000302041}, {75, 0.}, {76, 0.00152355}, {77, 0.000412421}, {78,
  0.}, {79, 0.0122548}, {80, 0.000515716}, {81, 0.0012993}, {82,
  0.00309704}, {83, 0.}, {84, 0.00892564}, {85, 0.00442469}, {86,
  0.00144133}, {87, 0.000185497}, {88, 0.00653466}, {89, 0.}, {90,
  0.00237935}, {91, 0.00073034}, {92, 0.}, {93, 0.}, {94,
  0.00278581}, {95, 0.00237614}, {96, 0.00739771}, {97,
  0.0000300997}, {98, 0.00158219}, {99, 0.00126458}, {100,
  0.00473862}, {101, 0.00298033}, {102, 0.0051156}, {103,
  0.00890871}, {104, 0.0184042}, {105, 0.0696573}, {106,
  0.0696573}, {107, 0.0184042}, {108, 0.00890871}, {109,
  0.0051156}, {110, 0.00298033}, {111, 0.00473862}, {112,
  0.00126458}, {113, 0.00158219}, {114, 0.0000300997}, {115,
  0.00739771}, {116, 0.00237614}, {117, 0.00278581}, {118, 0.}, {119,
  0.}, {120, 0.00073034}, {121, 0.00237935}, {122, 0.}, {123,
  0.00653466}, {124, 0.000185497}, {125, 0.00144133}, {126,
  0.00442469}, {127, 0.00892564}, {128, 0.}, {129, 0.00309704}, {130,
  0.0012993}, {131, 0.000515716}, {132, 0.0122548}, {133, 0.}, {134,
  0.000412421}, {135, 0.00152355}, {136, 0.}, {137,
  0.000302041}, {138, 0.0030058}, {139, 0.00460294}, {140,
  0.01853}, {141, 0.0435319}, {142, 0.00643135}, {143,
  0.00215718}, {144, 0.000454217}, {145, 0.00345539}, {146,
  0.00556107}, {147, 0.}, {148, 0.00070455}, {149, 0.0032529}, {150,
  0.}, {151, 0.00504905}, {152, 0.}, {153, 0.}, {154, 0.}, {155,
  0.}, {156, 0.00204909}, {157, 0.00265721}, {158, 0.0311822}, {159,
  0.00632593}, {160, 0.00066038}, {161, 0.}, {162, 0.}, {163,
  0.00208189}, {164, 0.0104684}, {165, 0.}, {166, 0.}, {167,
  0.00196647}, {168, 0.0027325}, {169, 0.0229158}, {170,
  0.000381978}, {171, 0.}, {172, 0.}, {173, 0.}, {174, 0.}, {175,
  0.00106814}, {176, 0.0186021}, {177, 0.00101775}, {178,
  0.000554153}, {179, 0.000473552}, {180, 0.000548606}, {181,
  0.0156784}, {182, 0.000843381}, {183, 0.00163458}, {184,
  0.00116221}, {185, 0.00300304}, {186, 0.00150855}, {187,
  0.00169337}, {188, 0.00188677}, {189, 0.0020895}, {190,
  0.01113}, {191, 0.00252641}, {192, 0.00470466}, {193,
  0.00301237}, {194, 0.00342261}, {195, 0.00396796}, {196,
  0.00971531}, {197, 0.0096738}, {198, 0.0058591}, {199,
  0.00659574}, {200, 0.00767528}, {201, 0.0112525}, {202,
  0.0104293}, {203, 0.0124713}, {204, 0.0156639}, {205,
  0.0184698}, {206, 0.023397}, {207, 0.0309863}, {208,
  0.043772}, {209, 0.0697363}, {210, 0.148849}}
  ]
\dataplot[linecolor=red,linewidth=0.8pt,plotstyle=line]{\mydata}
\dataplot[linecolor=black,linewidth=0.8pt,plotstyle=dots]{\mydata}

\psaxes[linecolor=gray,Ox=0,Oy=0,Dx=50,dx=50,Dy=0.05,dy=0.05]{->}(0,0)(-20,-0.02)(220,0.16)

\psset{linecolor=black,linewidth=1pt,dotsize=2.6pt 0,dotstyle=*}
\savedata{\mydata}[
{{1, 0.148849}, {2, 0.0697363}, {105, 0.0696573}, {106,
  0.0696573},  {209,
  0.0697363}, {210, 0.148849}}
]
\dataplot[linecolor=black,linewidth=0.8pt,plotstyle=dots]{\mydata}

\psline[linestyle=dashed,linecolor=darkgray,linewidth=0.5pt](-5,0.06807)(216,0.06807)
  \rput(247,0.068){$U \approx 0.068076$}

\rput(219,-0.017){$h$}
\rput(-24,0.15){$\Psi(h/211)$}


\end{pspicture}
\caption{$\Psi(h/k)$ for $1 \leqslant h <k$ and $k=211$}\label{afig}
\end{center}
\end{figure}


So, among the simple poles of $Q(z;N,\sigma)$, Figure \ref{afig} leads us to expect that the largest contribution to the left of \eqref{ch} should be from the sum  $\mathcal A_1(N,\sigma)$ as defined in \eqref{sum}.

With \eqref{q1b} and \eqref{q1} we obtain the identity
\begin{equation*}
    2\pi i  \res_{z=1-h/k} Q(z;N,\sigma) = \overline{2\pi i  \res_{z=h/k} Q(z;N,\sigma) }, \qquad (\sigma \in \Z).
\end{equation*}
Therefore, assuming $\sigma \in \Z$ from now on,
\begin{align*}
    \mathcal A_1(N,\sigma) & :=  2\pi i \sum_{h/k \in \mathcal A(N)} \res_{z=h/k} Q(z;N,\sigma) \\
   & \phantom{:}=  2 \Re \Bigl[ 2\pi i \sum_{ \frac{N}{2}  <k \lqs N}  \res_{z=1/k} Q(z;N,\sigma)  \Bigr].
\end{align*}
So setting $h=1$ in Proposition \ref{simple} and simplifying yields \eqref{cj1}.

\subsection{The dilogarithm} \label{dilogg}
We assemble here some of the properties of the dilogarithm we will need. See for example \cite{max}, \cite{Zag07} for more details.
Initially defined as
\begin{equation}\label{def0}
\li(z):=\sum_{n=1}^\infty \frac{z^n}{n^2} \quad \text{ for }|z|\lqs 1,
\end{equation}
the dilogarithm has an analytic continuation given by
\begin{equation}\label{li_def}
 -\int_{C(z)} \log(1-u) \frac{du}{u}
\end{equation}
where the contour of integration $C(z)$ is a path from $0$ to $z \in \C$. This makes the dilogarithm a multi-valued holomorphic function with a branch points at $1,$ $\infty$ (and off the principal branch another branch point at $0$).
 We let $\li(z)$ denote the dilogarithm on its principal branch so that $\li(z)$ is a single-valued holomorphic  function on $\C-[1,\infty)$.

To see why  the dilogarithm appears in our calculations, recall that $Q_{1k\sigma}(N)=-k^{-2} q^{\sigma} \prod_{j=1}^m 1/(1-q^j)$ for
$
q=e^{2\pi i/k}$ and $m=N-k$ by \eqref{cuv1} when $N/2<k\lqs N$.
Then for $0 \lqs m <k$
\begin{multline}
    \prod_{j=1}^m \frac{1}{1-q^j}   = \exp\Biggl(  -\sum_{j=1}^m \log\left(1-q^j \right)\Biggr) \\
     \approx \exp\left(  -\int_{0}^m \log\left(1-q^x \right) \,dx \right)
     = \exp\Biggl(  -\frac{k}{2\pi i} \int_{1}^{e^{2\pi i m/k}} \log\left(1-z\right) \,\frac{dz}{z} \Biggr) \\
     = \exp\left(  \frac{k}{2\pi i} \bigl(\li(e^{2\pi i m/k})-\li(1)\bigr) \right). \label{tijb}
\end{multline}
Of course, the approximation "$\approx$" above is not very accurate. We make it precise by separating the argument of $\prod_{j=1}^m 1/(1-q^j)$ from its absolute value $\spr{1/k}{m}$,  as in \eqref{wla}, and then carefully estimating $\spr{1/k}{m}$ in Propositions 3.9 and 3.14. The dilogarithms in \eqref{tijb} will reappear in \eqref{p(z)} and \eqref{a3(n)}.

For $z\in \C$ we have the functional equations
\begin{alignat}{2}\label{dilog1}
    \li(1/z) & =-\li(z)-\li(1)-\frac 12 \log^2(-z) \quad \quad & & z \not\in [0,\infty),\\
    \li(1-z) & =-\li(z)+\li(1)- \log(z)\log(1-z) \quad \quad & & z \not\in (-\infty,0] \cup [1,\infty) \label{dilog2}
\end{alignat}
from  \cite[Eqs. (3.2), (3.3)]{max}, where we mean the principal branch of the logarithm on $\C-(-\infty,0]$.
Replacing $z$ by $e^{2\pi i z}$  in \eqref{dilog1} and \eqref{dilog2} gives:
\begin{itemize}
\item For $m \in \Z$ and $m< \Re(z) < m+1$
\begin{equation}\label{ss}
    \li \left(e^{-2\pi i z} \right) = -\li \left(e^{2\pi i z} \right) +2\pi^2\left(z^2-(2m+1)z +m^2+m+1/6 \right).
\end{equation}
\item Let $(-i \infty,m]$ denote the vertical line in $\C$ made up of all points with real part $m \in \Z$ and imaginary part at most $0$. Then for $m-1/2< \Re(z)  <m+1/2$ and $z \not\in (-i \infty,m]$
\begin{equation}\label{sssp}
    \li \left(e^{2\pi i z} \right) = -\li \left(1-e^{2\pi i z} \right) + \li(1) -2\pi i (z-m) \log  \left(1-e^{2\pi i z} \right).
\end{equation}
\end{itemize}

We may describe $\li(z)$ for $z$ on the unit circle as
\begin{alignat}{2}
    \Re(\li(e^{2\pi i x}) ) & = \sum_{n=1}^\infty \frac{\cos(2\pi n x)}{n^2} = \pi^2 B_2(x-\lfloor x \rfloor) \qquad & & (x\in \R), \label{reli}\\
    \Im(\li(e^{2\pi i x}) ) & = \sum_{n=1}^\infty \frac{\sin(2\pi n x)}{n^2} = \cl(2\pi x) \qquad & & (x\in \R) \label{imc}
\end{alignat}
where $B_2(x):=x^2-x+1/6$ is the second  Bernoulli polynomial and
\begin{equation}\label{simo}
\cl(\theta):=-\int_0^\theta \log |2\sin( x/2) | \, dx \qquad (\theta \in \R)
\end{equation}
is Clausen's integral.
Note that $\li(1)=\zeta(2)=\pi^2/6$.


\SpecialCoor
\psset{griddots=5,subgriddiv=0,gridlabels=0pt}
\psset{xunit=0.8cm, yunit=0.5cm}
\psset{linewidth=1pt}
\psset{dotsize=4pt 0,dotstyle=*}

\begin{figure}[h]
\begin{center}
\begin{pspicture}(-5,-2.1)(11,2.3) 

\savedata{\mydata}[
{{-4., 1.35326}, {-3.95, 1.29496}, {-3.9, 1.23517}, {-3.85,
  1.17398}, {-3.8, 1.11149}, {-3.75, 1.04778}, {-3.7,
  0.98293}, {-3.65, 0.917028}, {-3.6, 0.850156}, {-3.55,
  0.782392}, {-3.5, 0.713817}, {-3.45, 0.644506}, {-3.4,
  0.574538}, {-3.35, 0.503986}, {-3.3, 0.432927}, {-3.25,
  0.361433}, {-3.2, 0.289578}, {-3.15, 0.217435}, {-3.1,
  0.145077}, {-3.05,
  0.0725742}, {-3., 0}, {-2.95, -0.0725742}, {-2.9,
-0.145077}, {-2.85, -0.217435}, {-2.8, -0.289578}, {-2.75,
-0.361433}, {-2.7, -0.432927}, {-2.65, -0.503986}, {-2.6, -0.574538},
{-2.55, -0.644506}, {-2.5, -0.713817}, {-2.45, -0.782392}, {-2.4,
-0.850156}, {-2.35, -0.917028}, {-2.3, -0.98293}, {-2.25, -1.04778},
{-2.2, -1.11149}, {-2.15, -1.17398}, {-2.1, -1.23517}, {-2.05,
-1.29496}, {-2., -1.35326}, {-1.95, -1.40997}, {-1.9, -1.46501},
{-1.85, -1.51826}, {-1.8, -1.56963}, {-1.75, -1.61901}, {-1.7,
-1.66628}, {-1.65, -1.71134}, {-1.6, -1.75405}, {-1.55, -1.79429},
{-1.5, -1.83193}, {-1.45, -1.86683}, {-1.4, -1.89884}, {-1.35,
-1.9278}, {-1.3, -1.95355}, {-1.25, -1.97592}, {-1.2, -1.99471},
{-1.15, -2.00973}, {-1.1, -2.02075}, {-1.05, -2.02756}, {-1.,
-2.02988}, {-0.95, -2.02746}, {-0.9, -2.01998}, {-0.85, -2.00713},
{-0.8, -1.98852}, {-0.75, -1.96374}, {-0.7, -1.93235}, {-0.65,
-1.8938}, {-0.6, -1.84751}, {-0.55, -1.79277}, {-0.5, -1.72876},
{-0.45, -1.6545}, {-0.4, -1.5688}, {-0.35, -1.47016}, {-0.3,
-1.35668}, {-0.25, -1.22581}, {-0.2, -1.07398}, {-0.15, -0.895776},
{-0.1, -0.682065}, {-0.05, -0.413607}, {0., 0.}, {0.05,
  0.413607}, {0.1, 0.682065}, {0.15, 0.895776}, {0.2, 1.07398}, {0.25,
   1.22581}, {0.3, 1.35668}, {0.35, 1.47016}, {0.4, 1.5688}, {0.45,
  1.6545}, {0.5, 1.72876}, {0.55, 1.79277}, {0.6, 1.84751}, {0.65,
  1.8938}, {0.7, 1.93235}, {0.75, 1.96374}, {0.8, 1.98852}, {0.85,
  2.00713}, {0.9, 2.01998}, {0.95, 2.02746}, {1., 2.02988}, {1.05,
  2.02756}, {1.1, 2.02075}, {1.15, 2.00973}, {1.2, 1.99471}, {1.25,
  1.97592}, {1.3, 1.95355}, {1.35, 1.9278}, {1.4, 1.89884}, {1.45,
  1.86683}, {1.5, 1.83193}, {1.55, 1.79429}, {1.6, 1.75405}, {1.65,
  1.71134}, {1.7, 1.66628}, {1.75, 1.61901}, {1.8, 1.56963}, {1.85,
  1.51826}, {1.9, 1.46501}, {1.95, 1.40997}, {2., 1.35326}, {2.05,
  1.29496}, {2.1, 1.23517}, {2.15, 1.17398}, {2.2, 1.11149}, {2.25,
  1.04778}, {2.3, 0.98293}, {2.35, 0.917028}, {2.4, 0.850156}, {2.45,
  0.782392}, {2.5, 0.713817}, {2.55, 0.644506}, {2.6,
  0.574538}, {2.65, 0.503986}, {2.7, 0.432927}, {2.75,
  0.361433}, {2.8, 0.289578}, {2.85, 0.217435}, {2.9,
  0.145077}, {2.95, 0.0725742}, {3.,
  0}, {3.05, -0.0725742}, {3.1, -0.145077}, {3.15,
-0.217435}, {3.2, -0.289578}, {3.25, -0.361433}, {3.3, -0.432927},
{3.35, -0.503986}, {3.4, -0.574538}, {3.45, -0.644506}, {3.5,
-0.713817}, {3.55, -0.782392}, {3.6, -0.850156}, {3.65, -0.917028},
{3.7, -0.98293}, {3.75, -1.04778}, {3.8, -1.11149}, {3.85, -1.17398},
{3.9, -1.23517}, {3.95, -1.29496}, {4., -1.35326}, {4.05, -1.40997},
{4.1, -1.46501}, {4.15, -1.51826}, {4.2, -1.56963}, {4.25, -1.61901},
{4.3, -1.66628}, {4.35, -1.71134}, {4.4, -1.75405}, {4.45, -1.79429},
{4.5, -1.83193}, {4.55, -1.86683}, {4.6, -1.89884}, {4.65, -1.9278},
{4.7, -1.95355}, {4.75, -1.97592}, {4.8, -1.99471}, {4.85, -2.00973},
{4.9, -2.02075}, {4.95, -2.02756}, {5., -2.02988}, {5.05, -2.02746},
{5.1, -2.01998}, {5.15, -2.00713}, {5.2, -1.98852}, {5.25, -1.96374},
{5.3, -1.93235}, {5.35, -1.8938}, {5.4, -1.84751}, {5.45, -1.79277},
{5.5, -1.72876}, {5.55, -1.6545}, {5.6, -1.5688}, {5.65, -1.47016},
{5.7, -1.35668}, {5.75, -1.22581}, {5.8, -1.07398}, {5.85,
-0.895776}, {5.9, -0.682065}, {5.95, -0.413607}, {6.,
0}, {6.05, 0.413607}, {6.1, 0.682065}, {6.15,
  0.895776}, {6.2, 1.07398}, {6.25, 1.22581}, {6.3, 1.35668}, {6.35,
  1.47016}, {6.4, 1.5688}, {6.45, 1.6545}, {6.5, 1.72876}, {6.55,
  1.79277}, {6.6, 1.84751}, {6.65, 1.8938}, {6.7, 1.93235}, {6.75,
  1.96374}, {6.8, 1.98852}, {6.85, 2.00713}, {6.9, 2.01998}, {6.95,
  2.02746}, {7., 2.02988}, {7.05, 2.02756}, {7.1, 2.02075}, {7.15,
  2.00973}, {7.2, 1.99471}, {7.25, 1.97592}, {7.3, 1.95355}, {7.35,
  1.9278}, {7.4, 1.89884}, {7.45, 1.86683}, {7.5, 1.83193}, {7.55,
  1.79429}, {7.6, 1.75405}, {7.65, 1.71134}, {7.7, 1.66628}, {7.75,
  1.61901}, {7.8, 1.56963}, {7.85, 1.51826}, {7.9, 1.46501}, {7.95,
  1.40997}, {8., 1.35326}, {8.05, 1.29496}, {8.1, 1.23517}, {8.15,
  1.17398}, {8.2, 1.11149}, {8.25, 1.04778}, {8.3, 0.98293}, {8.35,
  0.917028}, {8.4, 0.850156}, {8.45, 0.782392}, {8.5,
  0.713817}, {8.55, 0.644506}, {8.6, 0.574538}, {8.65,
  0.503986}, {8.7, 0.432927}, {8.75, 0.361433}, {8.8,
  0.289578}, {8.85, 0.217435}, {8.9, 0.145077}, {8.95,
  0.0725742}, {9.,
  0}, {9.05, -0.0725742}, {9.1, -0.145077}, {9.15,
-0.217435}, {9.2, -0.289578}, {9.25, -0.361433}, {9.3, -0.432927},
{9.35, -0.503986}, {9.4, -0.574538}, {9.45, -0.644506}, {9.5,
-0.713817}, {9.55, -0.782392}, {9.6, -0.850156}, {9.65, -0.917028},
{9.7, -0.98293}, {9.75, -1.04778}, {9.8, -1.11149}, {9.85, -1.17398},
{9.9, -1.23517}, {9.95, -1.29496}, {10., -1.35326}}
]

\psline[linecolor=gray]{->}(-4,0)(10,0)
\psline[linecolor=gray]{->}(0,-2.3)(0,2.5)
\psline[linecolor=gray](-0.15,2)(0.15,2)
\psline[linecolor=gray](-0.15,-2)(0.15,-2)
\multirput(-3,-0.15)(1,0){13}{\psline[linecolor=gray](0,0)(0,0.3)}

\dataplot[linecolor=orange,linewidth=0.8pt,plotstyle=line]{\mydata}

\rput(3,-0.6){$\pi$}
\rput(-3.1,-0.6){$-\pi$}
\rput(6,-0.6){$2\pi$}
\rput(9,-0.6){$3\pi$}
\rput(10,-0.6){$\theta$}
\rput(1,-0.6){$\pi/3$}
\rput(-0.35,2){$1$}
\rput(-0.45,-2){$-1$}
\rput(9,1.7){$\cl(\theta)$}

\psdots(1,2.02988)

\end{pspicture}
\caption{Clausen's integral}\label{bfig}
\end{center}
\end{figure}
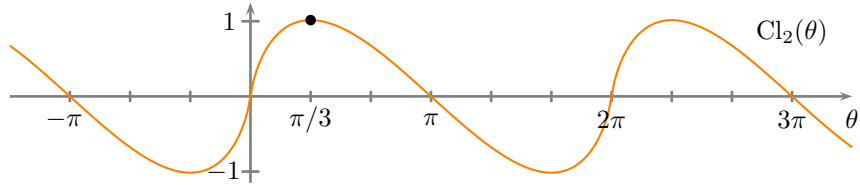



The graph of  $\cl(\theta)$ is shown in Figure \ref{bfig}. It is an odd function with period $2\pi$ and maximum value
$\cl(\pi/3)  \approx 1.0149416$ indicated.
Since, with \eqref{imc},
\begin{align}
    \cl(\theta) & = \left(\li(e^{i \theta})-\overline{\li(e^{i \theta})}\right)/(2i) \notag\\
    & = \left(\li(e^{i \theta})-\li(e^{-i \theta})\right)/(2i), \label{licl}
\end{align}
we may use \eqref{licl} to obtain the analytic continuation of $\cl(\theta)$ to $\theta \in \C$ with $0<\Re(\theta)<2\pi$ for example.
Another approach to this continuation combines \eqref{reli} and \eqref{imc} to get, for $m\lqs z \lqs m+1, m\in \Z$,
\begin{equation}\label{dli}
\cl\left(2\pi  z\right) = -i \li\left( e^{2\pi i z}\right)+i\pi^2\left(z^2-(2m+1)z+m^2+m+1/6 \right).
\end{equation}
Then  the right of \eqref{dli} gives the continuation of $\cl\left(2\pi  z\right)$ to $z \in \C$ with $m<\Re(z)<m+1$.

  As $z$ crosses the branch cuts  the dilogarithm enters new branches. From \cite[Sect.  3]{max}, the value of the analytically continued dilogarithm is always given by
\begin{equation}\label{dilogcont}
\li(z) + 4\pi^2  A +   2\pi i  B  \log \left(z\right)
\end{equation}
for some $A$, $B \in \Z$. For example, the simplest way to get to a branch corresponding to \eqref{dilogcont} is  to circle $z=1$ once in the negative direction, then circle $z=0$  in the positive direction $A$ times and  circle $z=1$  in the negative direction  $B-1$ times (using the opposite directions if $A$ or $B$ is negative).

Zeros of the analytically continued dilogarithm will play a key role in our asymptotic calculations and in \cite{OS3} we have made a study of them all. When the continued dilogarithm takes the form \eqref{dilogcont} with $B=0$, there will be a  zero if and only if $A\gqs 0$ and, for each such $A$, the zero will be unique and lie on the real line. The cases we will require have $B\neq 0$. In these cases there are no real zeros so we may avoid the branch cuts and look for solutions to
\begin{equation}\label{dilogzero}
\li(z) + 4\pi^2  A +   2\pi i  B  \log \left(z\right)=0 \qquad (z\in \C, z \not\in (-\infty,0] \cup [1,\infty), \ A,B \in \Z).
\end{equation}
The next result is shown in Theorems 1.1 and 1.3 of \cite{OS3}.

\begin{theorem} \label{dilab}
For nonzero $B \in \Z$, \eqref{dilogzero} has solutions  if and only if $-|B|/2<A\lqs |B|/2$. For such a pair $A,B$ the solution $z$ is unique.
This unique solution, $w(A,B)$,  may be found to arbitrary precision using Newton's method.
\end{theorem}
\begin{proof}[Sketch of proof]
By considering the two curves where the real part and the imaginary part of \eqref{dilogcont} vanish, it can be shown that they intersect if and only if $-|B|/2<A\lqs |B|/2$. It can also be shown that $w(A,B)$ is close to $e^{2\pi i A/B}$ and from this starting point (the case $A=0$ needs an adjustment) Newton's method will always converge to  $w(A,B)$.
\end{proof}

 By conjugating \eqref{dilogzero} it is clear that
\begin{equation*}
    w(A,-B) = \overline{w(A,B)}.
\end{equation*}
So for nonzero $B$ the first zeros are $w(0,1)$ and its conjugate $w(0,-1)$.   We have
\begin{equation*}\label{w0-1}
    w(0,-1) \approx \phantom{-}0.9161978162 - 0.1824588972 i
\end{equation*}
and this zero was denoted by $w_0$ in Section \ref{sec-int}.
The next few zeros are
\begin{align*} 
    w(0,-2) &\approx \phantom{-}0.9684820460-0.1095311065 i \\
    w(1,-2) &\approx -0.9943069304-0.0648889318 i \\ 
    w(-1,-3) &\approx -0.5459030969+0.8812307423 i \\ 
    w(0,-3) &\approx \phantom{-}0.9832603795-0.0777596389 i \\ 
    w(1,-3) &\approx -0.4594734813-0.8485350380 i 
\end{align*}
where $w(0,-2)$ and $w(1,-3)$ will be required in Section \ref{sec-fur}.

Define
\begin{equation}\label{pdfn}
    p_d(z):=\frac{ - \li\left(e^{2\pi i z}\right) +\li(1) +4\pi^2 d}{2\pi i z},
\end{equation}
a single-valued holomorphic function away from the branch cuts $(-i\infty,n]$ for $n \in \Z$. In Section \ref{sec-sad} we will require the solution of $p'_0(z)=0$ (and in \cite{OS2} solutions to $p'_d(z)=0$ more generally).

\begin{theorem} \label{disol}
Fix integers $m$ and $d$ with $-|m|/2<d\lqs |m|/2$. Then there is a unique solution to $p_d'(z)=0$ for $z \in \C$ with $m-1/2<\Re(z)<m+1/2$ and $z \not\in (-i\infty,m]$. Denoting this solution by $z^*$, it is given by
\begin{equation}\label{uniq}
    z^*=m+\frac{\log \bigl(1-w(d,-m)\bigr)}{2\pi i}
\end{equation}
and satisfies
\begin{equation} \label{pzlogw}
    p_d(z^*)=\log \bigl(w(d,-m)\bigr).
\end{equation}
\end{theorem}
\begin{proof}
Note that
\begin{equation*}
    \frac{d}{dz} \li \left(e^{2\pi i z}\right) = -2\pi i  \log \left(1-e^{2\pi i z}\right)
\end{equation*}
for $z$ not on any of the vertical lines $(-i \infty,n]$, $n\in \Z$. So
\begin{align*}
    p_d(z)+z p_d'(z) & =  \frac{d}{dz}  \left( z p_d(z)\right)\\
    & = \frac{d}{dz}  \left(\frac{\li(1) +4\pi^2 d}{2\pi i} - \frac{\li\left(e^{2\pi i z}\right)}{2\pi i} \right) =  \log \left(1-e^{2\pi i z}\right)
\end{align*}
and hence
\begin{equation}\label{pzzz}
    p_d'(z)=-\frac 1z \left(p_d(z)-\log \left(1-e^{2\pi i z}\right) \right).
\end{equation}
Similarly
\begin{equation}\label{pzzz2}
    p_d''(z)=-\frac 1z \left(2 p_d'(z) + \frac{2\pi i \cdot e^{2\pi i z}}{1-e^{2\pi i z}} \right).
\end{equation}

With \eqref{pdfn} we may expand \eqref{pzzz} into
\begin{equation}\label{pz0}
    2\pi i z^2 p_d'(z)=\li \left( e^{2\pi i z}\right)-\li(1)- 4\pi^2 d +2\pi i z \log \left(1-e^{2\pi i z}\right).
\end{equation}
Applying the functional equation \eqref{sssp} to \eqref{pz0} implies
$$
2\pi i z^2 p_d'(z) = -\li\left(1-e^{2\pi i z}\right) - 4\pi^2 d +2\pi i m \log \left(1-e^{2\pi i z}\right)
$$
for $m-1/2<\Re(z)<m+1/2$.
Letting
$
w = 1-e^{2\pi i z}
$,
we are now looking for solutions to the equation
\begin{equation}\label{dz1}
\li\left(w\right) +  4\pi^2 d-2\pi i m \log \left(w\right)=0
\end{equation}
and Theorem \ref{dilab} gives the unique solution as $w(d,-m)$ when $-|m|/2<d\lqs |m|/2$. The formula \eqref{uniq} follows and then \eqref{pzzz} implies
\eqref{pzlogw}.
\end{proof}

\section{Estimates for the sine product $\spr{h/k}{m}$} \label{sec-est}



Sudler in \cite{Su} approximates  $\spn{\theta}{m}$ using the first term of the Euler-Maclaurin summation formula
 and finds that the $\theta \in (0,1)$ that maximizes $|\spn{\theta}{m}|$ is approximately $x_0/m$ where $x_0 \approx 0.791227$ is the $x$ value in $(0,1)$ where $\frac{d}{dx}(\cl(2\pi x)/(2\pi x))$ vanishes.   Wright in \cite{Wri64} uses more terms in the summation to get more detailed results, as do Freiman and Halberstam in \cite{FH}. We use similar techniques in the next subsection but require arbitrarily many terms of the summation formula.

\subsection{Euler-Maclaurin summation}
We need to estimate the size of $\spn{\theta}{m}$ accurately and also replace it with a continuous (and later holomorphic) function of $m$.

Let $\rho(z):=\log \bigl((\sin z)/z\bigr)$, a holomorphic function for $|z|<\pi$ that satisfies $\rho(-z)=\rho(z)$.
Also
\begin{equation} \label{crho}
    \cot(\pi z) = \frac 1{\pi z}+\rho'(\pi z)
\end{equation}
and so
\begin{equation}\label{mar}
    \cot^{(k)}(\pi z) = \frac{(-1)^k k!}{(\pi z)^{k+1}}+\rho^{(k+1)}(\pi z).
\end{equation}

\begin{prop} \label{prrh}
For $m$, $L \in \Z_{\gqs 1}$ and $-1/m<\theta<1/m$ with $\theta \neq 0$ we have
\begin{multline} \label{pimv2}
    \spn{\theta}{m} = \left(\frac{ \theta}{|\theta|}\right)^{m}\left(\frac{2 \sin(\pi m \theta)}{\theta}\right)^{1/2}  \exp\left(-\frac{\cl(2\pi m \theta)}{2\pi \theta} \right) \\
      \quad \times \exp\left(\sum_{\ell=1}^{L-1} \frac{B_{2\ell}}{(2\ell)!} ( \pi \theta)^{2\ell-1} \cot^{(2\ell-2)} (\pi m \theta)\right) \exp \bigl(T_{L}(m,\theta)\bigr)
  \end{multline}
for
\begin{equation*}
    T_{L}(m,\theta) := \left( \pi \theta\right)^{2L} \int_0^m \frac{B_{2L}-B_{2L}(x-\lfloor x \rfloor)}{(2L)!} \rho^{(2L)}(\pi x \theta) \, dx + \int_0^\infty \frac{B_{2L}-B_{2L}(x - \lfloor x \rfloor)}{2L(x+m)^{2L}} \, dx.
\end{equation*}
  \end{prop}
\begin{proof}
Write
\begin{align}
  \spn{\theta}{m} &= (2\pi \theta)^m m! \prod_{j=1}^m \frac{\sin (\pi j \theta)}{\pi j \theta} \notag\\
   &= (2\pi \theta)^m m!  \prod_{j=1}^m \exp\left(\rho (\pi j \theta) \right) = (2\pi \theta)^m m!  \exp \sum_{j=1}^m \rho (\pi j \theta). \label{tagr}
\end{align}
With Euler-Maclaurin summation, as in \cite[Chap. 2]{Ra} or \cite[p. 285]{Ol}, we obtain for $|\theta|<1/m$,
\begin{multline} \label{ems}
\sum_{j=1}^m \rho(\pi j \theta) = \int_0^m \rho(\pi x \theta)\, dx +\frac 12\left(\rho(\pi m \theta) - \rho(\pi 0 \theta)\right)\\
+ \sum_{\ell =1}^{L-1} \frac{B_{2\ell }}{(2\ell )!} (\pi \theta)^{2\ell -1}\left\{ \rho^{(2\ell -1)}(\pi m \theta) - \rho^{(2\ell -1)}(\pi 0 \theta)\right\} +R_L(m,\theta)
\end{multline}
where $L \gqs 1$ and
\begin{equation}\label{drl}
R_L(m,\theta) :=\left( \pi \theta\right)^{2L} \int_0^m \frac{B_{2L}-B_{2L}(x-\lfloor x \rfloor)}{(2L)!} \rho^{(2L)}(\pi x \theta) \, dx.
\end{equation}
The integral in \eqref{ems} may be evaluated using \eqref{simo} to get
\begin{equation*}
\int_0^m \rho(\pi x \theta)\, dx = -m \log |2\pi m \theta| +m -\frac{\cl(2\pi m \theta)}{2\pi \theta} \qquad (\theta \neq 0)
\end{equation*}
 and
therefore
\begin{multline} \label{frm1}
\spn{\theta}{m} =  \left(\frac{\theta e}{|\theta| m}\right)^m m! \left(\frac{\sin(\pi m \theta)}{\pi m \theta}\right)^{1/2}
\exp\left(-\frac{\cl(2\pi m \theta)}{2\pi \theta} \right)\\
 \times \exp\left(\sum_{\ell =1}^{L-1} \frac{B_{2\ell }}{(2\ell )!} ( \pi \theta)^{2\ell -1} \rho^{(2\ell -1)}\left(\pi m \theta \right)\right) \exp\left(R_L(m,\theta)\right).
\end{multline}
Stirling's formula is
\begin{equation}\label{sumgam}
\log \G(m)=(m-\frac 12)\log m - m + \frac 12 \log 2\pi + \sum_{\ell =1}^{L-1} \frac{B_{2\ell }}{2\ell (2\ell -1)m^{2\ell -1}} + S_L(m)
\end{equation}
with
\begin{equation} \label{dsl}
S_L(m) := \int_0^\infty \frac{B_{2L}-B_{2L}(x - \lfloor x \rfloor)}{2L(x+m)^{2L}} \, dx
\end{equation}
as in \cite[(4.03) p. 294]{Ol}. Hence
\begin{equation}\label{oasi}
\left(\frac{e}{m}\right)^m m! = \left(\frac{e}{m}\right)^m m \Gamma(m) = (2\pi m)^{1/2} \exp \left( \sum_{\ell =1}^{L-1} \frac{B_{2\ell }}{2\ell (2\ell -1)m^{2\ell -1}}  \right) \exp \left( S_L(m) \right).
\end{equation}
Putting \eqref{oasi} into \eqref{frm1},
recombining the two sums with \eqref{mar}, and  setting $T_{L}(m,\theta) := R_{L}(m,\theta)+S_{L}(m)$
 completes the proof.
\end{proof}

\subsection{Derivatives of the cotangent} \label{dcot}
We next examine the cotangent function and its derivatives in detail.
For all $z \in \C$,
\begin{equation}
    \cot(\pi z)  = i + \frac{2i}{e^{2\pi i z}-1}
     = -i - \frac{2i}{e^{-2\pi i z}-1}. \label{cotb}
\end{equation}
For $k \gqs 1$, by induction,
\begin{align}
    \cot^{(k)}(z) & = (-1)^k (2i)^{k+1} \sum_{r=1}^{k+1} (r-1)! \stirb{k+1}{r} \frac 1{(e^{2 i z}-1)^r}, \label{cotc}\\
     & = (-1)^k (-2i)^{k+1} \sum_{r=1}^{k+1} (r-1)! \stirb{k+1}{r} \frac 1{(e^{-2 i z}-1)^r}. \label{cotd}
\end{align}
As in \cite[Chap. 6]{Knu} these Stirling numbers satisfy the relations
\begin{equation} \label{stir}
    \stirb{k}{r-1} + r \stirb{k}{r}  = \stirb{k+1}{r}, \qquad \sum_{r=0}^k \stirb{k}{r} x(x-1) \cdots (x-r+1)  = x^k.
\end{equation}
 For $k \gqs 0$, $\cot^{(k)}(\pi z)$ is clearly holomorphic in $\C$ except for poles when $z \in \Z$.

\begin{lemma} \label{2c}
For $c>0$ and $k \in \Z_{\gqs 0}$
\begin{equation*}
   \sum_{\ell =1}^\infty \ell^k e^{-c \ell} \lqs k! \left( \frac 2c \right)^{k+1} .
\end{equation*}
\end{lemma}
\begin{proof}
The result follows by comparing the series to the integral
\begin{equation*}
    \int_0^\infty x^k e^{-c x} \, dx = \frac{\Gamma(k+1)}{c^{k+1}} = \frac{k!}{c^{k+1}}. \qedhere
\end{equation*}
\end{proof}

\begin{theorem} \label{cotder}
For all nonzero $z \in \C$ with $-1/2 \lqs \Re(z) \lqs 1/2$ we have
\begin{equation} \label{cotthma}
    \left| \cot^{(k)}(\pi z) \right| \lqs \delta_{0,k}+ 20 \frac{k!}{\pi^{k+1}} \left(\frac 1{|z|^{k+1}}+ 8^{k+1} \right)  e^{-\pi |y|}.
\end{equation}
Also, for all $z \in \C$ with $|y|\gqs 1$,
\begin{equation} \label{cotthmb}
    \left| \cot^{(k)}(\pi z) \right| \lqs \delta_{0,k}+\frac{k!}{\pi^{k+1}} \left(\frac{4.01}{|y|} \right)^{k+1}  e^{-\pi |y|}.
\end{equation}
\end{theorem}
\begin{proof}
By \cite[(11.1)]{Ra} and \eqref{crho}
\begin{equation} \label{simo2}
    -\rho'(w) = \sum_{r=1}^\infty \frac{2^{2r} |B_{2r}|}{(2r)!} w^{2r-1} \qquad (|w|<\pi)
\end{equation}
so that all the coefficients of $-\rho'(w)$ are positive. Hence, with $|w|<\pi$ and the bound
\begin{equation}\label{bk2}
\frac{|B_{2n}|}{(2n)!}  \lqs  \frac{\pi^2}{3(2\pi)^{2n}}
\end{equation}
from \cite[(9.6)]{Ra}, we have
\begin{equation*}
    |\rho'(w)| \lqs -\rho'(|w|)  \lqs \frac{\pi}3 \sum_{r=1}^\infty \left( \frac{|w|}{\pi} \right)^{2r-1}
     \lqs \frac{\pi}3 \sum_{r=0}^\infty \left( \frac{|w|}{\pi} \right)^{r}
     \lqs \frac{\pi}{3 (1-|w|/\pi)}.
\end{equation*}
Letting $f(t):=\displaystyle \frac{\pi}{3 (1-t/\pi)}$ we see that
\begin{equation*}
    f^{(k)}(t) = \frac{k!}{3 \pi^{k-1}(1-t/\pi)^{k+1}}.
\end{equation*}
Since the power series coefficients of $f$ are greater than the corresponding power series coefficients of $-\rho'$, and all coefficients are positive, it follows that the coefficients of $f^{(k)}$ are greater than the corresponding coefficients of $-\rho^{(k+1)}$. Therefore
\begin{equation} \label{poiu}
    \left| \rho^{(k+1)}(w) \right| \lqs \frac{k!}{3 \pi^{k-1}(1-|w|/\pi)^{k+1}} \qquad (k\gqs 0, \ |w|<\pi).
\end{equation}
With \eqref{mar} and \eqref{poiu}
we have proved that
\begin{equation} \label{sttr}
    \left| \cot^{(k)}(\pi z) \right| \lqs \frac{\pi^2}{3} \frac{k!}{\pi^{k+1}} \left( \frac{1}{|z|^{k+1}}+ \frac{1}{(1-|z|)^{k+1}}\right) \qquad(|z|<1).
\end{equation}

Next we assume $y \neq 0$.
Formulas  \eqref{cotb}, \eqref{cotc} and \eqref{cotd} imply
\begin{align}
    \left| \cot^{(k)}(\pi z) \right| &  \lqs \delta_{0,k}+2^{k+1} \sum_{r=1}^{k+1} (r-1)! \stirb{k+1}{r} \frac{e^{-2\pi |y|r}}{(1-e^{-2\pi |y|})^r} \notag\\
     & \lqs  \delta_{0,k}+\frac{2^{k+1}}{(1-e^{-2\pi |y|})^{k+1}}\sum_{r=1}^{k+1} r^k e^{-2\pi |y| r}\notag\\
     & \lqs \delta_{0,k}+\left(\frac{2}{1-e^{-2\pi |y|}}\right)^{k+1} e^{-\pi |y| }\sum_{r=1}^{k+1} r^k e^{-\pi |y| r} \label{tttag}
\end{align}
where we used that $(r-1)! \stirb{k+1}{r} \lqs r^k$ which follows from the relation on the right of \eqref{stir} with $x=r$.
Lemma \ref{2c} applied to \eqref{tttag}  shows
\begin{equation} \label{sttr2x}
    \left| \cot^{(k)}(\pi z) \right|  \lqs \delta_{0,k}+k! \left(\frac{4}{\pi|y|(1-e^{-2\pi |y|})}\right)^{k+1} e^{-\pi |y| } \qquad(y \neq 0).
\end{equation}
Then \eqref{cotthma} in the statement of the theorem follows by combining \eqref{sttr} for $|y|\lqs 0.55$ and \eqref{sttr2x} for $|y|\gqs 0.55$. Finally, \eqref{cotthmb} follows from \eqref{sttr2x}.
\end{proof}

\subsection{Initial estimates for $\spr{h/k}{m}$}

Stirling's formula implies
\begin{equation*}
  2\sqrt{n} \left(\frac n e\right)^n   <  n!    <  3\sqrt{n} \left(\frac n e\right)^n \qquad (n \in \Z_{\gqs 1})
\end{equation*}
and it follows that
\begin{equation}\label{k!}
(n-1)!  <  3 \left( n / e\right)^n.
\end{equation}
Since $|B_{2n}(x-\lfloor x \rfloor)| \lqs |B_{2n}|$ as in \cite[Thm 1.1, p. 283]{Ol}, we see that $B_{2n}-B_{2n}(x-\lfloor x \rfloor)$ has the same sign as $B_{2n}$ if it's non-zero. Therefore the terms in \eqref{sumgam} alternate in sign and it follows that (\cite[(4.05) p. 294]{Ol})
\begin{equation}\label{slm}
|S_L(m)| \lqs \frac{|B_{2L}|}{2L(2L-1)m^{2L-1}}.
\end{equation}
Employing \eqref{bk2} and \eqref{k!} in \eqref{slm} gives
\begin{equation}\label{slb}
|S_L(m)| \lqs  \frac{\pi}{6} \frac{(2L-2)!}{(2\pi m)^{2L-1}} < \frac{\pi}{2} \left(\frac{2L-1}{2\pi e m}\right)^{2L-1} \qquad (L, m \in \Z_{\gqs 1}).
\end{equation}

\begin{lemma} \label{melrh}
For  $m$, $L \in \Z_{\gqs 1}$ and $-1/m<\theta<1/m$ we have
\begin{equation}
    |R_L(m,\theta)|
    \lqs \frac{\pi^3}{3} \left(\frac{(2L-1)|\theta|}{2\pi e (1-m |\theta|)}\right)^{2L-1}. \label{rlb}
\end{equation}
\end{lemma}
\begin{proof}
With \eqref{drl}, \eqref{bk2} and the inequality $|B_{2n}-B_{2n}(x-\lfloor x \rfloor)|  \lqs  2|B_{2n}|$ we have
\begin{align*}
    |R_L(m,\theta)| & \lqs  \frac{2 \pi^2}{3(2\pi)^{2L}} \left( \pi |\theta|\right)^{2L} \int_0^m \left|\rho^{(2L)}\left(\pi x \theta\right)\right| \, dx  \\
    & = \frac{\pi}{3} \left( \frac{ |\theta|}{2}\right)^{2L-1} \left|\rho^{(2L-1)}\left(\pi m \theta\right)\right|.
\end{align*}
Then applying \eqref{poiu} shows
\begin{equation}\label{rlm}
    |R_L(m,\theta)| \lqs \frac{\pi^3}{9} (2L-2)! \left(\frac{|\theta|}{2\pi  (1-m |\theta|)}\right)^{2L-1}.
\end{equation}
The result follows with \eqref{k!}.
\end{proof}

We now concentrate on the case where $\theta=h/k$ for relatively prime integers $k>h \gqs 1$. We think of $h$, $k$ as fixed with integer $m$  varying in the range $1 \lqs m <k/h$.

\begin{lemma} \label{melrh2}
For  $1 \lqs m <k/h$ we have
\begin{equation} \label{1/12}
    |T_1(m,h/k)| \lqs \pi^2 h/18 + 1/12.
\end{equation}
\end{lemma}
\begin{proof}
Note that since $m$ is an integer
\begin{equation}\label{impl}
    1 \lqs m <k/h \implies h \lqs m h \lqs k-1.
\end{equation}
Consequently
$
1/(1-mh/k) \lqs k
$ and using this in \eqref{rlm} gives a bound
for $R_1$. Bound  $S_1$  with \eqref{slm}.
\end{proof}

Define
\begin{equation}\label{csh}
     c(h):=\frac{h^{1/2}}{2} \exp( \pi^2 h/18 +1/6).
\end{equation}
(We increased $1/12$ in \eqref{1/12} to $1/6$ to ensure $c(h)>1$ for $h\gqs 1$, as needed in Proposition \ref{gprop}.) The next result gives us our initial estimate for $\spr{h/k}{m}$.

\begin{prop} \label{pbnd}
 For $1 \lqs m <k/h$
\begin{equation*}
  \spr{h/k}{m}  \lqs c(h)  \exp\left(\frac {k}{2\pi h} \cl\bigl(2\pi  m h/k \bigr) \right).
\end{equation*}
\end{prop}
\begin{proof}
Combining Lemma \ref{melrh2} with Proposition \ref{prrh} shows
\begin{equation} \label{prrh2}
  \spr{h/k}{m}  \lqs  \left(\frac{h}{2k \sin(\pi m h/k)}\right)^{1/2} \exp(\pi^2 h/18 +1/12) \exp\left(\frac {k}{2\pi h} \cl\bigl(2\pi  m h/k \bigr) \right) .
\end{equation}
Note the simple inequality
\begin{equation} \label{trig}
    1 \lqs \frac x{\sin x} \lqs \frac{\pi}{2} \qquad (-\pi/2 \lqs x \lqs \pi/2)
\end{equation}
and hence
\begin{equation*}
    1\lqs \frac 1{\sin x} \lqs \frac 1{\sin \varepsilon} \lqs \frac{\pi}{2 \varepsilon} \qquad \text{for} \qquad 0 \lqs \varepsilon \lqs x \lqs \pi-\varepsilon.
\end{equation*}
It
  follows from \eqref{impl} that $\pi/k \lqs \pi m h/k \lqs \pi - \pi/k$ and so
\begin{equation}
    \frac h{2k \sin(\pi m h/k)} \lqs \frac{h}{4} \qquad (1 \lqs m < k/h). \label{lbl}
\end{equation}
Inequalities \eqref{prrh2} and \eqref{lbl} complete the proof.
\end{proof}

Proposition \ref{pbnd} implies that
\begin{equation} \label{pbndq}
\spr{h/k}{m}  \lqs c(h) \qquad \text{for} \qquad k/2h \lqs m <k/h,
\end{equation}
 since $\cl(\theta)\lqs 0$ for $\pi \lqs \theta \lqs 2\pi$. For the rest of this subsection we focus on $m$ in the range $1 \lqs m \lqs k/2h$. Our next goal is to show that for $m$ near the end points of this range the product $\spr{h/k}{m} $ is also quite small - see Figure \ref{prods}. We first develop a simpler version of the bound in Proposition \ref{pbnd}.

\begin{lemma} \label{clthe}
For  $0 \lqs \theta \lqs \pi$ we have
\begin{align}
\cl(\theta) & \lqs \theta -\theta \log \theta +\theta^3/54, \label{clthe1}\\
\cl(\pi - \theta) & \lqs \theta \log 2. \label{clthe2}
\end{align}
\end{lemma}
\begin{proof}
Integrate \eqref{simo2} twice and use \eqref{simo} to show that, for $0 \lqs \theta \lqs \pi$,
\begin{align*}
    \cl(\theta) & = \theta - \theta \log \theta + \sum_{n=1}^\infty \frac{|B_{2n}|}{2n (2n+1)!} \theta^{2n+1} \\
    & \lqs \theta - \theta \log \theta + \theta^3 \frac{\pi^2}{3} \sum_{n=1}^\infty \frac{\theta^{2n-2}}{ 2n (2n+1) (2\pi)^{2n}} .
\end{align*}
The series above is bounded by
$$
\sum_{n=1}^\infty \frac{\pi^{2n-2}}{ 2n (2n+1) (2\pi)^{2n}} = \frac{1}{\pi^{2}}\sum_{n=1}^\infty \frac{1}{ 2n (2n+1) 4^{n}}
< \frac{1}{24 \pi^{2}}\sum_{n=0}^\infty  4^{-n}
$$
and \eqref{clthe1} follows.

Put $f(\theta):=\theta \log 2 - \cl(\pi - \theta)$. Then $f'(\theta)=-\log \sin((\pi -\theta)/2) \gqs 0$ and so $f(\theta)$ is increasing on $[0,\pi)$ and therefore $f(\theta)\gqs f(0)=0$, proving \eqref{clthe2}.
\end{proof}

\begin{lemma} \label{vm}
For $1 \lqs m \lqs k/2h$ we have
\begin{align}\label{vm1}
    \spr{h/k}{m} & < c(h) \left( \frac k{ m h} \right)^m,\\
    \spr{h/k}{m} & < c(h) 2^{\displaystyle  k/(2 h)-m}. \label{vm2}
\end{align}
\end{lemma}
\begin{proof}
From Proposition \ref{pbnd} and \eqref{clthe1},
\begin{align*}
  \spr{h/k}{m}  & \lqs c(h) \exp\left(m\left(1-\log(2\pi  m h/k)+\left( 2\pi  m h/k \right)^2/54 \right)\right) \\
   &\lqs  c(h) \exp\left(m\left(1+\log(k/(2\pi  m h))+\pi^2/54 \right)\right)  \\
   &= c(h) \left(\frac{e}{2\pi} e^{\pi^2/54}\frac  k{m h} \right)^m <   c(h) \left( \frac  k{ m h} \right)^m.
\end{align*}
Similarly, Proposition \ref{pbnd} and \eqref{clthe2} give \eqref{vm2}.
\end{proof}



\SpecialCoor
\psset{griddots=5,subgriddiv=0,gridlabels=0pt}
\psset{xunit=0.8cm, yunit=0.8cm}
\psset{linewidth=1pt}
\psset{dotsize=2pt 0,dotstyle=*}

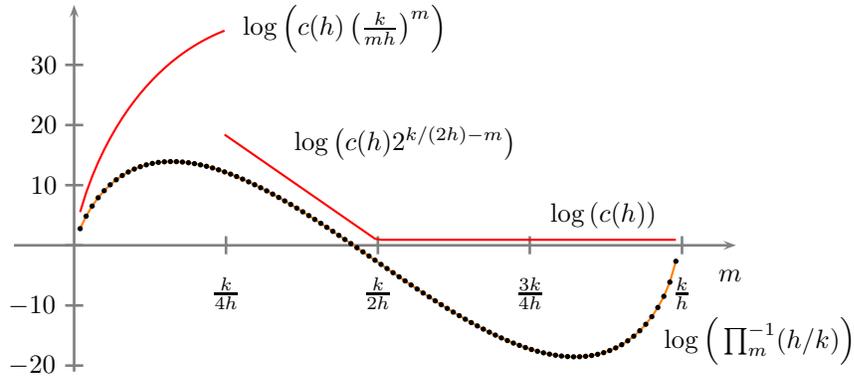
\begin{figure}[h]
\begin{center}
\begin{pspicture}(-1,-2.1)(11.5,4) 

\savedata{\mydata}[
{{0.1, 0.552687}, {0.2, 0.875074}, {0.3, 1.14514}, {0.4,
  1.38122}, {0.5, 1.59203}, {0.6, 1.78271}, {0.7, 1.95664}, {0.8,
  2.11624}, {0.9, 2.26331}, {1., 2.39924}, {1.1, 2.52516}, {1.2,
  2.64197}, {1.3, 2.75044}, {1.4, 2.85121}, {1.5, 2.94483}, {1.6,
  3.03178}, {1.7, 3.11248}, {1.8, 3.18729}, {1.9, 3.25654}, {2.,
  3.32052}, {2.1, 3.37951}, {2.2, 3.43373}, {2.3, 3.4834}, {2.4,
  3.52872}, {2.5, 3.56988}}
]

\savedata{\mydatab}[
{{2.5, 1.84187}, {2.6, 1.77255}, {2.7, 1.70324}, {2.8, 1.63392}, {2.9,
   1.56461}, {3., 1.49529}, {3.1, 1.42598}, {3.2, 1.35667}, {3.3,
  1.28735}, {3.4, 1.21804}, {3.5, 1.14872}, {3.6, 1.07941}, {3.7,
  1.01009}, {3.8, 0.940777}, {3.9, 0.871462}, {4., 0.802147}, {4.1,
  0.732833}, {4.2, 0.663518}, {4.3, 0.594203}, {4.4, 0.524889}, {4.5,
  0.455574}, {4.6, 0.386259}, {4.7, 0.316944}, {4.8, 0.24763}, {4.9,
  0.178315}, {5., 0.109}, {5.1, 0.0916716}, {5.2, 0.0916716}, {5.3,
  0.0916716}, {5.4, 0.0916716}, {5.5, 0.0916716}, {5.6,
  0.0916716}, {5.7, 0.0916716}, {5.8, 0.0916716}, {5.9,
  0.0916716}, {6., 0.0916716}, {6.1, 0.0916716}, {6.2,
  0.0916716}, {6.3, 0.0916716}, {6.4, 0.0916716}, {6.5,
  0.0916716}, {6.6, 0.0916716}, {6.7, 0.0916716}, {6.8,
  0.0916716}, {6.9, 0.0916716}, {7., 0.0916716}, {7.1,
  0.0916716}, {7.2, 0.0916716}, {7.3, 0.0916716}, {7.4,
  0.0916716}, {7.5, 0.0916716}, {7.6, 0.0916716}, {7.7,
  0.0916716}, {7.8, 0.0916716}, {7.9, 0.0916716}, {8.,
  0.0916716}, {8.1, 0.0916716}, {8.2, 0.0916716}, {8.3,
  0.0916716}, {8.4, 0.0916716}, {8.5, 0.0916716}, {8.6,
  0.0916716}, {8.7, 0.0916716}, {8.8, 0.0916716}, {8.9,
  0.0916716}, {9., 0.0916716}, {9.1, 0.0916716}, {9.2,
  0.0916716}, {9.3, 0.0916716}, {9.4, 0.0916716}, {9.5,
  0.0916716}, {9.6, 0.0916716}, {9.7, 0.0916716}, {9.8,
  0.0916716}, {9.9, 0.0916716}, {10., 0.0916716}}
]

\savedata{\mydatac}[
{{0.1, 0.277244}, {0.2, 0.485223}, {0.3, 0.652736}, {0.4,
  0.791596}, {0.5, 0.908287}, {0.6, 1.00693}, {0.7, 1.09036}, {0.8,
  1.16069}, {0.9, 1.21952}, {1., 1.26812}, {1.1, 1.30754}, {1.2,
  1.33863}, {1.3, 1.36213}, {1.4, 1.37867}, {1.5, 1.38878}, {1.6,
  1.39296}, {1.7, 1.39162}, {1.8, 1.38514}, {1.9, 1.37388}, {2.,
  1.35813}, {2.1, 1.3382}, {2.2, 1.31433}, {2.3, 1.28678}, {2.4,
  1.25576}, {2.5, 1.2215}, {2.6, 1.18418}, {2.7, 1.14399}, {2.8,
  1.10111}, {2.9, 1.0557}, {3., 1.00792}, {3.1, 0.957924}, {3.2,
  0.90585}, {3.3, 0.851838}, {3.4, 0.796021}, {3.5, 0.738528}, {3.6,
  0.679481}, {3.7, 0.619001}, {3.8, 0.557204}, {3.9, 0.494205}, {4.,
  0.430114}, {4.1, 0.365039}, {4.2, 0.299087}, {4.3, 0.232363}, {4.4,
  0.164969}, {4.5, 0.0970072}, {4.6,
  0.0285776}, {4.7, -0.0402201}, {4.8, -0.109287}, {4.9, -0.178526},
{5., -0.247837}, {5.1, -0.317125}, {5.2, -0.38629}, {5.3, -0.455234},
{5.4, -0.52386}, {5.5, -0.592069}, {5.6, -0.659759}, {5.7,
-0.726831}, {5.8, -0.793182}, {5.9, -0.858709}, {6., -0.923305},
{6.1, -0.986864}, {6.2, -1.04928}, {6.3, -1.11043}, {6.4, -1.17021},
{6.5, -1.22849}, {6.6, -1.28516}, {6.7, -1.34009}, {6.8, -1.39315},
{6.9, -1.44421}, {7., -1.49311}, {7.1, -1.53973}, {7.2, -1.58389},
{7.3, -1.62545}, {7.4, -1.66422}, {7.5, -1.70004}, {7.6, -1.73271},
{7.7, -1.76202}, {7.8, -1.78775}, {7.9, -1.80969}, {8., -1.82756},
{8.1, -1.8411}, {8.2, -1.85001}, {8.3, -1.85397}, {8.4, -1.8526},
{8.5, -1.84551}, {8.6, -1.83225}, {8.7, -1.81231}, {8.8, -1.78509},
{8.9, -1.74994}, {9., -1.70605}, {9.1, -1.65247}, {9.2, -1.58807},
{9.3, -1.51142}, {9.4, -1.42068}, {9.5, -1.31343}, {9.6, -1.18628},
{9.7, -1.03413}, {9.8, -0.84843}, {9.9, -0.611712}, {10., -0.265165}}
]

\psline[linecolor=gray]{->}(-1,0)(11,0)
\psline[linecolor=gray]{->}(0,-2.1)(0,4)
\multirput(2.525,-0.15)(2.525,0){4}{\psline[linecolor=gray](0,0)(0,0.3)}
\multirput(-0.15,-2)(0,1){6}{\psline[linecolor=gray](0,0)(0.3,0)}

\dataplot[linecolor=red,linewidth=0.8pt,plotstyle=line]{\mydata}
\dataplot[linecolor=red,linewidth=0.8pt,plotstyle=line]{\mydatab}
\dataplot[linecolor=orange,linewidth=0.8pt,plotstyle=line]{\mydatac}
\dataplot[linecolor=black,linewidth=0.8pt,plotstyle=dots]{\mydatac}

\rput(11.4,-1.6){$\log\left( \spr{h/k}{m} \right)$}
\rput(8.8,0.5){$\log\left( c(h)\right)$}
\rput(5.5,1.7){$\log\left( c(h) 2^{  k/(2 h)-m} \right)$}
\rput(4.5,3.6){$\log\left( c(h) \left( \frac k{ m h} \right)^m \right)$}

\rput(-0.7,-2){$-20$}
\rput(-0.7,-1){$-10$}
\rput(-0.5,1){$10$}
\rput(-0.5,2){$20$}
\rput(-0.5,3){$30$}
\rput(10.9,-0.5){$m$}
\rput(5.05,-0.8){$\frac{k}{2h}$}
\rput(2.525,-0.8){$\frac{k}{4h}$}
\rput(7.575,-0.8){$\frac{3k}{4h}$}
\rput(10.1,-0.8){$\frac{k}{h}$}

\end{pspicture}
\caption{Bounds for $\spr{h/k}{m}$ with $1\lqs m < k/h$ and $h=2$, $k=201$}\label{prods}
\end{center}
\end{figure}


Define $g(x):=x \log(1/x)$. Then $g$ is an increasing function on $[0,1/e]$ with $g(0)=0$ and $g(1/e)=1/e$.

\begin{prop} \label{gprop}
Let $W>0$. For  $\delta$ satisfying $0<\delta \lqs 1/e$ and $\delta \log(1/\delta) \lqs W$ we have
\begin{equation*}
    \spr{h/k}{m}   \lqs c(h) \exp\left( \frac{kW}h \right) \quad \text{for} \quad 0 \lqs \frac {mh}k \lqs \delta
    \quad \text{and} \quad \frac 12 - \delta \lqs \frac {mh}k <1.
\end{equation*}
\end{prop}
\begin{proof}
The result is true for $m=0$ since $\spr{h/k}{0}=1<c(h)$.
For $0< \frac {mh}k \lqs \delta$, starting with  \eqref{vm1},
\begin{align*}
    \spr{h/k}{m}  & \lqs c(h) \left( \frac k{ m h} \right)^m\\
    & = c(h) \exp\left( \frac k{h} g\left( \frac {mh}k\right) \right)\\
    & \lqs c(h) \exp\left( \frac k{h} g\left( \delta \right) \right)\\
    & \lqs c(h) \exp\left( \frac k{h} W \right)
\end{align*}
since  $g(\delta) \lqs W$.
For $\frac 12 \lqs \frac {mh}k <1 $ we have already seen in \eqref{pbndq} that $\spr{h/k}{m}  \lqs c(h)$.
For $\frac 12 - \delta \lqs \frac {mh}k \lqs \frac 12$ we have, starting with \eqref{vm2},
\begin{align*}
    \spr{h/k}{m}  & \lqs c(h) \exp\left(\log 2 \frac k{2\pi  h}\left( \pi - \frac {2\pi mh}k\right) \right)\\
    & \lqs c(h) \exp\left(\log 2 \frac k{2\pi  h} \cdot 2\pi \delta  \right)\\
    & \lqs c(h) \exp\left( \frac k{h} g\left( \delta \right) \right)\\
    & \lqs c(h) \exp\left( \frac k{h} W \right).  \qedhere
\end{align*}
\end{proof}

\begin{lemma} \label{tn}
For $1 \lqs L$ and $1 \lqs m \lqs k/2h$ we have
\begin{equation*}
    |T_L(m,h/k)| \lqs \frac{\pi^3}{2} \left(\frac{2L-1}{2\pi e m}\right)^{2L-1}.
\end{equation*}
\end{lemma}
\begin{proof}
Combine the bounds \eqref{slb}, \eqref{rlb} and use
\begin{equation*}
\frac{h/k}{ 1-m h/k} = \frac{1}{ m(k/(mh)-1)} \lqs \frac{1}{ m}. \qedhere
\end{equation*}

\end{proof}

\subsection{Controlling the error term} \label{sec-err}

We have seen with Proposition \ref{prrh}  that
for $1 \lqs m <k/h$
\begin{multline} \label{pimv2b}
    \spr{h/k}{m} = \left(\frac{h}{2k \sin(\pi m h/k)}\right)^{1/2}  \exp\left(\frac k{2\pi h} \cl\bigl(2\pi  m h/k \bigr) \right) \\
      \quad \times \exp\left(-\sum_{\ell=1}^{L-1} \frac{B_{2\ell}}{(2\ell)!} \left( \frac{\pi h}{k}\right)^{2\ell-1} \cot^{(2\ell-2)}\left(\frac{\pi m h}{k}\right)\right) \exp\left(-T_{L}(m,h/k)\right).
  \end{multline}
For $L$ large, what is the effect of removing the factor $\exp\left(-T_{L}(m,h/k)\right)$ above? Our bound on $T_{L}(m,h/k)$ from Lemma \ref{tn} is
poor for $m$ small but gets much better when $m>(2L-1)/(2\pi e)$. Proposition \ref{gprop}  proves that $\spr{h/k}{m}$ is small for small $m$,
 so  we may assume $m > \delta k/h$ for a fixed $\delta>0$. As $m$ increases we have $\spr{h/k}{m}$ getting bigger and the bound from Lemma \ref{tn} getting smaller.
Our goal in this subsection is to choose an integer $L$ (depending on a large parameter $s$ where $0<h<k \lqs s$) so that these competing bounds produce a small enough error.

We first give some preliminary results that shall be required in Proposition \ref{dela} below.
Given  $\Delta>0$, we will need real numbers $r \in [0,1]$ that satisfy both the inequalities
\begin{align}\label{ri1}
    \frac{1}{e^{r+1}}\left(1+r \log \frac r{2\pi}\right) & \lqs  \Delta \log \frac 1{\Delta},\\
    \frac{r}{e^{r+1}} & \lqs  2 \pi e \Delta. \label{ri2}
\end{align}
Note that the left side of \eqref{ri1} is a decreasing continuous function of $r \in [0,1]$, decreasing from the value $1/e$ to become negative at $r=1$. The left side of \eqref{ri2} is increasing and continuous from $0$ at $r=0$ to $1/e^2$ at $r=1$. Therefore,
there exist $r_1=r_1(\Delta)$, $r_2=r_2(\Delta)$ so that
\begin{align}
    \frac{1}{e^{r_1+1}}\left(1+r_1 \log \frac {r_1}{2\pi}\right) & =  \Delta \log \frac 1{\Delta}, \label{rj1}\\
    \frac{r_2}{e^{r_2+1}} & =  2 \pi e \Delta \label{rj2}
\end{align}
where we assume
$$\Delta \lqs \frac 1{2\pi e^3} \approx 0.0079$$
so that \eqref{rj2} has a solution.
If $r_1 \lqs r_2$ then the set of all $r \in [0,1]$ satisfying both \eqref{ri1} and \eqref{ri2} is the interval $[r_1,r_2]$. Calculations displayed in Table \ref{tbl} show that $r_1(\Delta)<r_2(\Delta)$ for $0.0048 \lqs \Delta \lqs 0.0079$.

\begin{table}[h]
\begin{center}
\begin{tabular}{c | c | c | c | c}
$\Delta$ & $\Delta \log 1/\Delta$ & $r_1(\Delta)$ & $r_2(\Delta)$ & $R_\Delta$ \\ \hline
0.0079 & 0.0382 & 0.276 & 0.924 & 51.9  \\
0.007 & 0.0347 & 0.282 & 0.581 & 72.6 \\
0.006 & 0.0307 & 0.288 & 0.427 & 130.7 \\
0.005 & 0.0265 & 0.295 & 0.320 & 665.2  \\
0.00477 & 0.0255 & 0.297 & 0.298 & 11701.6
\end{tabular}
\end{center}
\caption{Some values for $\Delta$ and related quantities.} \label{tbl}
\end{table}

Suppose $u>0$ and $L-1/2=\pi e \Delta \cdot u$. If $0.0048 \lqs \Delta \lqs 0.0079$ then we can find $r$ satisfying both \eqref{ri1} and \eqref{ri2} such that
\begin{equation}\label{2lr}
    2L-1= \frac{r}{e^{r+1}} u
\end{equation}
since we may take $r=r_2(\Delta)$. We next show that \eqref{2lr} still has a solution $r\in [r_1,r_2]$ if we  replace $L$ by the integer $\lfloor \pi e \Delta \cdot u \rfloor$. All this requires is that $r_1$ and $r_2$ are far enough apart.
Define $R_\Delta$ as
\begin{equation*}
    R_\Delta :=3/\left(\frac{r_2}{e^{r_2+1}} - \frac{r_1}{e^{r_1+1}} \right) \qquad (r_1< r_2).
\end{equation*}

\begin{lemma} \label{kdel}
Given $\Delta$ satisfying $0.0048 \lqs \Delta \lqs 0.0079$, suppose $u \in \R$ satisfies $u \gqs R_\Delta$. Set $L:=\lfloor \pi e \Delta \cdot u \rfloor$. Then
\begin{equation}\label{errx}
2L-1=\frac{r}{e^{r+1}} u
\end{equation}
for $r$ satisfying both \eqref{ri1} and \eqref{ri2}.
\end{lemma}
\begin{proof}
Since, as we have seen, $r/e^{r+1}$ increases from $0$ to $1/e^2$ with $r\in [0,1]$, we may use \eqref{errx} to define $r$.
From the definitions of $L$ and $r_2$ we obtain
$$
u \frac{r_2}{2e^{r_2+1}} -1 < L \lqs u \frac{r_2}{2e^{r_2+1}}
$$
and hence
\begin{equation}\label{hen}
u \frac{r_2}{e^{r_2+1}} -3 < 2L-1 \lqs u \frac{r_2}{e^{r_2+1}}-1.
\end{equation}
The right inequality in \eqref{hen} implies that $r<r_2$. Also $u \gqs R_\Delta$ implies
\begin{equation*}
    u\frac{r_1}{e^{r_1+1}} \lqs u \frac{r_2}{e^{r_2+1}} -3
\end{equation*}
so that the left inequality in \eqref{hen} implies that $r_1<r$. Then $r \in [r_1,r_2]$ implies $r$ satisfies \eqref{ri1} and \eqref{ri2} as required.
\end{proof}

\begin{prop} \label{dela}
Suppose   $\Delta$ and $W$ satisfy $0.0048 \lqs \Delta \lqs 0.0079$ and $\Delta \log 1/\Delta \lqs W$.
For the integers $h$, $k$, $s$ and $m$ we require
\begin{equation*}
    0<h<k \lqs s, \quad  R_\Delta \lqs s/h, \quad \Delta s/h \lqs m \lqs k/(2h).
\end{equation*}
 Then for $L:=\lfloor \pi e \Delta \cdot s/h \rfloor$ we have
\begin{align}
\left| \spr{h/k}{m}  T_L(m,h/k) \right| & \lqs (\pi^3/2) c(h) \cdot e^{sW/h}, \label{slma}\\
    \left| T_L(m,h/k) \right| & \lqs \pi^3/2. \label{slmb}
\end{align}
\end{prop}
\begin{proof}
We write  $2L-1=  \beta s/h$ for some $\beta =r/e^{r+1}$
and $r$ satisfying both \eqref{ri1} and \eqref{ri2} by Lemma \ref{kdel}.
With the bounds from Lemmas \ref{vm} and \ref{tn} we have
\begin{equation}\label{pri}
\left| \spr{h/k}{m}  T_L(m,h/k) \right| \lqs  c(h) \left( \frac k{mh} \right)^m \frac{\pi^3}{2} \left( \frac{\beta s}{2\pi e m h} \right)^{\beta s/h}.
\end{equation}
Taking the $h/s$ power of both sides, we see that \eqref{slma} follows if we can establish that
$$
\left( \frac s{m h} \right)^{m h/s} \left( \frac{\beta s}{2\pi e m h} \right)^{\beta }\lqs e^{W}
$$
or equivalently, for $t=s/(m h)$ and $2 < t \lqs 1/\Delta$, that
\begin{equation}\label{tat}
t^{\beta+1/t} \left( \frac{\beta}{2\pi e } \right)^{\beta} \lqs e^{W}.
\end{equation}
We see that $t^{\beta+1/t}$ has maxima on the interval $(0,1/\Delta]$ at $t=1/\Delta$ and $t=e^{r+1}$.
 To prove \eqref{tat} we therefore just need to
 verify it at $t=1/\Delta$ and $t=e^{r+1}$.

 Since $(1/\Delta)^\Delta \lqs e^W$ by the definition of $\Delta$ and
\begin{equation}\label{1}
 \frac{\beta}{2\pi e \Delta} \lqs 1
\end{equation}
by \eqref{ri2} we see that \eqref{tat} is true for $t=1/\Delta$.
Next, a short calculation shows that, for $t=e^{r+1}$,
$$
\log \left( t^{\beta+1/t} \left( \frac{\beta}{2\pi e } \right)^{\beta}\right) = \frac{1}{e^{r+1}}\left(1+r \log \frac {r}{2\pi}\right).
$$
Therefore \eqref{ri1} implies that \eqref{tat} is true for $t=e^{r+1}$. We have proved  \eqref{slma}.

We also have
$$
\left| T_L(m,h/k) \right| \lqs \frac{\pi^3}{2}\left( \frac{\beta s}{2\pi e m h} \right)^{\beta s/h} \lqs \frac{\pi^3}{2}\left( \frac{\beta }{2\pi e \Delta}  \right)^{\beta s/h} \lqs \frac{\pi^3}{2}
$$
using \eqref{1}. This proves the inequality  \eqref{slmb}.
\end{proof}

As a numerical check of Proposition \ref{dela}, take for example  $\Delta=0.006$, $W=0.031$, $s=500$ and $h=1$. Then $L=\lfloor \pi e \Delta s/h\rfloor = 25$ and we require
$3\lqs m \lqs k/2\lqs 500/2$. For  these $m$ and $k$ we find
\begin{align}
\max_{m,k}\left| \spr{1/k}{m}  T_{25}(m,1/k) \right| \approx 144.7 \qquad &  < \qquad 8.54\times 10^{7} \approx (\pi^3/2) c(1) \cdot e^{500W/1}, \label{cloo}\\
    \max_{m,k}\left| T_{25}(m,1/k) \right| \approx 0.002 \qquad &  < \qquad 15.5 \approx \pi^3/2. \notag
\end{align}
(The maximum of the bound on the right of \eqref{pri} is $8.22\times 10^{6}$, closer to the right side of \eqref{cloo}.)
Similarly, with the same  $\Delta$, $W$ and $s$, take $h=3$ so that $L=8$ and $1\lqs m \lqs k/6\lqs 500/6$. For  these $m$ and $k$ we find
\begin{align*}
\max_{m,k}\left| \spr{3/k}{m}  T_{8}(m,3/k) \right| \approx 0.133 \qquad &  < \qquad 1.4\times 10^{4} \approx (\pi^3/2) c(3) \cdot e^{500W/3}, \\
    \max_{m,k}\left| T_{8}(m,3/k) \right| \approx 0.005 \qquad &  < \qquad 15.5 \approx \pi^3/2.
\end{align*}

Since the bounds for $T_L(m,h/k)$ used in the proof of Proposition \ref{dela} (coming from  Lemma \ref{tn})
are independent of $h/k$, a short verification shows the following generalization  of Proposition \ref{dela}, needed in \cite{OS2}.

\begin{cor}\label{delacor}
Let $W, \Delta, s, h,k,m$ and $L$ be as in Proposition \ref{dela}. Suppose also that $0< u/v\lqs h/k$. Then
\begin{align}
\left| \spr{h/k}{m}  T_L(m,u/v) \right| & \lqs (\pi^3/2) c(h) \cdot e^{sW/h}, \label{slmax}\\
    \left| T_L(m,u/v) \right| & \lqs \pi^3/2. \label{slmbx}
\end{align}
\end{cor}

\begin{prop}\label{hard}
For $W, \Delta, s, h,k,m$ and $L$ as in Proposition \ref{dela} we have
\begin{multline} \label{pimv9}
    \spr{h/k}{m}  = \left(\frac{h}{2k \sin(\pi m h/k)}\right)^{1/2}  \exp\left(\frac k{2\pi h} \cl\bigl(2\pi  m h/k \bigr) \right) \\
      \quad \times \exp\left(-\sum_{\ell=1}^{L-1} \frac{B_{2\ell}}{(2\ell)!} \left( \frac{\pi h}{k}\right)^{2\ell-1} \cot^{(2\ell-2)}\left(\frac{\pi m h}{k}\right)\right) +O\left(e^{sW/h}\right)
  \end{multline}
for an  implied constant depending only on $h$.
\end{prop}
\begin{proof}
With \eqref{pimv2b} we see that \eqref{pimv9} follows if we can prove
\begin{equation*}
    \spr{h/k}{m}  = \spr{h/k}{m}  \exp\bigl(T_{L}(m,h/k)\bigr) +O\left(e^{sW/h}\right).
\end{equation*}
For any $\kappa>0$, say,
 note that the  simple inequality
 \begin{equation}\label{siine}
    |e^x-1| \lqs |x| \frac{e^\kappa -1}{\kappa} \quad \text{ for } \quad x \in (-\infty,\kappa]
 \end{equation}
 follows from the fact that $(e^x-1)/x$ is positive and increasing.
Then, using Proposition \ref{dela},
\begin{align*}
    \left| \spr{h/k}{m}  \left[\exp\bigl(T_{L}(m,h/k)\bigr)-1\right]\right| & \lqs \left| \spr{h/k}{m}  T_{L}(m,h/k)\right| \frac{e^{\pi^3/2}-1}{\pi^3/2}\\
    & \lqs \left(e^{\pi^3/2}-1\right) c(h) e^{sW/h}. \qedhere
\end{align*}
\end{proof}

\begin{remark} \label{remt}
{\rm The requirement $\Delta s/h \lqs m$ in Propositions \ref{dela} and \ref{hard} is essential since the bound we are using
from Lemma \ref{tn},
\begin{equation*}
    |T_L(m,h/k)| \lqs \frac{\pi^3}{2} \left(\frac{2L-1}{2\pi e m}\right)^{2L-1},
\end{equation*}
worsens dramatically for
\begin{equation*}
    2L-1 \approx 2\pi e \Delta s/h > 2\pi e m.
\end{equation*}
See inequality \eqref{1}.}
\end{remark}

\section{Expressing $\mathcal A_1(N,\sigma)$ as an integral} \label{sec-a1n}

\subsection{First results for $\mathcal A_1(N,\sigma)$} \label{31}

Rewrite \eqref{cj1} more in terms of $N/k$ as
\begin{multline}\label{cj1aa}
\mathcal A_1(N,\sigma) = \Im \sum_{ \frac{N}{2}  <k \lqs N} \frac{2(-1)^{k}}{k^2}
\exp\left( N\left[ \frac{i\pi}{2}\left(-\frac{N}{k}+3 \right)\right]\right)\\
\times\exp\left( -\frac{i\pi}{2}\frac{N}{k}\right)
\exp\left( \frac 1N\left[ 2 i\pi \sigma \frac{N}{k}\right]\right)
\spr{1/k}{N-k}
\end{multline}
and define
\begin{equation}\label{grz}
g_\ell(z):=-\frac{B_{2\ell}}{(2\ell)!} \left( \pi z\right)^{2\ell-1} \cot^{(2\ell-2)}\left(\pi z\right).
\end{equation}

\begin{theorem}[Sine product approximation] \label{cplx}
Fix $W>0$. Let $\Delta$ be in the range $0.0048 \lqs \Delta \lqs 0.0079$ and set $\alpha := \Delta \pi e$.  Suppose $\delta$ and $\delta'$ satisfy
\begin{equation}\label{mw}
    \frac{\Delta}{1-\Delta} < \delta \lqs \frac{1}{e}, \ 0<\delta' \lqs \frac{1}{e} \quad \text{ and  } \quad \delta \log 1/\delta, \ \ \delta' \log 1/\delta' \lqs W.
\end{equation}
Then for all $N \gqs R_\Delta$  we have
\begin{equation}\label{ott1}
    \spr{1/k}{N-k} = O\left(e^{WN}\right) \quad \text{ for } \quad  \frac Nk \in [1,1+\delta] \cup  [3/2 -\delta' ,2)
\end{equation}
and
\begin{multline} \label{ott2}
    \spr{1/k}{N-k} =  \left(\frac{N/k}{2N \sin (\pi (N/k-1))}\right)^{1/2}  \exp\left(\frac N{2\pi N/k} \cl\bigl(2\pi  N/k \bigr) \right) \\
      \quad \times
      \exp\left(\sum_{\ell=1}^{L-1} \frac{g_{\ell}(N/k)}{N^{2\ell-1}} \right) +O\left(e^{WN}\right)
        \quad \text{ for } \quad \frac Nk \in  (1+\delta, 3/2 -\delta')
  \end{multline}
with $L=\lfloor \alpha \cdot N\rfloor$. The implied constants in \eqref{ott1}, \eqref{ott2} are absolute.
\end{theorem}
\begin{proof}
 The bound \eqref{ott1} follows directly from Proposition \ref{gprop} with $m=N-k$ and $h=1$. Next, in Proposition \ref{hard}, we set  $s=N$ and again $m=N-k$ and $h=1$. The condition $\Delta \log 1/\Delta \lqs W$ we need for that result follows from \eqref{mw} since $\Delta < \Delta/(1-\Delta)<\delta$ and $\Delta \log 1/\Delta$ is increasing. We also see from Table \ref{tbl} in Section \ref{sec-err} that our choice of $\Delta \in [0.0048, 0.0079]$ ensures $R_\Delta$ is finite. The condition on $m$ in Proposition \ref{hard} is equivalent to
$$
1+\frac{\Delta}{1-\Delta} \lqs \frac Nk \lqs \frac 32.
$$
So \eqref{ott2} follows from Proposition \ref{hard} if $\Delta/(1-\Delta)<\delta$, as we assumed.
\end{proof}

We will later fix some of the parameters in Theorem \ref{cplx}:

\begin{cor} \label{cplx2}
Let $W=0.05$ and $\alpha = 0.006 \pi e  \approx 0.0512$. Then for all $N \gqs 131$  we have that
\eqref{ott1}, \eqref{ott2} hold when $L=\lfloor \alpha \cdot N\rfloor$ and
\begin{equation}\label{dels}
   0.0061 \lqs \delta, \ \delta' \lqs 0.01.
\end{equation}
\end{cor}

It follows  from  \eqref{cj1aa} and Theorem \ref{cplx} that
\begin{multline} \label{lucy}
    \mathcal A_1(N,\sigma)   =  \Im \sum_{ k \ : \ \frac Nk \in  (1+\delta, \frac 32 -\delta')}  \frac{2(-1)^{k}}{k^2}
    \exp\left( N\left[\frac{\cl\bigl(2\pi  N/k \bigr)}{2\pi N/k} + \frac{i\pi}{2}\left(-\frac{N}{k}+3 \right)\right]\right) \\
     \quad \times \left(\frac{N/k}{2N \sin (\pi (N/k-1))}\right)^{1/2}  \exp\left( -\frac{i\pi}{2}\frac{N}{k}\right)
      \exp\left(\frac 1N\left[ 2 i\pi \sigma \frac{N}{k}\right] +\sum_{\ell=1}^{L-1} \frac{g_{\ell}(N/k)}{N^{2\ell-1}} \right)+ O(e^{WN}).
\end{multline}
To describe this concisely we use the notation, with $z \in (1,2)$ to begin,
\begin{align}
r(z) & := \frac{\cl(2\pi z)}{2\pi z} + \frac{\pi i}{2}(-z+3), \label{rz}\\
q(z) & := \left( \frac{z }{2\sin(\pi(z -1))}\right)^{1/2} \exp(-i\pi z/2 ), \label{qz}\\
v(z;N,\sigma)  & := \frac{2\pi i \sigma z}N +\sum_{\ell=1}^{L-1} \frac{g_{\ell}(z)}{N^{2\ell-1}}, \qquad (L=\lfloor \alpha \cdot N \rfloor). \label{vz}
\end{align}
(If we need to show the dependence of $v(z;N,\sigma)$ on $\alpha$  we may write $v(z;N,\sigma,\alpha)$.)
For $z=z(N,k):=N/k$, set
\begin{equation*}
    \mathcal A_2(N,\sigma) := \frac 2{N^{1/2}} \Im  \sum_{ k \ : \ z \in  (1+\delta, \frac 32 -\delta')} \frac{(-1)^{k}}{k^2}  \exp\bigl(N  \cdot r \left(z \right) \bigr) q\left(z \right) \exp\bigl(v \left(z;N,\sigma \right) \bigr)
\end{equation*}
and \eqref{lucy} now implies that for $\sigma \in \Z$ and an absolute implied constant
\begin{equation}\label{a1n}
\mathcal A_2(N,\sigma) = \mathcal A_1(N,\sigma) +O(e^{WN}).
\end{equation}

We wish to replace the  sum defining $\mathcal A_2(N,\sigma)$ with an integral.
Our goal in the rest of this section is to prove the following, (with $\alpha = 0.006 \pi e$ as in Corollary \ref{cplx2}).
\begin{theorem} \label{anmain}
For $W = 0.05$ and an  implied constant depending only on $\sigma$, we have
\begin{equation*}
\mathcal A_2(N,\sigma) = \frac{2}{N^{3/2}} \Im \int_{1.01}^{1.49} \exp\bigl(N \left[ r \left(z \right) - \pi i /z\right] \bigr) q(z)  \exp\bigl(v(z;N,\sigma)\bigr) \, dz + O(e^{WN}).
\end{equation*}
\end{theorem}

In the next subsections we develop properties of $r(z)$, $q(z)$ and $v(z;N,\sigma)$ considered as functions of $z= x+ i y$ in a vertical strip in $\C$.

\subsection{Properties of $v(z;N,\sigma)$ and $q(z)$}

\begin{lemma} \label{useful}
Suppose $\beta$, $\gamma >0$ satisfy $ \beta \gamma <1$. Then given any $N$, $d \gqs 1$ we have
\begin{equation*}
    \sum_{j=d}^{\lfloor N \cdot \beta \rfloor} \left( \frac{\gamma j}{N}\right)^j = O\left( \frac{1}{N^d}\right)
\end{equation*}
for an implied constant depending only on $\beta$, $\gamma$ and $d$.
\end{lemma}
\begin{proof}
Note that, as $j$ varies, $\left( \frac{\gamma j}{N}\right)^j$ decreases for $1\lqs j \lqs N/(e \gamma)$ and increases for $ j \gqs N/(e \gamma)$. Consequently, for $d+1 \lqs j \lqs \lfloor N \cdot \beta \rfloor$,
\begin{align*}
    \left( \frac{\gamma j}{N}\right)^j & \lqs \max \left\{ \left( \frac{\gamma (d+1)}{N}\right)^{(d+1)}, \left( \frac{\gamma N \beta}{N}\right)^{\lfloor N \cdot \beta \rfloor}\right\}\\
   & \ll \max \left\{ N^{-d-1}, \left(  \beta \gamma \right)^{\lfloor N \cdot \beta \rfloor} \right\}\\
   & \ll  N^{-d-1}
\end{align*}
since $\beta \gamma <1$ implies $\left( \beta \gamma  \right)^{\lfloor N \cdot \beta \rfloor}$ decays exponentially with $N$.
Thus
\begin{equation*}
    \sum_{j=d}^{\lfloor N \cdot \beta \rfloor} \left( \frac{\gamma j}{N}\right)^j  \ll \left( \frac{\gamma d}{N}\right)^d + \left(\lfloor N \cdot \beta \rfloor -d \right) \frac{1}{N^{d+1}} \ll \frac{1}{N^{d}}
\end{equation*}
as we wanted.
\end{proof}

\begin{prop} \label{add}
Suppose $1/2 \lqs \Re(z) \lqs 3/2$ and $|z-1| \gqs \varepsilon >0$. Also assume that
\begin{equation} \label{asm}
    \max \Bigl\{1+\frac{1}{\varepsilon}, \ 16\Bigr\} \ < \frac{\pi e}{\alpha}.
\end{equation}
Then, for an implied constant depending only on $\varepsilon$, $\alpha$ and $d$,
\begin{equation} \label{we}
    \sum_{\ell=d}^{L-1} \frac{g_{\ell}(z)}{N^{2\ell-1}} \ll \frac{1}{N^{2d-1}}e^{-\pi|y|} \qquad (d \gqs 2, \ L=\lfloor \alpha \cdot N \rfloor).
\end{equation}
\end{prop}
\begin{proof}
We have
\begin{align*}
    \frac{ g_\ell(z) }{N^{2\ell -1}} & \ll \frac{|B_{2\ell}|}{(2\ell)!}\left( \frac{\pi |z|}{ N}\right)^{2\ell -1} \left| \cot^{(2\ell -2)}(\pi z)\right|\\
    & \ll \left( \frac{|z|}{2 N}\right)^{2\ell -1} \left| \cot^{(2\ell -2)}(\pi z)\right|
\end{align*}
using \eqref{bk2}.
Suppose $\ell \gqs 2$ and write $z=1+w$. Then \eqref{cotthma} from Theorem \ref{cotder} and \eqref{k!} show
\begin{equation}
\frac{ g_\ell(z) }{N^{2\ell -1}} \ll \left( \frac{|z|(2\ell-1)}{2 \pi e N}\right)^{2\ell -1} \left( \frac{1}{|w|^{2\ell-1}}+ 8^{2\ell-1}\right)e^{-\pi|y|}. \label{sg1}
\end{equation}
Since $|z|/|w| \lqs (1+|w|)/|w| \lqs 1+1/\varepsilon$ and also $8|z| <16$ if $|y|\lqs 1$, we see that
\begin{equation*}
    e^{\pi|y|} \sum_{\ell=d}^{L-1} \frac{g_{\ell}(z)}{N^{2\ell-1}}  \ll \sum_{\ell=2d-1}^{2L}\left( \frac{1+\frac 1 \varepsilon}{2 \pi e } \frac \ell N\right)^{\ell}+ \sum_{\ell=2d-1}^{2L}\left( \frac{16}{2 \pi e } \frac \ell N\right)^{\ell}
\end{equation*}
and the proposition follows in this case with an application of Lemma \ref{useful}, using assumption \eqref{asm}. The $|y| \gqs 1$ case is similar, employing \eqref{cotthmb} from Theorem \ref{cotder}.
\end{proof}

It is now convenient to fix the choice of constants  in Corollary \ref{cplx2} for the rest of this section.
We note that condition \eqref{asm} in Proposition \ref{add} is met for $\varepsilon = 0.0061$ and  $\alpha = 0.006 \pi e$:
\begin{equation} \label{stxr}
    \max \left\{1+\frac{1}{\varepsilon}, 16\right\} \approx \max \left\{164.9, \ 16\right\} = 164.9 \quad < \quad 166.\overline{6} = \frac{\pi e}{\alpha}.
\end{equation}
We have therefore shown the next result.
\begin{cor} \label{acdc}
With $\delta, \delta' \in [0.0061, 0.01]$ and  $z \in \C$ such that $1+\delta \lqs \Re(z) \lqs 3/2-\delta'$  we have
\begin{equation*}
    v(z;N,\sigma) =  \frac{2\pi i \sigma z}N +\sum_{\ell=1}^{d-1} \frac{g_{\ell}(z)}{N^{2\ell-1}} + O\left( \frac 1{N^{2d-1}}\right)
\end{equation*}
 for $2 \lqs d\lqs L=\lfloor 0.006 \pi e \cdot N \rfloor$ and an implied constant depending only on  $d$.
\end{cor}

Inequality \eqref{stxr} holds  because the bounds obtained for $g_\ell(z)$ in the above proof of Proposition \ref{add} are similar to those obtained for $T_L(m,h/k)$ in Lemma \ref{tn}. In particular, for $z=N/k$ and $m=N-k$, \eqref{sg1} implies
\begin{equation*}
    \frac{ g_\ell(N/k) }{N^{2\ell -1}}  \ll \left( \frac{2\ell-1}{2 \pi e m}\right)^{2\ell -1} + \left( \frac{16(2\ell-1)}{2 \pi e N}\right)^{2\ell -1}.
\end{equation*}
We see that condition \eqref{asm} is equivalent to
\begin{equation} \label{kja}
     \frac{2\ell-1}{2 \pi e m}, \quad \frac{16(2\ell-1)}{2 \pi e N} < 1
\end{equation}
when $z=N/k$. The requirement $\Delta s/h \lqs m$, see Remark \ref{remt}, is in place in Theorem \ref{cplx} with $s=N$ and $h=1$ so that $\Delta N \lqs m$. Hence
\begin{equation*}
    2L-1 \approx 2\pi e \Delta N < 2 \pi e m
\end{equation*}
and
\begin{equation*}
    \frac{2\ell-1}{2 \pi e m} < \frac{2L-1}{2 \pi e m}<1, \qquad \frac{16(2\ell-1)}{2 \pi e N}<\frac{16(2L-1)}{2 \pi e N}<16\Delta.
\end{equation*}
Recall that $\Delta \lqs 0.0079$ so that $16\Delta<1$. Therefore $\Delta N \lqs m$ and $16\Delta<1$ imply \eqref{kja} and \eqref{stxr}.


\begin{prop} \label{qh}
The functions  $q(z)$ and $v(z;N,\sigma)$ are holomorphic in $z$ for $1< \Re(z)<3/2$. In  the box with $1+\delta \lqs \Re(z) \lqs 3/2-\delta'$ and $-1\lqs \Im(z) \lqs 1$,
\begin{equation*} \label{efb3}
    q(z), \quad \exp \bigl(v(z;N,\sigma)\bigr) \ll 1
\end{equation*}
for an  implied constant depending only on $\sigma \in \R$.
\end{prop}
\begin{proof}
Check that for $w\in \C$,
\begin{equation*}
    -\pi/2< \arg \bigl(\sin(\pi w) \bigr) <\pi/2 \quad \text{for} \quad 0<\Re(w)<1.
\end{equation*}
Consequently, $-\pi<\arg\bigl( z/\sin(\pi (z-1))\bigr) <\pi$ for $1< \Re(z)<3/2$ and so $q(z)$ is holomorphic in this strip. Also
$v(z;N,\sigma)$ is holomorphic here since the only poles of $g_\ell(z)$ are at $z\in \Z$.

Finally, $q(z)$ is clearly bounded on the compact box, as is $\exp \bigl(v(z;N,\sigma)\bigr)$ by Corollary \ref{acdc}.
\end{proof}

\subsection{Properties of $r(z)$}

We defined $r(z)$ in \eqref{rz} for $1<z<2$. Use \eqref{dli} to extend it as
\begin{equation*}
    r(z) = \frac{\li(e^{2\pi i z})}{2\pi i z} + \frac{13\pi i}{12 z},
\end{equation*}
now holomorphic in the strip $1<\Re(z)<2$.
Adding a parameter $j$, we get
\begin{equation}\label{idru}
    r \left(z \right) +\frac{\pi i j}{z} =  \frac{1}{2\pi i z} \Bigl[-\li(1)+ \li(e^{2\pi i z}) -2\pi^2(j+1) \Bigr].
\end{equation}
From \eqref{ss} we obtain the identity
$$
\frac{\li(e^{2\pi i z})}{2\pi i z} = \frac{-\li(e^{-2\pi i z})}{2\pi i z} -\pi i(z-3)-\frac{13 \pi i}{6 z} \qquad (1<\Re(z)<2)
$$
and substituting in \eqref{idru} produces the alternate expression, valid only for $1<\Re(z)<2$,
\begin{equation}\label{om}
 r(z)+\frac{\pi i j}{z}=-\pi i(z-3)+\frac{1}{2\pi i z} \Bigl[\li(1)- \li(e^{-2\pi i z}) -2\pi^2(j-1) \Bigr].
\end{equation}

\begin{lemma} \label{dil1}
Consider $\Im(\li(e^{2\pi i z}))$  as a function of $y \in \R$. It is positive and decreasing for fixed $x\in (0,1/2)$ and negative and increasing for fixed $x\in (1/2,1)$.
\end{lemma}
\begin{proof}
We have
\begin{equation*}
    \frac{d}{dy} \Im(\li(e^{2\pi i z})) = \Im(\frac{d}{dy} \li(e^{2\pi i z})) = 2\pi \arg (1-e^{2\pi i z}).
\end{equation*}
Clearly this derivative is negative for $x\in (0,1/2)$ and positive for $x\in (1/2,1)$. Also, we have
\begin{equation*}
    \lim_{y \to \infty}\Im(\li(e^{2\pi i z})) = \Im(\li(0)) = 0
\end{equation*}
implying the function decreases or increases to zero.
\end{proof}

\begin{lemma} \label{dil2}
For $y\gqs 0$ we have $|\li(e^{2\pi i z})| \lqs \li(1)$.
\end{lemma}
\begin{proof}
With $y\gqs 0$ we have $|e^{2\pi i z}|\lqs 1$ and
\begin{equation*}
    |\li(e^{2\pi i z})| = \left| \sum_{m=1}^\infty \frac{e^{2\pi i m z}}{m^2}\right| \lqs \sum_{m=1}^\infty \frac{1}{m^2} = \li(1). \qedhere
\end{equation*}
\end{proof}

\begin{theorem} \label{rzj}
The function $r(z)$ is holomorphic for $1< \Re(z)<3/2$. In this strip, for $j\in \R$,
\begin{alignat}{2}
    \Re\left(r(z) +\frac{\pi i j}{z}\right) & \lqs \frac{1}{2\pi |z|^2} \left(x \cl(2\pi x) +\pi^2 |y|\left[\frac 13 +2(j+1) \right]  \right) \qquad & & (y\gqs 0) \label{ef1}\\
   \Re\left(r(z) +\frac{\pi i j}{z}\right) & \lqs \frac{1}{2\pi |z|^2} \left(x \cl(2\pi x) +\pi^2 |y|\left[\frac 13 -2j \right] \right) \qquad & & (y\lqs 0). \label{ef2}
\end{alignat}
\end{theorem}

\begin{proof} With \eqref{idru}, we see that $r(z)$ is actually holomorphic for all $z\in \C$ away from the vertical branch cuts $(-i\infty,n]$, $n\in \Z$.
Equation \eqref{idru} implies
\begin{equation}\label{ee0}
\Re\left(r \left(z \right) +\frac{\pi i j}{z}\right) = y\left(\frac{\li(1)-\Re\left(\li (e^{2\pi i z})\right)+ 2\pi^2(j+1) }{2\pi|z|^2} \right)  + \frac{x \ \Im\left(\li (e^{2\pi i z})\right)}{2\pi|z|^2}.
\end{equation}
For $y \gqs 0$ we have
\begin{equation} \label{ee1}
    \Im\left(\li (e^{2\pi i z}) \right) \lqs \Im\left(\li (e^{2\pi i x})\right)  = \cl(2\pi x)
\end{equation}
by Lemma \ref{dil1}. Also, using Lemma \ref{dil2},
\begin{align}
    \li(1)-\Re\left(\li (e^{2\pi i z})\right)+ 2\pi^2(j+1)  & \lqs 2\li(1)+ 2\pi^2(j+1)  \notag \\
    & = \pi^2\left[1/3+ 2(j+1)\right] \label{ee2}
\end{align}
and \eqref{ef1} follows from \eqref{ee0}, \eqref{ee1} and \eqref{ee2}.

Equation \eqref{om} implies
\begin{equation}\label{eq0}
\Re\left(r \left(z \right) +\frac{\pi i j}{z}\right) = \pi y-y\left(\frac{\li(1)-\Re\left(\li (e^{-2\pi i z})\right)-2\pi^2(j-1)}{2\pi|z|^2} \right)  - \frac{x \ \Im\left(\li (e^{-2\pi i z})\right)}{2\pi|z|^2}.
\end{equation}
For $y \lqs 0$ we have
\begin{equation} \label{eq1}
    -x\Im\left(\li (e^{-2\pi i z}) \right) \lqs -x\Im\left(\li (e^{-2\pi i x})\right)  = -x\cl(-2\pi x) = x\cl(2\pi x)
\end{equation}
by Lemma \ref{dil1}. Then
Lemma \ref{dil2} shows
\begin{align}
    \li(1)-\Re\left(\li (e^{-2\pi i z})\right)- 2\pi^2(j-1)   & \lqs 2\li(1)- 2\pi^2(j-1) \notag \\
    & = \pi^2\left[1/3- 2 (j-1)\right]\label{eq2}
\end{align}
and
writing
\begin{equation} \label{urb}
    \pi y = -\pi |y| = \frac{\pi^2|y| \cdot (-2)|z|^2}{2\pi |z|^2} \lqs \frac{1 }{2\pi |z|^2}\pi^2|y|(-2)
\end{equation}
we see that \eqref{ef2} follows from \eqref{eq0} - \eqref{urb}.
\end{proof}

\subsection{Contour integrals} \label{contint}

Recall that $1/(2i \sin(\pi z))$ has poles exactly at $z=m \in \Z$. Each such pole is simple with residue $(-1)^m/(2\pi i)$. By the calculus of residues, see for example \cite[p. 300]{Ol},
\begin{equation} \label{str}
\sum_{ k =a}^b (-1)^k \varphi(k) =  \int_C \frac{\varphi(z)}{2i\sin(\pi z)} \, dz
\end{equation}
for $\varphi(z)$  a holomorphic function and $C$ a positively oriented closed contour surrounding the interval $[a,b]$ and not surrounding any integers outside this interval.
Next, let $a,b \in \Z$ so that $0<a<b$. With a change of variables in \eqref{str} we obtain
\begin{equation}\label{contour}
\sum_{ k =a}^b \frac{(-1)^k}{k^2} \varphi(N/k) = -\frac{1}{N} \int_C \frac{\varphi(z)}{2i\sin(\pi N/z)} \, dz
\end{equation}
for $C$ now surrounding $\{N/k \ | \ a\lqs k \lqs b\}$.

With \eqref{contour}, we have
\begin{equation}
    \mathcal A_2(N,\sigma) = -\frac{2}{N^{3/2}} \Im \int_C \exp\bigl(N \cdot r \left(z \right) \bigr) \frac{q(z)}{2i\sin(\pi N/z)}  \exp\bigl( v(z;N,\sigma) \bigr)\, dz \label{intf}
\end{equation}
where we may
 take $C$ to be a positively oriented rectangle with left and right vertical sides
\begin{equation*}
C_L:=\left\{1+\delta+i y \ : \ |y| \lqs 1/N^2\right\}, \quad C_R := \left\{3/2-\delta'+i y \ : \ |y| \lqs 1/N^2\right\}
\end{equation*}
and with corresponding horizontal sides $C^+$, $C^-$ with imaginary parts $1/N^2$ and $-1/N^2$, respectively as shown in Figure \ref{cfig}.
%
\SpecialCoor
\psset{griddots=5,subgriddiv=0,gridlabels=0pt}
\psset{xunit=1cm, yunit=1cm}
\psset{linewidth=1pt}
\psset{dotsize=2pt 0,dotstyle=*}
\begin{figure}[h]
\begin{center}
\begin{pspicture}(0,-2)(11,2) 

\psline[linecolor=gray]{->}(1,-2)(1,2)
\psline[linecolor=gray]{->}(0,0)(11,0)
\psline[linecolor=gray](0.85,1)(1.15,1)
\psline[linecolor=gray](0.85,-1)(1.15,-1)
\psline[linecolor=gray](2,-0.15)(2,0.15)
\psline[linecolor=gray](10,-0.15)(10,0.15)

\newrgbcolor{darkbrn}{0.4 0.2 0}

\psset{arrowscale=2}
\psline[linecolor=darkbrn]{->}(9,1)(5.8,1)
\psline[linecolor=darkbrn]{->}(3,-1)(6.2,-1)
\pspolygon[linecolor=darkbrn](2.4,-1)(9.6,-1)(9.6,1)(2.4,1)

\rput(2,-0.4){$1$}
  \rput(10,-0.4){$3/2$}
  \rput(0.4,1){$1/N^2$}
  \rput(0.2,-1){$-1/N^2$}
  \rput(2.4,-1.4){$1+\delta$}
  \rput(9.6,-1.4){$3/2-\delta'$}
  \rput(6,1.5){$C^+$}
  \rput(6,-1.5){$C^-$}
  \rput(2.8,0.3){$C_L$}
  \rput(9.2,0.3){$C_R$}

\end{pspicture}
\caption{The rectangle $C=C^+ \cup C_L \cup C^- \cup C_R$}\label{cfig}
\end{center}
\end{figure}
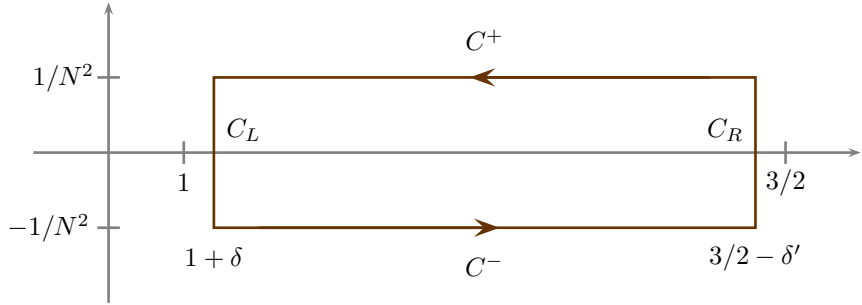
%
Recall that we have some flexibility with $\delta$, $\delta'$ and are free to choose them in $[0.0061,0.01]$. So that the path of integration in \eqref{intf} passes midway between the poles of $1/\sin(\pi N/z)$, we require

\begin{equation} \label{abz}
1+\delta = \frac{N}{b+1/2}, \quad 3/2-\delta' = \frac{N}{a-1/2} \quad \quad \text{for \ }a,b \in \Z, \ \ \delta, \delta' \in [0.0061,0.01].
\end{equation}
The relation in \eqref{abz} implies
\begin{equation*}
    \frac{d\delta}{db} = -\frac{N}{(b+1/2)^2} = -\frac{(1+\delta)^2}{N}
\end{equation*}
so that changing $b$ by $1$ corresponds to changing $\delta$ by $\approx 1/N$. Similarly for $a$ and $\delta'$. Thus, by adjusting $\delta$ and $\delta'$, we can ensure that \eqref{abz} is true for $N$ sufficiently large.

\begin{prop} \label{cpcm}
With  $\delta$, $\delta'$ chosen as in \eqref{abz} we have
\begin{equation*}
\mathcal A_2(N,\sigma) = -\frac{2}{N^{3/2}} \Im \int_{C^+ \cup C^-} \exp\bigl(N \cdot r \left(z \right) \bigr) \frac{q(z)}{2i\sin(\pi N/z)}  \exp\bigl( v(z;N,\sigma) \bigr)\, dz + O(e^{WN})
\end{equation*}
for $W = 0.05$ and an  implied constant depending only on $\sigma$.
\end{prop}
\begin{proof}
The proposition follows from \eqref{intf} if we can show $\int_{C_L \cup C_R} = O(e^{WN})$.
Note that for $z$ on the left vertical side we have $|\pi N/z - \pi(b+1/2)|<1/N$ and similarly on the right vertical side. Since $|\sin(\pi(b+1/2))| =1$ it follows that for large $N$
\begin{equation} \label{19a}
    \frac{1}{2i\sin(\pi N/z)} \ll 1 \qquad (z \in C_L \cup C_R).
\end{equation}
Proposition \ref{qh} implies
\begin{equation} \label{19b}
    q(z) \exp\bigl( v(z;N,\sigma) \bigr)  \ll 1 \qquad (z \in C_L \cup C_R).
\end{equation}
Theorem \ref{rzj} with $j=0$ implies
\begin{equation*}
    \Re\left(r \left(z \right) \right)  < \frac{1}{2\pi} \left(x \cl(2\pi x) + \frac{3\pi^2}{N^2} \right) \qquad (z \in C_L \cup C_R)
\end{equation*}
and we have, using Lemma \ref{clthe} for example,
\begin{equation}\label{futcl}
    \cl(2\pi x) < 0.24 \quad \text{if} \quad 1 \lqs x \lqs 1.01, \qquad \cl(2\pi x) < 0.05 \quad \text{if} \quad 1.49 \lqs x \lqs 1.5.
\end{equation}
Therefore
\begin{equation} \label{klr}
    \Re\left(r \left(z \right) \right)  < \frac{1}{2\pi} \left(1.01 \times 0.24 + \frac{3\pi^2}{N^2} \right)  < 0.05 \qquad (z \in C_L, \ N \gqs 25).
\end{equation}
We obtain \eqref{klr} for $z \in C_R$ in the same way. Consequently
\begin{equation} \label{19c}
    \exp\left(N \cdot r \left(z \right) \right) \ll  \exp(0.05 N) \qquad (z \in C_L \cup C_R).
\end{equation}
The proposition now follows from the bounds \eqref{19a}, \eqref{19b} and \eqref{19c}.
\end{proof}

We note for future reference that
\begin{equation}\label{fut}
    x\cl(2\pi x)/(2\pi) < 0.04 \qquad (x \in [1,1.01] \cup [1.49,1.5]).
\end{equation}

\subsection{Integrating over $C^+$ and $C^-$}
For the integral $\int_{C^+ \cup C^-}$ in Proposition \ref{cpcm}, we consider separately $\int_{C^+}$ and $\int_{C^-}$.
Now
\begin{equation}\label{1/s}
    \frac{1}{2i\sin(\pi N/z)} = \frac{e^{-\pi i N/z}}{1-e^{-2\pi i N/z}} = \quad \sum_{j<0, \text{ odd}} e^{\pi i j N/z} \qquad (\Im(z)>0)
\end{equation}
so that
\begin{equation}\label{smal}
\int_{C^+}  = \sum_{j<0, \text{ odd}} \int_{C^+} \exp\bigl(N \left[ r \left(z \right) +\pi i j/z\right] \bigr) q(z) \exp\bigl(v(z;N,\sigma)\bigr) \, dz.
\end{equation}
Similarly,
$$
\frac{1}{2i\sin(\pi N/z)} = \frac{e^{\pi i N/z}}{e^{2\pi i N/z}-1} = -\sum_{j >0, \text{ odd}} e^{\pi i j N/z} \qquad (\Im(z)<0)
$$
and
\begin{equation}\label{smal22}
\int_{C^-}   = -\sum_{j>0, \text{ odd}} \int_{C^-}  \exp\bigl(N \left[ r \left(z \right) +\pi i j/z\right] \bigr) q(z)  \exp\bigl(v(z;N,\sigma)\bigr) \, dz.
\end{equation}

The contributions to \eqref{smal}, \eqref{smal22} when $|j|>N^2$ are next shown to be negligible.
\begin{lemma} \label{bigj}
For $N \gqs 2$ and $m \gqs 0$,
\begin{alignat}{2}
    \Re\left(r \left(z \right) +\frac{\pi i (-N^2-1-m)}{z}\right) & \lqs -\frac{m}{N^2}  \qquad & & (z \in C^+) \label{ref1}\\
   \Re\left(r \left(z \right) +\frac{\pi i (N^2+1+m)}{z}\right) & \lqs -\frac{m}{N^2}  \qquad & & (z \in C^-). \label{ref2}
\end{alignat}
\end{lemma}
\begin{proof}
With $z \in C^+$ and $j=-N^2-1-m$, Theorem \ref{rzj} implies
\begin{equation} \label{1984}
    \Re\left(r \left(z \right) +\frac{\pi i (-N^2-1-m)}{z}\right) \lqs \frac{1}{2\pi |z|^2} \left(x \cl(2\pi x) +\frac{\pi^2}{N^2}\left[\frac 13 -2N^2-2m \right]  \right).
\end{equation}
Since $1<|z|^2<3$ and $x \cl(2\pi x) \lqs 3 \cl(\pi/3) /2$, we see that \eqref{1984} implies \eqref{ref1}.
The proof of \eqref{ref2} is similar.
\end{proof}

With Proposition \ref{qh} and \eqref{ref1} it follows that
\begin{multline*}
    \sum_{j < -N^2, \text{ odd}} \int_{C^+}  \exp\bigl(N \left[ r \left(z \right) +\pi i j/z\right] \bigl) q(z)  \exp\bigl(v(z;N,\sigma)\bigr) \, dz \\
     \ll \sum_{j < -N^2} \int_{C^+}  \exp\bigl(N  \Re\left[ r \left(z \right) +\pi i j/z\right] \bigl)  \, dz
     \ll \sum_{m \gqs 0} \int_{C^+}  \exp\left(N  (-m/N^2) \right)  \, dz
\end{multline*}
and this last is bounded by
\begin{equation*}
    \sum_{m \gqs 0} e^{-m/N} = \frac 1{1-e^{-1/N}} =O(N).
\end{equation*}
The same is true for $j>N^2$ on $C^-$ and therefore the total contribution to \eqref{smal} and \eqref{smal22} from terms with $|j|>N^2$ is $O(N)$.

\begin{proof}[\bf Proof of Theorem \ref{anmain}]
With Proposition \ref{cpcm} and the above arguments we have shown
\begin{multline} \label{wors}
    \mathcal A_2(N,\sigma) =  -\frac{2}{N^{3/2}} \Im \left[
    \sum_{ -N^2 \lqs j <0, \ j\text{ odd}} \int_{C^+}  \exp\left(N \left[ r \left(z \right) +\pi i j/z\right] \right) q(z)  \exp\bigl(v(z;N,\sigma)\bigr) \, dz
    \right. \\
      -\left. \sum_{0< j \lqs N^2, \ j\text{ odd}} \int_{C^-}  \exp\left(N \left[ r \left(z \right) +\pi i j/z\right] \right) q(z)  \exp\bigl(v(z;N,\sigma)\bigr) \, dz
    \right] + O(e^{WN}).
\end{multline}
We claim that  all terms in \eqref{wors} are $O(e^{0.04 N})$ except the $j=-1$ term.

Let $D^{+}$ be the three lines which, when added to $C^+$, make a rectangle with top side having imaginary part $1$.
Orient $D^{+}$ so that it has the same starting and ending points as $C^+$. Since the integrand in \eqref{wors} is holomorphic here we see that $\int_{C^+ } = \int_{D^+ }$.
We have $q(z)  \exp\bigl(v(z;N,\sigma)\bigr) \ll 1$ for $z \in D^+$ by Proposition \ref{qh}. On the vertical sides of $D^+$ we have
\begin{equation*}
    \Re\left(r \left(z \right) +\frac{\pi i j}{z}\right) < \frac{x \cl(2\pi x)}{2\pi} <0.04
\end{equation*}
by Theorem \ref{rzj} and \eqref{fut} if $j<-1$. On the horizontal side of $D^+$, with $y=1$, Theorem \ref{rzj} implies
\begin{equation*}
    \Re\left(r \left(z \right) +\frac{\pi i j}{z}\right) \lqs \frac{1}{2\pi |z|^2} \left(\frac{3 \cl(\pi/3)}{2} +\pi^2 \left[\frac 13 +2(j+1) \right]  \right) <0
\end{equation*}
if $j<-1$. Hence, for each integer $j$ with $-N^2 \lqs j<-1$, the integral in \eqref{wors} over $C^+$ is $O(e^{0.04 N})$. We will see later that the integral with $j=-1$ cannot be bounded by $O(e^{0.04 N})$.

Similarly, the integral in \eqref{wors} over $C^-$ is $O(e^{0.04 N})$, this time for all odd $j$ with $0<j \lqs N^2$. Hence
\begin{equation*}
    \mathcal A_2(N,\sigma) =  -\frac{2}{N^{3/2}} \Im \int_{C^+}  \exp\left(N \left[ r \left(z \right) +\pi i (-1)/z\right] \right) q(z)  \exp\bigl(v(z;N,\sigma)\bigr) \, dz + O(e^{WN}).
\end{equation*}
We may change the path of integration from $C^+$ to $[1.01,1.49]$. By \eqref{19b}, \eqref{19c} this introduces an error of size $O(e^{WN})$.
\end{proof}

\section{Asymptotics for $\mathcal A_1(N,\sigma)$} \label{sec-sad}

\subsection{The saddle-point method}
Define
\begin{align}\label{p(z)}
p(z) & :=-\left(r(z)-\frac{\pi i}{z} \right)= \frac{\li(1) - \li\left(e^{2\pi i z}\right)}{2\pi i z},\\
\mathcal A_3(N,\sigma) & :=  \frac{2}{N^{3/2}} \Im \int_{1.01}^{1.49} e^{-N \cdot p(z)}  q(z) \cdot \exp\bigl(v(z;N,\sigma)\bigr) \, dz. \label{a3(n)}
\end{align}
Then $p(z)$ is the $d=0$ case of the function $p_d(z)$ we met earlier in \eqref{pdfn}.
We have established with \eqref{a1n} and Theorem \ref{anmain} that, for $W=0.05$,
\begin{equation}\label{lincoln}
    \mathcal A_1(N,\sigma) = \mathcal A_3(N,\sigma) + O(e^{WN}).
\end{equation}
The form of \eqref{a3(n)} allows us to find its asymptotic expansion by the
 saddle-point method. We state a simpler version of \cite[Theorem 7.1, p. 127]{Ol} that is all we need:
\begin{theorem}[Saddle-point method] \label{sdle}
Let $\mathcal P$ be a finite  polygonal path in $\C$ with $p(z)$, $q(z)$  holomorphic functions in a neighborhood of $\mathcal P$. Assume $p$, $q$ and $\mathcal P$ are independent of a parameter $N>0$. Suppose $p'(z)$ has a simple zero at a non-corner point $z_0 \in \mathcal P$ with  $\Re(p(z)-p(z_0))>0$ for $z\in \mathcal P$ except at $z=z_0$. Then there exist explicit  numbers $a_{2s}$ depending on $p$, $q$, $z_0$ and $\mathcal P$ so that we have
\begin{equation}\label{sad}
\int_{\mathcal P} e^{-N \cdot p(z)}q(z)\, dz = 2e^{-N \cdot p(z_0)}\left(\sum_{s=0}^{S-1} \G(s+1/2)\frac{a_{2s}}{N^{s+1/2}} + O\left(\frac{1}{N^{S+1/2}}\right)\right)
\end{equation}
for $S$ an arbitrary positive integer and an implied constant independent of $N$.
\end{theorem}

We need to set up some notation to describe the numbers $a_{2s}$.
Write the power series for $p$ and $q$ near $z_0$ as
\begin{align}
    p(z) & = p(z_0)+ p_0(z-z_0)^2+p_1(z-z_0)^3+ \cdots, \label{psp}\\
    q(z) & = q_0+q_1(z-z_0)+q_2(z-z_0)^2+ \cdots. \label{psq}
\end{align}
(We have $p_0 \neq 0$ by our assumption that $p'(z)$ has a simple zero at $z_0$. For simplicity we also assume that $q_0 \neq 0$. This corresponds to the case $(\mu,\lambda)=(2,1)$ in \cite{Ol} and the case $(\mu,\alpha)=(2,1)$ in \cite{Woj}.)
Choose $\omega \in \C$ giving the direction of the path $\mathcal P$ through $z_0$: near $z_0$,  $\mathcal P$  looks like $z=z_0+\omega t$ for small $t\in \R$ increasing. Note that the condition $\Re(p(z)-p(z_0))>0$ implies $\Re(\omega^2 p_0)>0$.

We also need the {\em partial ordinary Bell polynomials}, see \cite[p. 136]{Comtet}, defined as
\begin{equation} \label{pobell}
    \hat{B}_{i,j}(p_1, p_2, p_3, \dots):= \sum_{\substack{1\ell_1+2 \ell_2+ 3\ell_3+\dots = i \\ \ell_1+ \ell_2+ \ell_3+\dots = j}}
    \frac{j!}{\ell_1! \ell_2! \ell_3! \cdots } p_1^{\ell_1} p_2^{\ell_2} p_3^{\ell_3} \cdots
\end{equation}
where the sum is over all possible $\ell_1$, $\ell_2$, $\ell_3, \dots \in \Z_{\gqs 0}$.
They satisfy, for example,
\begin{equation} \label{pobell2}
    \left( p_1 x +p_2 x^2+ \cdots \right)^j = \sum_{i=j}^\infty \hat{B}_{i,j}(p_1, p_2, \dots) x^i
\end{equation}
and
are related to the usual partial  Bell polynomials by $\hat{B}_{i,j}(p_1, p_2,  \dots) = j! B_{i,j}(1!p_1, 2!p_2, \dots)/i!$. The numbers $a_{2s}$ in Theorem \ref{sdle} may be found by  complicated  manipulations of the series \eqref{psp} and \eqref{psq}, see \cite[pp. 85-86, 121-127]{Ol}. Wojdylo in \cite[Theorem 1.1]{Woj} found  an explicit formula for them. Adapted to the saddle-point method, a special case of his result is
\begin{equation} \label{a2s}
    a_{2s}= \frac{\omega}{2(\omega^2 p_0)^{1/2}} \sum_{i=0}^{2s} q_{2s-i} \sum_{j=0}^i p_0^{-s-j} \binom{-s-1/2}{j} \hat{B}_{i,j}(p_1, p_2, \dots)
\end{equation}
where we must choose the square root $(\omega^2 p_0)^{1/2}$ in \eqref{a2s} so that $\Re \bigl((\omega^2 p_0)^{1/2}\bigr)>0$. Note that
\begin{equation*}
    \frac{\omega}{(\omega^2 p_0)^{1/2}} = \pm \frac{1}{( p_0)^{1/2}},
\end{equation*}
so we see that the dependence of each $a_{2s}$ on the path $\mathcal P$ just involves a sign, corresponding to the direction of the path through the saddle-point.
The first cases are
\begin{equation} \label{a2sb}
    a_0= \frac{\omega}{2(\omega^2 p_0)^{1/2}} q_0, \qquad a_2 = \frac{\omega}{2(\omega^2 p_0)^{1/2}}\left(
    \frac{q_2}{p_0} - \frac{3}{2} \frac{p_1 q_1 + p_2 q_0}{p_0^2} + \frac{15}{8} \frac{p_1^2 q_0}{p_0^3}\right),
\end{equation}
agreeing with \cite[p. 127]{Ol}.
For $\mathcal P$, $p$ and $z_0$ fixed and $q$ possibly varying in \eqref{sad} we write $a_{2s}(q)$ in what follows.


\subsection{A path through the saddle-point}
To apply Theorem \ref{sdle} to $\mathcal A_3(N,\sigma)$ in \eqref{a3(n)}, we need to find the saddle-point for $p(z)$.  By Theorem \ref{disol}, the
unique solution to $p'(z)=0$ for $1/2<\Re(z)<3/2$ is given by the $z_0$ we met earlier in \eqref{w0x}.


\SpecialCoor
\psset{griddots=5,subgriddiv=0,gridlabels=0pt}
\psset{xunit=0.6cm, yunit=0.45cm}
\psset{linewidth=1pt}
\psset{dotsize=4pt 0,dotstyle=*}

\newrgbcolor{darkbrn}{0.4 0.2 0}

\begin{figure}[h]
\begin{center}
\begin{pspicture}(-1,-1)(10,5) 

\psline[linecolor=gray]{->}(0,-1)(0,5.1)
\psline[linecolor=gray]{->}(-1,0)(10,0)
\psline[linecolor=gray](-0.15,4.08844)(0.15,4.08844)
\psline[linecolor=gray](1,-0.15)(1,0.15)
\psline[linecolor=gray](9,-0.15)(9,0.15)
\psline[linecolor=gray](3.9036,-0.15)(3.9036,0.15)

\psset{arrowscale=2,arrowinset=0.5}
\psline[linecolor=darkbrn]{->}(1.16,0)(1.16,2.1)
\psline[linecolor=darkbrn]{->}(8.84,5.156)(8.84,2.3)
\psline[linecolor=darkbrn]{->}(5.4,4.4228)(5.5,4.4448)
\psline[linecolor=darkbrn](1.16,0)(1.16,3.49)(8.84,5.1796)(8.84,0)

\rput(1,-0.6){$1$}
  \rput(9,-0.6){$3/2$}
  \rput(-1,4.08844){$0.255$}
  \rput(3.9036,-0.6){$1.181$}
  \rput(0.5,1.8){$\mathcal P_1$}
  \rput(9.5,2.5){$\mathcal P_3$}
  \rput(5.5,3.7){$\mathcal P_2$}
  \rput(3.9036,4.8){$z_0$}

\psdots(3.9036,4.08844)

\end{pspicture}
\caption{The path $\mathcal P = \mathcal P_1 \cup \mathcal P_2 \cup \mathcal P_3$ through $z_0$}\label{pth}
\end{center}
\end{figure}
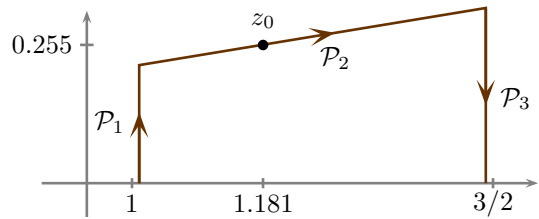
Now we replace the path $[1.01,1.49]$ in \eqref{a3(n)} with a path through  $z_0$.
We noted in Section \ref{dilogg} that $z_0$ may be found to arbitrary precision, and
later  we will  set the parameter $v$ to be
\begin{equation} \label{v}
v=\Im(z_0)/\Re(z_0) \approx 0.216279,
\end{equation}
but  we only require for the results below  that $0.21 \lqs v \lqs 0.22$.

Write
$
    c:=1+i v $.
Let $\mathcal P_1$ be the vertical line from $1.01$ to $1.01c$,  $\mathcal P_2$  the line from $1.01c$ to $1.49c$:
\begin{equation} \label{p2ct}
    \mathcal P_2 = \{ c t \ \mid \ 1.01 \lqs t \lqs 1.49\},
\end{equation}
and let $\mathcal P_3$ be the vertical line from  $1.49c$ to $1.49$.
So the path $\mathcal P := \mathcal P_1 \cup \mathcal P_2 \cup \mathcal P_3$ goes from $1.01$ to $1.49$ and when $v$ is given by \eqref{v} it passes through $z_0$  as in Figure \ref{pth}.
Our goal in this subsection is to prove the following.
 \begin{theorem} \label{sdlever}
 For the path $\mathcal P$ above, passing through the saddle point $z_0$, we have $\Re(p(z)-p(z_0))>0$ for $z\in \mathcal P$ except at $z=z_0$.
 \end{theorem}

 This is exactly the requirement of Theorem \ref{sdle} and it seems apparent from Figure \ref{rpz}. We  prove Theorem \ref{sdlever} by approximating $\Re(p(z))$ and its derivatives by the first terms in their series expansions and  reducing the issue to a finite computation. The path $\mathcal P$ is chosen to make this argument easier and does not use the line of steepest descent.


\SpecialCoor
\psset{griddots=5,subgriddiv=0,gridlabels=0pt}
\psset{xunit=0.46cm, yunit=0.4cm}
\psset{linewidth=1pt}
\psset{dotsize=4pt 0,dotstyle=*}

\begin{figure}[h]
\begin{center}
\begin{pspicture}(-1,-1)(16,8) 

\psline[linecolor=gray]{->}(0,-1)(0,7.5)
\psline[linecolor=gray]{->}(-1,0)(15.5,0)
\psline[linecolor=gray,linestyle=dashed](3,0)(3,4.74953)
\psline[linecolor=gray,linestyle=dashed](11,0)(11,3.99152)
\psline[linecolor=gray,linestyle=dashed](15,0)(15,0.465089)

  \rput(-1,1){$0.01$}
  \rput(-1,3){$0.03$}
   \rput(-1,5){$0.05$}
  \rput(-1,7){$0.07$}
  \rput(1.5,-0.6){$\mathcal P_1$}
  \rput(13,-0.6){$\mathcal P_3$}
  \rput(7,-0.6){$\mathcal P_2$}

\savedata{\mydata}[
{{0., 3.73005}, {0.1, 2.88831}, {0.2, 2.41882}, {0.3, 2.14558}, {0.4,
  1.98279}, {0.5, 1.89026}, {0.6, 1.84673}, {0.7, 1.83958}, {0.8,
  1.8607}, {0.9, 1.90445}, {1., 1.96678}, {1.1, 2.04459}, {1.2,
  2.13549}, {1.3, 2.23754}, {1.4, 2.34917}, {1.5, 2.46906}, {1.6,
  2.59609}, {1.7, 2.72932}, {1.8, 2.8679}, {1.9, 3.01112}, {2.,
  3.15833}, {2.1, 3.30897}, {2.2, 3.46253}, {2.3, 3.61855}, {2.4,
  3.77663}, {2.5, 3.93638}, {2.6, 4.09748}, {2.7, 4.25962}, {2.8,
  4.4225}, {2.9, 4.58589}, {3., 4.74953}, {3.1, 4.8924}, {3.2,
  5.031}, {3.3, 5.16511}, {3.4, 5.29452}, {3.5, 5.41903}, {3.6,
  5.53848}, {3.7, 5.65274}, {3.8, 5.76169}, {3.9, 5.86523}, {4.,
  5.96329}, {4.1, 6.05582}, {4.2, 6.14279}, {4.3, 6.22418}, {4.4,
  6.29999}, {4.5, 6.37024}, {4.6, 6.43497}, {4.7, 6.49421}, {4.8,
  6.54803}, {4.9, 6.59649}, {5., 6.63967}, {5.1, 6.67766}, {5.2,
  6.71055}, {5.3, 6.73845}, {5.4, 6.76145}, {5.5, 6.77968}, {5.6,
  6.79325}, {5.7, 6.80228}, {5.8, 6.80691}, {5.9, 6.80725}, {6.,
  6.80343}, {6.1, 6.79559}, {6.2, 6.78386}, {6.3, 6.76838}, {6.4,
  6.74927}, {6.5, 6.72668}, {6.6, 6.70073}, {6.7, 6.67156}, {6.8,
  6.63929}, {6.9, 6.60407}, {7., 6.56603}, {7.1, 6.52528}, {7.2,
  6.48197}, {7.3, 6.43621}, {7.4, 6.38813}, {7.5, 6.33786}, {7.6,
  6.2855}, {7.7, 6.23119}, {7.8, 6.17504}, {7.9, 6.11716}, {8.,
  6.05767}, {8.1, 5.99668}, {8.2, 5.93428}, {8.3, 5.8706}, {8.4,
  5.80574}, {8.5, 5.73979}, {8.6, 5.67286}, {8.7, 5.60504}, {8.8,
  5.53642}, {8.9, 5.46711}, {9., 5.39719}, {9.1, 5.32674}, {9.2,
  5.25587}, {9.3, 5.18464}, {9.4, 5.11315}, {9.5, 5.04146}, {9.6,
  4.96967}, {9.7, 4.89783}, {9.8, 4.82604}, {9.9, 4.75436}, {10.,
  4.68285}, {10.1, 4.61159}, {10.2, 4.54065}, {10.3, 4.47008}, {10.4,
  4.39994}, {10.5, 4.33031}, {10.6, 4.26122}, {10.7, 4.19275}, {10.8,
  4.12494}, {10.9, 4.05785}, {11., 3.99152}, {11.1, 3.92005}, {11.2,
  3.84817}, {11.3, 3.77587}, {11.4, 3.70314}, {11.5, 3.62998}, {11.6,
  3.55638}, {11.7, 3.4823}, {11.8, 3.40775}, {11.9, 3.3327}, {12.,
  3.25713}, {12.1, 3.181}, {12.2, 3.1043}, {12.3, 3.02699}, {12.4,
  2.94904}, {12.5, 2.8704}, {12.6, 2.79103}, {12.7, 2.71089}, {12.8,
  2.62992}, {12.9, 2.54806}, {13., 2.46526}, {13.1, 2.38144}, {13.2,
  2.29654}, {13.3, 2.21047}, {13.4, 2.12317}, {13.5, 2.03452}, {13.6,
  1.94445}, {13.7, 1.85285}, {13.8, 1.75962}, {13.9, 1.66464}, {14.,
  1.56778}, {14.1, 1.46893}, {14.2, 1.36795}, {14.3, 1.2647}, {14.4,
  1.15903}, {14.5, 1.05079}, {14.6, 0.939814}, {14.7,
  0.825935}, {14.8, 0.708978}, {14.9, 0.588759}, {15., 0.465089}}
]

\multirput(-0.15,1)(0,1){7}{\psline[linecolor=gray](0,0)(0.3,0)}

\dataplot[linecolor=red,linewidth=0.8pt,plotstyle=line]{\mydata}

  \rput(5.85792,0.7){$z_0$}

\psdots(5.85792,6.80762)(5.85792,0)

\rput(10,6.5){$\Re[-p(z)]$}

\end{pspicture}
\caption{Graph of $\Re[-p(z)]$ for $z \in \mathcal P$}\label{rpz}
\end{center}
\end{figure}
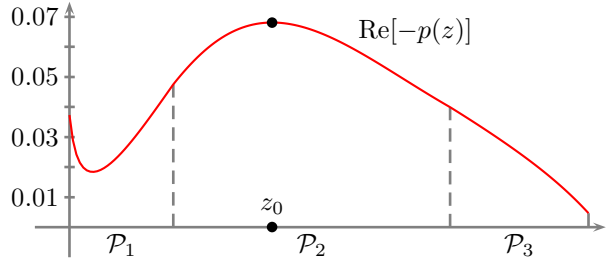


Generalizing to $p_d(z)$, we examine $\Re(p_d(z))$ for $z$ on the ray $z=ct$ for $c=1+ i v$ with $v>0$. We also
write
\begin{equation*}
    c = \rho e^{i \theta}  \qquad (0<\rho, \ 0<\theta <\pi/2).
\end{equation*}
Then, using \eqref{def0} since $|e^{2\pi i z}| \lqs 1$ when $\Im(z) \gqs 0$,
\begin{align}
    \Re [p_d(ct) ] & = \Re\left[\frac{-i(\li(1)+4\pi^2 d) e^{-i \theta}}{2\pi \rho t}+ \frac{i e^{-i \theta}}{2\pi \rho t} \sum_{m=1}^\infty \frac{e^{-2\pi m v t} e^{2\pi i m t}}{m^2}\right] \notag\\
    & = \frac{1}{2t}\left(\frac{-\pi(24d+1) \sin \theta}{6\rho}-\frac 1{\pi \rho} \sum_{m=1}^\infty \frac{e^{-2\pi m v t} \sin(2\pi m t-\theta)}{m^2}\right) \label{rp0}.
\end{align}
Similarly, employing \eqref{pzzz}, \eqref{pzzz2}
\begin{align}
   \frac{d}{dt}\Re [p_d(ct) ] = \Re [c p_d'(ct)] & =-\frac{1}{ t}\left(\Re [p_d(ct)] + \sum_{m=1}^\infty \frac{e^{-2\pi m v t} \cos(2\pi m t)}{m}\right) \label{rp1}\\
    \frac{d^2}{dt^2}\Re [p_d(ct) ] = \Re [c^2 p_d''(ct)] & =-\frac{2}{ t}\left(\Re [c p_d'(ct)] -\pi \rho \sum_{m=1}^\infty e^{-2\pi m v t} \sin(2\pi m t+\theta)\right). \label{rp2}
\end{align}
We may bound the tails of these series:
\begin{equation*}
    \left| \sum_{m=L}^\infty \frac{e^{-2\pi m v t} }{m^k} \right| \lqs  \frac{e^{-2L\pi  v t }}{L^k(1-e^{-2\pi  v t})} .
\end{equation*}
Collecting the first $L-1$ terms in \eqref{rp0}, \eqref{rp1} and \eqref{rp2} we obtain
\begin{equation*}
    \frac{d^2}{dt^2}\Re [p_d(ct) ] = R_2(L;t)+R^*_2(L;t)
\end{equation*}
(with the subscript $2$ indicating the second derivative) for
\begin{align*}
    R_2(L;t) & :=  -  \frac{\pi (24d+1)\sin \theta}{6\rho t^3} + \sum_{m=1}^{L-1}\Bigl( A_m(t) \cos(2\pi m t) + B_m(t) \sin(2\pi m t)\Bigr),\\
    A_m(t) & := e^{-2\pi m v t} \left(\frac{2}{m t^2}+\sin \theta\left(\frac{2\pi \rho}{t}+\frac{1}{m^2 \pi \rho t^3} \right) \right),\\
    B_m(t) & := e^{-2\pi m v t} \cos \theta \left(\frac{2\pi \rho}{t}-\frac{1}{m^2 \pi \rho t^3}  \right)
\end{align*}
and
\begin{equation*}
    |R^*_2(L;t)| \lqs E_2(L;t) := \frac{e^{-2\pi L v t}}{1-e^{-2\pi v t}}\left( \frac{  1}{\pi \rho L^2 t^3}
    + \frac{2}{L t^2} + \frac{2\pi \rho }{t}\right).
\end{equation*}
We see that $E_2(L;t)$ is a decreasing function of $L$ and $t$. We have $A_m(t)$ a positive and decreasing function of $t$. Also  $B_m(t)$ is a positive and decreasing function of $t$ when $t>\frac{\sqrt{3}}{\sqrt{2} \pi \rho m}$.

Let $v_1=0.21$ and $v_2=0.22$. Writing $\rho_1 e^{i \theta_1}=1+i v_1$ and $\rho_2 e^{i \theta_2}=1+i v_2$ we have
\begin{equation*}
    1<\rho_1\lqs \rho \lqs \rho_2, \quad 0< \theta_1 \lqs \theta \lqs \theta_2<\pi/2.
\end{equation*}
For $v$ in the interval $[v_1,v_2]$, we may bound $A_m(t)$, $B_m(t)$ and $E_2(L;t)$   from above and below by replacing $v$, $\rho$ and $\theta$ appropriately by $v_j$, $\rho_j$ and $\theta_j$, $j=1,2$. For example
\begin{equation*}
    0 < A^-_m(t)  \lqs A_m(t) \lqs A^+_m(t) \qquad (v \in [v_1,v_2])
\end{equation*}
with
\begin{align*}
    A^-_m(t) & := e^{-2\pi m v_2 t} \left(\frac{2}{m t^2}+\sin \theta_1\left(\frac{2\pi \rho_1}{t}+\frac{1}{m^2 \pi \rho_2 t^3} \right) \right),\\
    A^+_m(t) & := e^{-2\pi m v_1 t} \left(\frac{2}{m t^2}+\sin \theta_2\left(\frac{2\pi \rho_2}{t}+\frac{1}{m^2 \pi \rho_1 t^3} \right) \right)
\end{align*}
and similarly write $0< B^-_m(t)  \lqs B_m(t) \lqs B^+_m(t)$ and  $0< E^-_2(L;t) \lqs E_2(L;t) \lqs E^+_2(L;t)$.

\begin{lemma} \label{msy1}
Let $c=1+i v$ with $0.21 \lqs v \lqs 0.22$. Then  $\frac{d^2}{dt^2}\Re [p(ct) ] >0$ for $t \in [1,5/4]$.
\end{lemma}
\begin{proof}
Break up $[1,5/4]$ into $n$ equal segments $[x_{j-1},x_j]$. Then
\begin{equation} \label{dabel}
    \frac{d^2}{dt^2}\Re [p(ct) ] \gqs \min_{1\lqs j \lqs n} \left(\left(\min_{t \in [x_{j-1},x_j]} R_2(L;t) \right) -E^+_2(L;x_{j-1})\right).
\end{equation}
Let $t=x^*_{j,m}$ correspond to the minimum value of $\cos(2\pi m t)$ for $t\in [x_{j-1},x_j]$ (so that $x^*_{j,m}$ equals $x_{j-1}$, $x_j$ or a local minimum $k/2m$ for $k$ odd). Similarly, let $t=x^{**}_{j,m}$ correspond to the minimum value of $\sin(2\pi m t)$ for $t\in [x_{j-1},x_j]$. Then
\begin{equation} \label{dabell}
    \min_{t \in [x_{j-1},x_j]} R_2(L;t) \gqs  -  \frac{\pi \sin \theta_2}{6\rho_1 x_{j-1}^3} + \sum_{m=1}^{L-1}\Bigl( A^-_m(x_j) \cos(2\pi m x^*_{j,m}) + B^-_m(x_j) \sin(2\pi m x^{**}_{j,m})\Bigr)
\end{equation}
where we must replace $A^-_m(x_j)$  in \eqref{dabell} by $A^+_m(x_{j-1})$  if $\cos(2\pi m x^*_{j,m})<0$ and replace $B^-_m(x_j)$ in \eqref{dabell} by $B^+_m(x_{j-1})$ if $\sin(2\pi m x^{**}_{j,m})<0$.

A computation using \eqref{dabel} and \eqref{dabell} with $L=3$ and $n=2$ for example shows $\frac{d^2}{dt^2}\Re [p(ct) ] > 0.12$.
\end{proof}

We may analyze the first derivative in a similar way. We have
\begin{equation*}
    \frac{d}{dt}\Re [p_d(ct) ] = R_1(L;t)+R^*_1(L;t)
\end{equation*}
for
\begin{align*}
    R_1(L;t) & :=   \frac{\pi (24d+1)\sin \theta}{12\rho t^2} + \sum_{m=1}^{L-1}\Bigl( -C_m(t) \cos(2\pi m t) + D_m(t) \sin(2\pi m t)\Bigr),\\
    C_m(t) & := e^{-2\pi m v t} \left(\frac{1}{m t}+\frac{\sin \theta}{m^2 2\pi \rho t^2} \right), \qquad D_m(t)  := e^{-2\pi m v t} \frac{\cos \theta}{m^2 2\pi \rho t^2}\\
\end{align*}
and
\begin{equation*}
    |R^*_1(L;t)| \lqs E_1(L;t) := \frac{e^{-2\pi L v t}}{1-e^{-2\pi v t}}\left( \frac{  1}{2\pi \rho L^2 t^2}
    + \frac{1}{L t} \right).
\end{equation*}
We see that $E_1(L;t)$ is a decreasing function of $L$ and $t$. Also $C_m(t)$ and $D_m(t)$ are positive and decreasing functions of $t$.

\begin{lemma} \label{msy2}
Let $c=1+i v$ with $0.21 \lqs v \lqs 0.22$. Then  $\frac{d}{dt}\Re [p(ct) ] >0$ for $t \in [5/4, 3/2]$.
\end{lemma}
\begin{proof}
  Break $[5/4, 3/2]$ into $n$ equal segments and, as in the proof of Lemma \ref{msy1}, bound $\frac{d}{dt}\Re [p(ct) ]$ from below on each piece. Taking $n=2$ and $L=3$ shows $\frac{d}{dt}\Re [p(ct) ] > 0.03$ for example.
\end{proof}

\begin{cor} \label{bfl}
Let $c=1+i v$ with $0.21 \lqs v \lqs 0.22$. There is a unique solution to $\frac{d}{dt}\Re [p(ct)]=0 $   for $t\in [1,3/2]$ that we label as $t_0$. We then have $\Re [p(ct)-p(c t_0) ] >0$ for $t\in [1,3/2]$ except at $t=t_0$.
\end{cor}
\begin{proof}
 Check that $\frac{d}{dt}\Re [p(ct) ]<0$ when $t=1$ and $\frac{d}{dt}\Re [p(ct) ]>0$ when $t=5/4$. By Lemma \ref{msy1} we see that $\frac{d}{dt}\Re [p(ct) ]$ is strictly increasing for $t\in [1,5/4]$. It necessarily has a unique zero that we label $t_0$. By Lemma \ref{msy2}, $\frac{d}{dt}\Re [p(ct) ]$ remains $>0$ for $t\in [5/4,3/2]$ . Hence $\Re [p(ct)-p(c t_0)]$ is strictly decreasing on $[1,t_0)$ and strictly increasing on $(t_0,3/2]$ as required.
\end{proof}

Before the proof of Theorem \ref{sdlever}, we need a result
similar to Lemma \ref{dil1} to let us find bounds on $\mathcal P_1 \cup \mathcal P_3$.

\begin{lemma} \label{dil1a}
Consider $\Re(\li(e^{2\pi i z}))$  as a function of $y \gqs 0$. It is positive and decreasing for fixed $x$ with $|x| \lqs 1/6$. It is negative and increasing for fixed $x$ with $1/4 \lqs |x| \lqs 3/4$.
\end{lemma}
\begin{proof}
We have
\begin{equation} \label{dstr}
    \frac{d}{dy} \Re(\li(e^{2\pi i z})) = \Im(2\pi i \log(1-e^{2\pi i z})) = 2\pi \log |1-e^{2\pi i z}|.
\end{equation}
Noting that
\begin{align*}
    |1-e^{2\pi i z}| & \lqs 1 \qquad (|x| \lqs 1/6, \ y \gqs 0),  \\
  1 \lqs  |1-e^{2\pi i z}| & \lqs 2 \qquad (1/4 \lqs |x| \lqs 3/4, \ y \gqs 0)
\end{align*}
shows that the derivative \eqref{dstr} is negative for $|x| \lqs 1/6$ and positive for $1/4 \lqs |x| \lqs 3/4$. Also, we have
\begin{equation*}
    \lim_{y \to \infty}\Re(\li(e^{2\pi i z})) = \Re(\li(0)) = 0
\end{equation*}
implying the function decreases or increases to zero.
\end{proof}

\begin{prop} \label{bxl}
For $0.21 \lqs v \lqs 0.22$  we have $\Re[-p(z)]<0.06$ for $z \in \mathcal P_1 \cup \mathcal P_3$.
\end{prop}
\begin{proof} We have $x$ fixed as $1.01$ on $\mathcal P_1$ and $1.49$ on $\mathcal P_3$.
By \eqref{ee0} we know
\begin{equation*}
    \Re[-p(z)] = \frac{f(y)+g(y)}{2\pi |z|^2}
\end{equation*}
for
\begin{equation*}
    f(y):=y\left(\li(1)-\Re(\li(e^{2\pi i z}))\right), \quad g(y)= x \Im (\li(e^{2\pi i z})).
\end{equation*}
 If $x=1.01$ or $1.49$ it follows from Lemma \ref{dil1} that $g(y)$ is positive and decreasing. We claim that, for these $x$ values, $f(y)$ is always positive and increasing for $y>0$.

For $x=1.01$, Lemma \ref{dil1a} tells us that $\Re(\li(e^{2\pi i (x+i y)}))$ is a decreasing function of $y$. Recalling \eqref{reli}, we see it decreases from $\pi^2 B_2(0.01) < \pi^2/6$. Therefore $f(y)$ is  positive and increasing for $x=1.01$.

Next let $x=1.49$. We have
\begin{equation*}
    \frac{d}{dy} f(y)= \li(1)-\Re(\li(e^{2\pi i z})) -2\pi y \log |1-e^{2\pi i z}|
\end{equation*}
as in \eqref{dstr}. Lemma \ref{dil1a} implies that $-\Re(\li(e^{2\pi i z})) \gqs 0$ so that
\begin{equation*}
    \frac{d}{dy} f(y) \gqs \pi^2/6 -2\pi y \log (1+e^{-2\pi y})>0.
\end{equation*}
Since $f(0)=0$, we have shown  $f(y)$ is  positive and increasing for $x=1.49$.

For $z\in \mathcal P_1$, so that $x=1.01$ and $0\lqs y \lqs Y:=1.01\times 0.22=0.2222$,
\begin{equation*}
    \Re[-p(z)] \lqs \begin{cases} (f(Y/2)+g(0))/(2\pi 1.01^2) \approx 0.0558 & \quad y \in [0,Y/2]\\
    (f(Y)+g(Y/2))/(2\pi( 1.01^2+(Y/2)^2) \approx 0.054 & \quad y \in [Y/2,Y].
    \end{cases}
\end{equation*}
For $z\in \mathcal P_3$, so that $x=1.49$ and $0\lqs y \lqs Y:=1.49 \times 0.22 = 0.3278$,
\begin{equation*}
    \Re[-p(z)] \lqs (f(Y)+g(0))/(2\pi 1.49^2) \approx 0.0462,  \qquad y \in [0,Y]. \qedhere
\end{equation*}
\end{proof}

\begin{proof}[\bf Proof of Theorem \ref{sdlever}]
Let $v$ be given by \eqref{v}.  Then the saddle-point $z_0$ lies on $\mathcal P_2$, parameterized in \eqref{p2ct}, and when $t=\Re(z_0)$ we have $ct=z_0$. Then
\begin{equation*}
    \left.\frac{d}{dt}\Re [p(ct) ]\right|_{t=\Re(z_0)} = \Re [c p'(c\Re(z_0))] = \Re [c p'(z_0)]=0.
\end{equation*}
It follows from Corollary \ref{bfl} that $\Re [p(z)-p(z_0) ] >0$ for $z \in \mathcal P_2$ and $z \neq z_0$. We also note that $\Re [-p(z_0) ] = U \approx 0.068076$.

For $z \in \mathcal P_1 \cup \mathcal P_3$,  Proposition \ref{bxl} implies $\Re [p(z)-p(z_0) ] > -0.06 +0.068 >0$.
\end{proof}

\subsection{Asymptotic expansions}
In order to apply Theorem \ref{sdle} to \eqref{a3(n)} we need to understand the dependence of $\exp\bigl(v(z;N,\sigma)\bigr)$ on $N$ and remove this dependence from the integral. The result we prove in this subsection is the following.
\begin{prop} \label{gas}
For $1.01 \lqs \Re(z) \lqs 1.49$ and $|\Im(z)|\lqs 1$, say, there are holomorphic functions $u_{\sigma,j}(z)$ and $\zeta_d(z;N,\sigma)$ of $z$ so that
\begin{equation*}
    \exp\bigl(v(z;N,\sigma)\bigr) = \sum_{j=0}^{d-1} \frac{u_{\sigma,j}(z)}{N^j} + \zeta_d(z;N,\sigma) \quad \text{for} \quad \zeta_d(z;N,\sigma) = O\left(\frac{1}{N^d} \right)
\end{equation*}
with an implied constant depending only on $\sigma$ and  $d$ where $1 \lqs d \lqs 2L-1$ and $L=\lfloor 0.006 \pi e \cdot N \rfloor$.
\end{prop}

We first establish a general result.
Fix $M \in \Z_{\gqs 1}$. Suppose we have a function $f$ on the positive integers with the following property. There exist $a_1, \dots a_{M-1} \in \C$ and $K=K(M)>0$ so that
\begin{equation} \label{poiase}
    \left|f(N)-\sum_{i=1}^{M-1} \frac{a_i}{N^i} \right| \lqs \frac{K}{N^M}
\end{equation}
for all $N \in \Z_{\gqs 1}$. In other words
\begin{equation}\label{1=2}
    f(N) = \sum_{i=1}^{M-1} \frac{a_i}{N^i} +O\left( \frac{1}{N^M}\right).
\end{equation}
We next show that $\exp\bigl(f(N)\bigr)$ must have a similar expansion to \eqref{1=2}.

Set $A  :=|a_1|+|a_2|+ \dots +|a_{M-1}|$ and
\begin{equation} \label{pizza}
    b_j  := \sum_{i_1+2 i_2+ \dots +(M-1) i_{M-1} = j} \frac{a_1^{i_1} a_2^{i_2} \dots a_{M-1}^{i_{M-1}}}{i_1! i_2! \dots i_{M-1}!}.
\end{equation}

\begin{lemma} \label{awhi}
With the above $M$, $f$, $a_i$, $K$, $A$ and $b_j$ we have
\begin{equation*}
    \left|\exp\bigl(f(N)\bigr)-\sum_{j=0}^{m-1} \frac{b_j}{N^j} \right| \lqs \frac{e^{A+K}}{N^m}
\end{equation*}
for any $m$ with $1\lqs m \lqs M$ and all $N \in \Z_{\gqs 1}$.
\end{lemma}
\begin{proof}
Clearly
\begin{equation*}
    f(N)=\sum_{i=1}^{M-1} \frac{a_i}{N^i} + \frac{f_M(N)}{N^M}
\end{equation*}
for some $f_M(N)$ with $|f_M(N)| \lqs K$. Therefore
\begin{equation} \label{jbj2}
    \exp\bigl(f(N)\bigr)=\exp\left(\sum_{i=1}^{M-1} \frac{a_i}{N^i}\right)  \exp\left(\frac{f_M(N)}{N^M}\right).
\end{equation}
We have
\begin{equation} \label{jbj3}
    \exp\left(\sum_{i=1}^{M-1} \frac{a_i}{N^i}\right)  = \left(\sum_{i_1=0}^\infty \frac{a_1^{i_1}}{N^{1\cdot i_1} i_1!} \right) \dots
    \left(\sum_{i_{M-1}=0}^\infty \frac{a_{M-1}^{i_{M-1}}}{N^{(M-1)\cdot i_{M-1}} i_{M-1}!} \right)
     = \sum_{j=0}^\infty \frac{b_j}{N^j}.
\end{equation}
Note that if we replace the $a_i$s by their absolute values in \eqref{pizza}, \eqref{jbj3} we find
\begin{equation}\label{jbj}
    \sum_{j=0}^\infty \frac{|b_j|}{N^j} \lqs \exp\left(\sum_{i=1}^{M-1} \frac{|a_i|}{N^i}\right) \lqs e^A \qquad(N \in \Z_{\gqs 1}),
\end{equation}
and in particular, \eqref{jbj} is valid for $N=1$.

With \eqref{jbj2}, \eqref{jbj3}
\begin{equation*}
    \exp\bigl(f(N)\bigr)  = \sum_{j=0}^{m-1} \frac{b_j}{N^j} + \sum_{j=m}^{\infty} \frac{b_j}{N^j} + \left( \sum_{j=0}^{\infty} \frac{b_j}{N^j} \right) \left( \exp\left( \frac{f_M(N)}{N^M} \right) -1\right).
\end{equation*}
Recall the inequality \eqref{siine}
\begin{equation*}
    |e^x-1| \lqs |x| \frac{e^\kappa -1}{\kappa} \quad \text{ for } \quad x \in (-\infty,\kappa], \ \kappa>0.
\end{equation*}
It follows that
\begin{equation}\label{rt}
    \exp\left( \frac{f_M(N)}{N^M} \right) -1 \lqs \frac{f_M(N)}{N^M} \frac{e^K-1}{K} \lqs \frac{e^K-1}{N^M}.
\end{equation}
Hence
\begin{equation*}
    \left|\sum_{j=m}^{\infty} \frac{b_j}{N^j} + \left( \sum_{j=0}^{\infty} \frac{b_j}{N^j} \right) \left( \exp\left( \frac{f_M(N)}{N^M} \right) -1\right) \right|
    \lqs \frac{e^A}{N^m} + e^A \frac{e^K-1}{N^M} \lqs \frac{e^{A+K}}{N^m},
\end{equation*}
proving the lemma.
\end{proof}

If \eqref{poiase} is valid for every $M \in \Z_{\gqs 1}$ then this is an example of an {\em asymptotic  expansion}. It may be written formally as
\begin{equation*}
    f(N) \sim \sum_{i=1}^{\infty} \frac{a_i}{N^i}
\end{equation*}
where the right side does not necessarily converge. Lemma \ref{awhi} relates the asymptotic expansions of $\exp\bigl(f(N)\bigr)$ and $f(N)$. See also \cite[p. 22]{Ol} for similar exercises.

\begin{proof}[\bf Proof of Proposition \ref{gas}]
Recall from Corollary \ref{acdc} that
for  $z \in \C$ such that $1.01 \lqs \Re(z) \lqs 1.49$  we have
\begin{equation} \label{wworld}
    v(z;N,\sigma) =  \frac{2\pi i \sigma z}N +\sum_{\ell=1}^{d-1} \frac{g_{\ell}(z)}{N^{2\ell-1}} + O\left( \frac 1{N^{2d-1}}\right)
\end{equation}
 for $2 \lqs d\lqs L=\lfloor 0.006 \pi e \cdot N \rfloor$ and an implied constant, $K(d)$, depending only on  $d$.
For $j \in \Z_{\gqs 0}$ put
\begin{equation} \label{uiz}
    u_{\sigma,j}(z):=\sum_{m_1+3m_2+5m_3+ \dots =j}\frac{(2\pi i \sigma z +g_1(z))^{m_1}}{m_1!}\frac{g_2(z)^{m_2}}{m_2!} \cdots \frac{g_j(z)^{m_j}}{m_j!},
\end{equation}
with $u_{\sigma,0}=1$.
Apply Lemma \ref{awhi} with $f(N)$ replaced by $v(z;N,\sigma)$ and $a_1 = 2\pi i \sigma z +g_1(z)$, $a_2=0, \ \cdots$ and also $b_j = u_{\sigma,j}(z)$. Let $A(\sigma,d)$ be a bound for $|a_1|+ \cdots + |a_{d-1}|$ for $z$ in the stated range of the proposition. Set
\begin{equation*}
    \zeta_d(z;N,\sigma) := \exp\bigl(v(z;N,\sigma)\bigr) - \sum_{j=0}^{d-1} \frac{u_{\sigma,j}(z)}{N^j}.
\end{equation*}
Then $\zeta_d(z;N,\sigma)$ is clearly holomorphic in $z$ and  Lemma \ref{awhi} implies $|\zeta_d(z;N,\sigma)|\lqs \exp\bigl(A(\sigma,d)+K(d)\bigr)/N^d$ as required.
\end{proof}

\subsection{Proofs of Theorems \ref{mainthmb} and \ref{maina}} \label{prfs}
We restate Theorem \ref{maina} here:
\begin{maina} With $b_{0}=2  z_0 e^{-\pi i z_0}$ and  explicit  $b_{1}(\sigma),$ $b_{2}(\sigma), \dots $ depending on $\sigma \in \Z$ we have
\begin{equation} \label{pres}
   \mathcal A_1(N,\sigma) = \Re\left[\frac{w_0^{-N}}{N^{2}} \left( b_{0}+\frac{b_{1}(\sigma)}{N}+ \dots +\frac{b_{m-1}(\sigma)}{N^{m-1}}\right)\right] + O\left(\frac{|w_0|^{-N}}{N^{m+2}}\right)
\end{equation}
for an implied constant depending only on  $\sigma$ and $m$.
\end{maina}
\begin{proof}
Recall from \eqref{pzlogw} that
$e^{-p(z_0)} = w_0^{-1}$.
Proposition \ref{gas} implies
\begin{equation}\label{umand}
    \mathcal A_3(N,\sigma)  = \Im\Biggl[ \sum_{j=0}^{d-1} \frac{2}{N^{3/2+j}}  \int_\mathcal P e^{-N \cdot p(z)} \cdot q(z) \cdot u_{\sigma,j}(z) \, dz + \frac{2}{N^{3/2}}  \int_\mathcal P e^{-N \cdot p(z)} \cdot q(z) \cdot \zeta_d(z;N,\sigma) \, dz \Biggr]
\end{equation}
where the last term in \eqref{umand} is
\begin{equation*}
    \ll \frac{1}{N^{3/2}}  \int_\mathcal P \left|e^{-N \cdot p(z)}\right| \cdot 1 \cdot \frac{1}{N^d} \, dz
    \ll \frac{1}{N^{d+3/2}} e^{-N \Re(p(z_0))} = \frac{|w_0|^{-N}}{N^{d+3/2}}
\end{equation*}
by Theorem \ref{sdlever} and Propositions \ref{qh} and \ref{gas}.
Applying Theorem \ref{sdle} to each integral  in the first part of \eqref{umand}
 we obtain
\begin{equation} \label{wmand}
    \int_\mathcal P e^{-N \cdot p(z)} \cdot q(z) \cdot u_{\sigma,j}(z) \, dz = 2 e^{-N p(z_0)}\left(\sum_{s=0}^{S-1} \G(s+1/2)\frac{a_{2s}(q \cdot u_{\sigma,j})}{N^{s+1/2}}+O\left( \frac{1}{N^{S+1/2}}\right) \right).
\end{equation}
The error term in \eqref{wmand} corresponds to an error for $\mathcal A_3(N,\sigma)$ of size $O(|w_0|^{-N}/N^{s+j+2})$.
We choose $S=d$ so that this error  is less than $O(|w_0|^{-N}/N^{d+3/2})$ for all $j \gqs 0$.
Therefore
\begin{align*}
    \mathcal A_3(N,\sigma) & = \Im \left[
    \sum_{j=0}^{d-1} \frac{4}{N^{j+3/2}}   e^{-N \cdot p(z_0)} \sum_{s=0}^{d-1} \frac{\G(s+1/2) a_{2s}(q \cdot u_{\sigma,j})}{N^{s+1/2}}
    \right]+ O\left( \frac{|w_0|^{-N}}{N^{d+3/2}}\right) \\
    & = \Im \left[  w_0^{-N}
    \sum_{t=0}^{2d-2} \frac{4}{N^{t+2}}    \sum_{s=0}^{\min(t,d-1)} \G(s+1/2) a_{2s}(q \cdot u_{\sigma, t-s})
    \right]+ O\left( \frac{|w_0|^{-N}}{N^{d+3/2}}\right) \\
    & = \Re \left[  w_0^{-N}
    \sum_{t=0}^{d-2} \frac{-4i}{N^{t+2}}    \sum_{s=0}^{t} \G(s+1/2) a_{2s}(q \cdot u_{\sigma, t-s})
    \right]+ O\left( \frac{|w_0|^{-N}}{N^{d+1}}\right).
\end{align*}
Hence, recalling \eqref{lincoln} and with
\begin{equation} \label{btys}
    b_t(\sigma):=  -4i \sum_{s=0}^t \G(s+1/2) a_{2s}(q \cdot u_{\sigma,t-s}),
\end{equation}
we obtain \eqref{pres} in the statement of the theorem.

Use the formula for $a_0$ on the left of \eqref{a2sb} to get
\begin{equation} \label{oad}
    b_0(\sigma)=-4i\G(1/2)a_0(q \cdot u_{\sigma,0})= -4i\sqrt{\pi} \left(\frac{\omega}{2(\omega^2 p_0)^{1/2}} q_0\right)
\end{equation}
which is independent of $\sigma$. The terms $p_0$ and $q_0$ are defined in \eqref{psp}, \eqref{psq} so that, using \eqref{pzzz2},
\begin{equation} \label{oad2}
    p_0  =p''(z_0)/2 = \frac{-\pi i e^{2\pi i z_0}}{z_0 w_0},\qquad
    q^2_0  =q(z_0)^2 = \frac{i z_0}{w_0}.
\end{equation}
The square of the term in parentheses in \eqref{oad} is therefore
\begin{equation*}
    \frac{q_0^2}{4p_0} = \frac{-z_0^2}{4\pi e^{2\pi i z_0}}.
\end{equation*}
We may take $\omega=z_0$ since the path $\mathcal P_2$ is a segment of the ray from the origin through $z_0$. A numerical check then gives us the correct square root:
\begin{equation} \label{a0q}
    a_0(q) = \frac{\omega}{2(\omega^2 p_0)^{1/2}} q_0 = \frac{i z_0}{2\sqrt{\pi} e^{\pi i z_0}}
\end{equation}
and the formula $b_{0}=2  z_0 e^{-\pi i z_0}$ follows.
\end{proof}

For example, Table \ref{a1n1m=4} displays how well \eqref{pres} in Theorem \ref{maina} approximates $\mathcal A_1(N,\sigma)$ for $\sigma=1$ and some values of $m$ and $N$.
\begin{table}[h]
\begin{center}
\begin{tabular}{c|cccc|c}
$N$ & $m=1$ & $m=2$ & $m=3$ & $m=4$ & $\mathcal A_1(N,1)$  \\ \hline
$200$ & $-33.8689$ & $-32.5734$ & $-32.4829$ & $-32.4681$ & $-32.4692$ \\
$400$ & $2.17937  \times 10^{7}$ & $2.16780 \times 10^{7}$ & $2.16710 \times 10^{7}$ & $2.16712 \times 10^{7}$ & $2.16712 \times 10^{7}$ \\

$600$ & $1.80284 \times 10^{12}$ & $1.77324 \times 10^{12}$ & $1.77260 \times 10^{12}$ &  $1.77255 \times 10^{12}$ &  $1.77255 \times 10^{12}$ \\
$800$ & $-3.72536 \times 10^{18}$ &  $-3.71475 \times 10^{18}$ & $-3.71444 \times 10^{18}$ &  $-3.71444 \times 10^{18}$ &  $-3.71444 \times 10^{18}$ \\
$1000$ & $-2.58000 \times 10^{23}$ &  $-2.54119 \times 10^{23}$ & $-2.54072 \times 10^{23}$ &  $-2.54070 \times 10^{23}$ &  $-2.54070 \times 10^{23}$
\end{tabular}
\caption{Theorem \ref{maina}'s approximations to $\mathcal A_1(N,1)$.} \label{a1n1m=4}
\end{center}
\end{table}

The expansion coefficients $b_t(\sigma)$ may all be written explicitly in terms of $w_0$ and $z_0$. We give $b_1(\sigma)$ next.
\begin{prop} \label{propb1s}  Each $b_t(\sigma)$ is a polynomial in $\sigma$ of degree $t$. For $t=1$ we have
\begin{equation} \label{b1}
b_1(\sigma)= \frac{4\pi i z_0^2 }{ e^{\pi i z_0}} \sigma - \frac{ w_0}{ \pi i e^{3\pi i z_0}}
\left( \frac{(2\pi i z_0)^2}{12}-2\pi i z_0 +1  \right).
\end{equation}
\end{prop}
\begin{proof} Note that $u_{\sigma,j}(z)$ is a polynomial in $\sigma$ of degree $j$. Since $a_{2s}(q)$ is linear in its argument, $a_{2s}\bigl(c_1 q_1(z)+c_2 q_2(z)\bigr) = c_1 a_{2s}\bigl( q_1(z)\bigr) + c_2 a_{2s}\bigl( q_2(z)\bigr)$, it follows from \eqref{btys} that $b_t(\sigma)$ is a polynomial in $\sigma$ of degree $t$.

With $t=1$, \eqref{btys} implies
\begin{align*}
    b_1(\sigma) & =  -4i  \G(1/2) a_{0}(q \cdot u_{\sigma,1}) -4 i \G(3/2) a_{2}(q \cdot u_{\sigma,0})\\
     & =  -4i \sqrt{\pi} \cdot a_{0}(q) \cdot u_{\sigma,1}(z_0) -2i \sqrt{\pi} a_{2}(q ).
\end{align*}
Since
\begin{equation*}
    u_{\sigma,1}(z) = 2\pi i \sigma z + g_1(z) = \frac{\pi i z}{6}\left(12\sigma -\frac 12 + \frac 1{1-e^{2\pi i z}}\right)
\end{equation*}
we see that $u_{\sigma,1}(z_0) = \pi i z_0(12 \sigma -1/2+1/w_0)/6$. Also \eqref{a2sb} implies
\begin{equation*}
    a_2(q)=\frac{a_0(q)}{p_0}\left(\frac{q_2}{q_0}-\frac{3}{2} \frac{p_1}{p_0} \frac{q_1}{q_0} -\frac{3}{2} \frac{p_2}{p_0} +\frac{15}{8} \frac{p_1^2}{p_0^2}\right).
\end{equation*}
Taking derivatives of $q^2(z)=iz/(1-e^{2\pi i z})$ and evaluating at $z=z_0$ shows that
\begin{equation*}
    \frac{q_1}{q_0}  = -\pi i +\frac{1}{2z_0}+\frac{\pi i}{w_0}, \qquad
    \frac{q_2}{q_0}  = -\frac{\pi^2}{2}  -\frac{\pi i}{2z_0}+\frac{2\pi^2}{w_0} -\frac{1}{8z_0^2}+\frac{ \pi i}{2z_0 w_0} -\frac{3\pi^2}{2w_0^2}.
\end{equation*}
Similarly, recalling $p_0$ from \eqref{oad2} and using \eqref{pzzz}, \eqref{pzzz2} and their generalizations, we have
\begin{equation*}
    \frac{p_1}{p_0}  = -\frac{1}{z_0}+\frac{2\pi i}{3w_0}, \qquad
    \frac{p_2}{p_0}  = \frac{\pi^2}{3w_0}  +\frac{1}{z_0^2}-\frac{2 \pi i}{3z_0 w_0} -\frac{2\pi^2}{3w_0^2}.
\end{equation*}
Putting this all together with $a_0(q)$ from \eqref{a0q} and simplifying completes the proof.
\end{proof}

Assuming Theorem \ref{mainb} -- see the summary of its proof in the next section -- we may now prove our main result.

\begin{proof}[\bf Proof of Theorem \ref{mainthmb}]
Putting $h/k=0/1$ in \eqref{exprc} gives
\begin{equation*}
    C_{01\ell }(N) =   \sum_{\sigma=1}^\ell \binom{\ell -1}{\sigma-1} (-1)^{\ell-\sigma} Q_{01\sigma}(N).
\end{equation*}
Taking the same linear combination of \eqref{ch} produces
\begin{equation} \label{olym}
    \sum_{\sigma=1}^\ell \binom{\ell -1}{\sigma-1} (-1)^{\ell-\sigma} \sum_{h/k \in \farey_N} Q_{hk\sigma}(N) =  0  \qquad \text{for} \qquad N(N+1)/2 > \ell
\end{equation}
and we partition $\farey_N$ into three parts: $\farey_{100}$, $\mathcal A(N)$ and the rest. The sum over this third part is $O(e^{WN})$ by Theorem \ref{mainb} implying that \eqref{olym} breaks into
\begin{equation}\label{ssffg}
    C_{01\ell }(N) + \sum_{0<h/k \in \farey_{100}}\sum_{\sigma=1}^\ell \binom{\ell -1}{\sigma-1} (-1)^{\ell-\sigma}  Q_{hk\sigma}(N)
    + \sum_{\sigma=1}^\ell \binom{\ell -1}{\sigma-1} (-1)^{\ell-\sigma} \mathcal A_1(N,\sigma) = O(e^{WN}).
\end{equation}
Use \eqref{exprcinv} to see that
\begin{equation} \label{hark}
    \sum_{\sigma=1}^\ell \binom{\ell -1}{\sigma-1} (-1)^{\ell-\sigma}  Q_{hk\sigma}(N) = \sum_{j=1}^\ell (e^{2\pi i h/k} -1)^{\ell-j} C_{hkj}(N).
\end{equation}
Then  \eqref{hark} and  Theorem \ref{maina} let us write \eqref{ssffg} as
\begin{multline} \label{maineq2b}
   C_{01\ell}(N)+ \sum_{0<h/k \in \farey_{100}} \sum_{j=1}^\ell (e^{2\pi i h/k} -1)^{\ell-j} C_{hkj}(N) \\
   = \Re\left[\frac{w_0^{-N}}{N^{2}} \left( b^*_{\ell,0}+\frac{b^*_{\ell,1}}{N}+ \dots +\frac{b^*_{\ell,m-1}}{N^{m-1}}\right)\right] + O\left(\frac{|w_0|^{-N}}{N^{m+2}}\right)
\end{multline}
for
\begin{equation*}
    b^*_{\ell,t}:=-\sum_{\sigma=1}^\ell \binom{\ell -1}{\sigma-1} (-1)^{\ell-\sigma} b_{t}(\sigma).
\end{equation*}

We claim that $b^*_{\ell,t}=0$ for $0\lqs t\lqs \ell-2$ and  $b^*_{\ell,\ell-1}=-2  z_0 e^{-\pi i z_0} (2\pi i z_0)^{\ell -1}$.
To see this,  observe that $\mathcal A_1(N,\sigma)$ may be replaced in \eqref{ssffg} by $\mathcal A_3(N,\sigma)$, as defined in \eqref{a3(n)}, since \eqref{lincoln} is true. The dependence of $\mathcal A_3(N,\sigma)$ on $\sigma$ comes from the $\exp\bigl(v(z;N,\sigma)\bigr)$ term and we have
\begin{equation*}
\begin{split}
    \sum_{\sigma=1}^\ell \binom{\ell -1}{\sigma-1} (-1)^{\ell-\sigma} \exp&\bigl(v(z;N,\sigma)\bigr)
      = \sum_{\sigma=1}^\ell \binom{\ell -1}{\sigma-1} (-1)^{\ell-\sigma} \exp(2\pi i z/N)^\sigma \exp\bigl(v(z;N,0)\bigr)\\
     & = \bigl[ \exp(2\pi i z/N) -1\bigr]^{\ell-1} \exp\bigl(v(z;N,1)\bigr)\\
     & = \left( \frac{2\pi i z}{N}+  \frac{(2\pi i z)^2}{2!N^2}+ \dots\right)^{\ell-1} \left( u_{1,0}(z)+\frac{u_{1,1}(z)}{N}+\dots \right)\\
     & = \frac{1}{N^{\ell-1}}\left( v^*_{\ell,0}(z) + \frac{v^*_{\ell,1}(z)}{N}+ \frac{v^*_{\ell,2}(z)}{N^2}+\dots \right)
     \end{split}
\end{equation*}
where $v^*_{\ell,0}(z)=(2\pi i z)^{\ell-1}$, and in general, employing \eqref{pobell}, \eqref{pobell2},
\begin{equation*}
    v^*_{\ell,j}(z)=\sum_{t=0}^j \hat{B}_{\ell-1+t,\ell-1}(1/1!, 1/2!, 1/3!, \dots)\cdot (2\pi i z)^{\ell-1+t}\cdot u_{1,j-t}(z).
\end{equation*}
Now repeating the proof of Theorem \ref{maina} with $u_{\sigma,j}$ replaced by $v^*_{\ell,j}$ yields
\begin{equation}\label{clt}
    c_{\ell,t}= 4i \sum_{s=0}^t \G(s+1/2) a_{2s}(q \cdot v^*_{\ell,t-s})
\end{equation}
in the statement of the theorem, with $c_{\ell,0}=-2  z_0 e^{-\pi i z_0} (2\pi i z_0)^{\ell -1}=b^*_{\ell,\ell-1}$ as desired.
\end{proof}

With \eqref{clt} we may compute the coefficients $c_{\ell,t}$ explicitly. For example, similar calculations to those of Proposition \ref{propb1s} produce
$v^*_{\ell,1}(z)=(2\pi i z)^{\ell}(6\ell+11/2+1/(1-e^{2\pi i z}))/12$ and
\begin{equation}\label{clff}
    c_{\ell,1}=-\frac{(\ell+1)z_0(2\pi i z_0)^\ell}{ e^{\pi i z_0}} + \frac{ z_0 w_0 (2\pi i z_0)^\ell}{ e^{3\pi i z_0}}
\left( \frac 16-\frac{\ell+1}{2\pi i z_0}+\frac{\ell(\ell+1)}{(2\pi i z_0)^2}  \right).
\end{equation}

\section{Further results} \label{sec-fur}

\subsection{Proof of Theorem \ref{mainb}} \label{sec-fur1}
To prove Theorem \ref{mainb}
we first  need a general estimate for the sine product $\spr{h/k}{m}$, without the restriction $0\lqs m < k/h$ that was in place in Section \ref{sec-est}.
Define the set
\begin{equation*}
Z(h,k):=\bigl\{ (\beta,\gamma)\in \Z\times \Z \ : \ 1\lqs |\beta| < k, \ 1\lqs \gamma < k,  \ \beta h \equiv \gamma \bmod k  \bigr\}.
\end{equation*}

\begin{theorem} \label{mainest}
For all $m$, $h$, $k\in \Z$ with $1 \lqs  h < k$, $(h,k)=1$ and $0 \lqs m < k$ we have
\begin{equation}\label{logpx}
\frac{1}{k}  \log \left| \spr{h/k}{m} \right|  =  \frac{\cl(2\pi m \gamma_0 h/k)}{2\pi |\beta_0 \gamma_0|}  + O\left(\frac{\log k}{\sqrt{k}}\right)
\end{equation}
where $(\beta_0,\gamma_0)$ is a pair in $Z(h,k)$ with $|\beta_0 \gamma_0|$ minimal. The   implied constant in \eqref{logpx} is absolute.
\end{theorem}

Very briefly, the proof of Theorem \ref{mainest} involves showing, with another application of Euler-Maclaurin summation, that
\begin{equation*}
    \frac{1}{k}\log \left| \spr{h/k}{m} \right| = \frac{S(m;h,k)}{2\pi}  +O\left(\frac{\log^2 k}{k}\right)
\end{equation*}
for
\begin{equation*}
    S(m;h,k):=\sum_{(\beta,\gamma)\in Z(h,k)} \frac{\sin(2\pi m \gamma/k)}{|\beta \gamma|}
\end{equation*}
and then relating $S(m;h,k)$ to Clausen's integral.

Define $D(h,k)$ to be the above minimal value $|\beta_0 \gamma_0|$ in the statement of Theorem \ref{mainest}. For example, it is easy to see that
\begin{equation*} 
     D(h,k)  = 1 \iff  h \equiv  \pm 1 \bmod k
\end{equation*}
and if $D(h,k) \neq 1$ then
\begin{equation*} 
     D(h,k)  = 2 \iff   h \text{ or }  h^{-1} \equiv  \pm 2 \bmod k
\end{equation*}
with $k$ necessarily odd.
A simple corollary to Theorem \ref{mainest}   says there exists an absolute constant $\tau$ such that
\begin{equation}\label{logpxx}
\frac{1}{k} \Bigl| \log \bigl| \spn{h/k}{m} \bigr| \Bigr| \lqs  \frac{\cl(\pi/3)}{2\pi D(h,k)}  + \tau \frac{\log k}{\sqrt{k}}.
\end{equation}

The second result we need is a general bound for $Q_{hk\sigma}(N)$:
\begin{prop} \label{prc}
For $2 \lqs k \lqs N$, $\sigma \in \R$ and $s := \lfloor N/k \rfloor$
\begin{equation*}
    |Q_{hk\sigma}(N)|  \lqs \frac{3}{k^3} \exp\left( N \frac{2  +  \log \left(1+ 3k/4 \right)}{k} +\frac{|\sigma|}{N} \right)\left|\spr{h/k}{N-s k} \right|.
\end{equation*}
\end{prop}
This proposition is proved by expressing $Q_{hk\sigma}(N)$ as the integral of $Q(z;N,\sigma)$ around a small loop circling $h/k$ (recall \eqref{defqn}) and bounding the absolute value of $Q(z;N,\sigma)$ on this loop. A refinement of Proposition \ref{prc}, restricting the values of $k$ to $101 \lqs k \lqs N$, has
\begin{equation} \label{i7}
    |Q_{hk\sigma}(N)|  \lqs \frac{9}{k^3} \exp\left( N \frac{2  +  \log \left(\xi/2+ \xi' k/8 \right)}{k} +\frac{|\sigma|}{N} \right)\left|\spr{h/k}{N-s k} \right|
\end{equation}
for $\xi = 1.00038$ and $\xi'=1.01041$. Combining \eqref{logpxx} with \eqref{i7} gives
\begin{equation} \label{ukr}
    Q_{hk\sigma}(N)  \ll \frac{1}{k^3} \exp\left( N \frac{2  +  \log \left(\xi/2+ \xi' k/8 \right)}{k}  + \frac{\cl(\pi/3)}{2\pi D(h,k)}   \cdot k + \tau \sqrt{N} \log N \right)
  \end{equation}
for $k \gqs 101$. A straightforward calculation with \eqref{ukr} then shows that,  when $W > \cl(\pi/3)/(6\pi) \approx 0.0538$, we have
\begin{equation} \label{goa}
    Q_{hk\sigma}(N)  \ll   e^{WN}/k^3
\end{equation}
for all $h/k \in \farey_{N}-\farey_{100}$ except in the  cases where
\begin{align*}
    h \equiv \pm 1, \pm 2, (k\pm 1)/2 \bmod k & \qquad \text{and} \qquad N/2  <k \lqs N, \\
   \text{or} \qquad h \equiv \pm 1 \bmod k & \qquad \text{and} \qquad N/3  <k \lqs N/2,
\end{align*}
corresponding to $k$  large and $D(h,k)$  small. Hence, the three subsets of $\farey_N - (\farey_{100}  \cup \mathcal A(N))$
we must consider separately  are
\begin{align*}
    \mathcal C(N) & := \Bigl\{ h/k \ : \ \frac{N}{2}  <k \lqs N,  \ k \text{ odd}, \ h=2 \text{ \ or \ } h=k-2 \Bigr\}, \\
    \mathcal D(N) & := \Bigl\{ h/k \ : \ \frac{N}{2}  <k \lqs N,  \ k \text{ odd}, \ h=\frac{k-1}2 \text{ \ or \ } h=\frac{k+1}2 \Bigr\}, \\
    \mathcal E(N) & := \Bigl\{ h/k \ : \ \frac{N}{3}  <k \lqs \frac{N}{2},  \ h=1 \text{ \ or \ } h=k-1 \Bigr\}
\end{align*}
with $\mathcal C(N)$, $\mathcal D(N)$ sets of simple poles of $Q(z;N,\sigma)$ and $\mathcal E(N)$ a set of double poles.

To describe the asymptotics of the corresponding sums of $Q_{hk\sigma}(N)$s, recall the dilogarithm zero $w_0=w(0,-1)$ and its associated saddle-point $z_0$ given by \eqref{w0x}. We also need the new   saddle-points
\begin{equation*}
    z_3:=3+\log \bigl(1-w(1,-3)\bigr)/(2\pi i), \qquad z_1:=2+\log \bigl(1-w(0,-2)\bigr)/(2\pi i)
\end{equation*}
using the notation of Section \ref{dilogg}. The proof of Theorem \ref{maina}, giving the asymptotic expansion of $\mathcal A_1(N,\sigma)$, extends to cover these three new cases and, for implied constants  depending only on $\sigma$ and $m$, we obtain
\begin{align}
    \sum_{h/k \in \mathcal C(N)} Q_{hk\sigma}(N) & = \Re\left[\frac{w(1,-3)^{-N}}{N^{2}} \left( c_{0}^*+\frac{c_{1}^*(\sigma)}{N}+ \dots +\frac{c_{m-1}^*(\sigma)}{N^{m-1}}\right)\right] + O\left(\frac{|w(1,-3)|^{-N}}{N^{m+2}}\right), \label{casy}\\
    \sum_{h/k \in \mathcal D(N)} Q_{hk\sigma}(N) & = \Re\left[\frac{w_0^{-N/2}}{N^{2}} \left( d_{0}\bigl(\overline{N}\bigr) +\frac{d_{1}\bigl(\sigma, \overline{N}\bigr)}{N}+ \dots +\frac{d_{m-1}\bigl(\sigma, \overline{N}\bigr)}{N^{m-1}}\right)\right] + O\left(\frac{|w_0|^{-N/2}}{N^{m+2}}\right),\label{dasy}\\
    \sum_{h/k \in \mathcal E(N)} Q_{hk\sigma}(N) & = \Re\left[\frac{w(0,-2)^{-N}}{N^{2}} \left( e_{0}+\frac{e_{1}(\sigma)}{N}+ \dots +\frac{e_{m-1}(\sigma)}{N^{m-1}}\right)\right] + O\left(\frac{|w(0,-2)|^{-N}}{N^{m+2}}\right)\label{easy}
\end{align}
where $\overline{N}$ denotes $N \bmod 2$ and the  expansion coefficients may be given explicitly; the first ones are
\begin{equation*}
    c_{0}^*=-z_3 e^{-\pi i z_3}/4,  \quad d_{0}\bigl(\overline{N}\bigr) = z_0 \sqrt{2  e^{-\pi i z_0}\bigl(e^{-\pi i z_0}+(-1)^N \bigr)}, \quad e_{0}=-3z_1 e^{-\pi i z_1}/2.
\end{equation*}
 The right sides of \eqref{casy}, \eqref{dasy} and \eqref{easy} are
\begin{equation*}
    O(e^{0.0357 N}), \quad O(e^{0.0341 N}), \quad O(e^{0.0257 N})
\end{equation*}
respectively and so, with \eqref{goa}, we have
 completed the summary of the proof of Theorem \ref{mainb}. See \cite{OS2} for further details.

\begin{remark} \label{wha}{\rm
The reason that fractions $h/k$ in $\farey_{100}$ are excluded from Theorem \ref{mainb} is that the bounds we have for the corresponding $Q_{hk\sigma}(N)$s, due to Proposition \ref{prc} and its refinements, are larger than the error we want $O(e^{WN})$. By improving Proposition \ref{prc} it should be possible to reduce  $\farey_{100}$ to just  $\farey_{1}$. See also \cite[Remark 3.6]{OS2}.
}
\end{remark}

\subsection{Generalizations and conjectures} \label{sec-gen}
In Table \ref{c011} we give numerical evidence for Conjecture \ref{coj} in the case $\ell=1$ by comparing both sides of \eqref{cj5} for different values of $m$.
\begin{table}[h]
\begin{center}
\begin{tabular}{c|cccc|c}
$N$ & $m=1$ & $m=2$ & $m=3$ & $m=4$ & $\mathcal C_{011}(N)$  \\ \hline
$400$ & $-2.17937  \times 10^{7}$ & $-2.16780 \times 10^{7}$ & $-2.16710 \times 10^{7}$ & $-2.16712 \times 10^{7}$ & $-2.16712 \times 10^{7}$ \\
$600$ & $-1.80284 \times 10^{12}$ & $-1.77324 \times 10^{12}$ & $-1.77260 \times 10^{12}$ &  $-1.77255 \times 10^{12}$ &  $-1.77255 \times 10^{12}$ \\
$800$ & $3.72536 \times 10^{18}$ &  $3.71475 \times 10^{18}$ & $3.71444 \times 10^{18}$ &  $3.71444 \times 10^{18}$ &  $3.71444 \times 10^{18}$ \\
$1000$ & $2.58000 \times 10^{23}$ &  $2.54119 \times 10^{23}$ & $2.54072 \times 10^{23}$ &  $2.54070 \times 10^{23}$ &  $2.54070 \times 10^{23}$
\end{tabular}
\caption{Conjecture \ref{coj}'s approximations to $C_{011}(N)$.} \label{c011}
\end{center}
\end{table}
These entries match those of Table \ref{a1n1m=4} since $c_{1,t}=-b_t(1)$. Compare also \cite[Table 1]{OS}.
Table \ref{c014} shows the case $\ell=4$ of Conjecture \ref{coj}.
\begin{table}[h]
\begin{center}
\begin{tabular}{c|cccc|c}
$N$ & $m=1$ & $m=2$ & $m=3$ & $m=4$ & $\mathcal C_{014}(N)$  \\ \hline
$400$ & $-56.2851$ & $-58.7844$ & $-58.6802$ & $-58.6857$ & $-58.6545$ \\
$600$ & $-1.52353 \times 10^{7}$ & $-1.52212 \times 10^{7}$ & $-1.52136 \times 10^{7}$ &  $-1.52132 \times 10^{7}$ &  $-1.52133 \times 10^{7}$\\
$800$ & $1.44649 \times 10^{12}$ & $1.47247 \times 10^{12}$ & $1.47185 \times 10^{12}$ &  $1.47186 \times 10^{12}$ &  $1.47186 \times 10^{12}$
\end{tabular}
\caption{Conjecture \ref{coj}'s approximations to $C_{014}(N)$.} \label{c014}
\end{center}
\end{table}

The identity \eqref{ch} for $\sigma=1$ says
\begin{equation} \label{dow}
\sum_{h/k \in \farey_N} C_{hk1}(N) =  0 \qquad (N \gqs 2)
\end{equation}
since $C_{hk1}(N) = Q_{hk1}(N)$. With the proof of Theorem \ref{mainthma}, we have shown that the largest terms in \eqref{dow} have $h/k$ in $\farey_{100}$ and $\mathcal A(N)$. The $\ell=1$ case of Conjecture \ref{coj} indicates that all the terms with $h/k$ in $\farey_{100}$ are relatively small except for $h/k=0/1$. So we expect
\begin{equation}\label{sky1}
    C_{011}(N) \sim -\sum_{\substack{N/2<b\lqs N \\ a \equiv \pm 1 \bmod b}} C_{ab1}(N),
\end{equation}
i.e. that the asymptotic expansions of both sides of \eqref{sky1} are the same.

Can we match up the other terms of \eqref{dow} in the same way? A clue to the asymptotics of $C_{121}(N)$ comes from noticing how closely it matches $-1$ times \eqref{dasy}. This lets us expect
\begin{equation}\label{sky2}
    C_{121}(N) \sim -\sum_{\substack{N/2<b\lqs N, \ (b,2)=1 \\ a \equiv \pm 2^{-1} \bmod b}} C_{ab1}(N).
\end{equation}

\begin{conj} \label{coj121}
For the coefficients $d_{0}\bigl(\overline{N}\bigr)$, $d_{1}\bigl(\sigma, \overline{N}\bigr), \dots$ of \eqref{dasy} and an   implied constant depending only on the positive integer $m$, we have
\begin{equation*}
    C_{121}(N)
   = -\Re\left[\frac{w_0^{-N/2}}{N^{2}} \left( d_{0}\bigl(\overline{N}\bigr) +\frac{d_{1}\bigl(1, \overline{N}\bigr)}{N}+ \dots +\frac{d_{m-1}\bigl(1, \overline{N}\bigr)}{N^{m-1}}\right)\right] + O\left(\frac{|w_0|^{-N/2}}{N^{m+2}}\right).
\end{equation*}
\end{conj}

Some numerical evidence for Conjecture \ref{coj121} is given in Table \ref{c121}.
The $m=1$ case of Conjecture \ref{coj121} appeared already in \cite[Conj. 6.3]{OS}.
\begin{table}[h]
\begin{center}
\begin{tabular}{c|cccc|c}
$N$ & $m=1$ & $m=2$ & $m=3$ & $m=4$ & $C_{121}(N)$  \\ \hline
$1000$ & $1.76776 \times 10^{9}$ &  $1.77847 \times 10^{9}$ & $1.7778 \times 10^{9}$ &  $1.77778 \times 10^{9}$ &  $1.77778 \times 10^{9}$ \\
$1001$ & $2.10996 \times 10^{9}$ &  $2.11483 \times 10^{9}$ & $2.1142 \times 10^{9}$ &  $2.11418 \times 10^{9}$ &  $2.11418 \times 10^{9}$
\end{tabular}
\caption{Conjecture \ref{coj121}'s approximations to $C_{121}(N)$.} \label{c121}
\end{center}
\end{table}

Continuing the pattern from \eqref{sky1}, \eqref{sky2} we guess that
\begin{equation}\label{sky4}
    C_{131}(N) + C_{231}(N) \sim -\sum_{\substack{N/2<b\lqs N, \ (b,3)=1 \\ a \equiv \pm 3^{-1} \bmod b}} C_{ab1}(N)
\end{equation}
and indeed numerical evidence seems to support this. With more work, the techniques we have developed to prove Theorem \ref{maina} and \eqref{dasy} should give the asymptotic expansion of the right side of \eqref{sky4}.

We finally list some further interesting directions for investigation: 
\begin{enumerate}
\item [(i)] Rademacher's original conjecture \eqref{rademach} has been disproved, but what is the correct conjecture? As discussed in
\cite[Sect. 6]{OS}, for each triple $hk\ell$ there seems to be some initial agreement between $C_{hk\ell}(\infty)$ and the coefficients $C_{hk\ell}(N)$  for small $N$.

\item  [(ii)] Each Rademacher coefficient $C_{hk\ell}(N)$ is a linear combination of the numbers $Q_{hk\sigma}(N)$ for $1\lqs \sigma\lqs \ell$, as we have seen. For $\sigma$ negative, on the other hand, combinations of  $Q_{hk\sigma}(N)$ produce the restricted partition function: with $\sigma = -n$ in Theorem \ref{qn} we obtain Sylvester's result
\begin{equation} \label{wave}
   p_N(n)=\sum_{k=1}^N \Bigl[\sum_{0 \lqs h <k, \ (h,k)=1} -Q_{hk(-n)}(N)\Bigr].
\end{equation}
The inner sum in brackets is Sylvester's $k$th wave \cite{Sy3}, which  may be denoted as $W_k(N,n)$. See also \cite[Sect. 4]{OS}, for example. The techniques we have developed in this paper should allow  quantification of how well (or poorly) the first waves $W_1(N,n)$, $W_2(N,n), \cdots$ approximate $p_N(n)$  as $N$ and possibly $n$ tend to infinity. This ties in to work of Szekeres in \cite{Sze51} who found asymptotic formulas for $p_N(n)$ when $n \gqs 0.135 N^2$ by using the first wave, $W_1(N,n)$,  in the decomposition \eqref{wave}. He extended this result   in \cite{Sze53}, removing the restriction $n \gqs 0.135 N^2$, by using a different approach that incorporated the saddle-point method.

\item [(iii)]  If we replace the product $\prod_{j=1}^N 1/(1-q^j)$ in \eqref{tp} with a different product and examine the coefficients in its partial fraction decomposition as the number of factors goes to infinity, can we obtain  results similar to Theorem \ref{mainthmb}? For example, following Sylvester's general theory in \cite{Sy3}, we may replace the sequence $1,2,3, \cdots$ with any sequence $a_1, a_2, a_3, \cdots$ of possibly repeating positive integers and study the partial fractions of
    \begin{equation}\label{sylx}
        \prod_{j=1}^N \frac{1}{1-q^{a_j}}
    \end{equation}
    as $N \to \infty$. The coefficient of $q^n$ in \eqref{sylx} is now the number of solutions in nonnegative integers $x_i$ to $a_1 x_1+ \cdots +a_N x_N = n$, expressed by Sylvester as a sum of waves.
\end{enumerate}

{\small
\bibliography{raddata}
}

\textsc{Dept. of Mathematics, The CUNY Graduate Center , New York, NY 10016-4309, U.S.A.}

{\em E-mail address:} \texttt{cosullivan@gc.cuny.edu}

{\em Web page:} \texttt{http://fsw01.bcc.cuny.edu/cormac.osullivan}

\end{document}